\documentclass{calabi}
\usepackage{euscript}
\usepackage{parskip}
\usepackage{leftindex}
\usepackage{tikz-cd}

\newcommand{\scrC}{\EuScript{C}}

\newcommand{\scrF}{\EuScript{F}}
\newcommand{\scrH}{\EuScript{H}}
\renewcommand{\cA}{\EuScript{A}}
\newcommand{\scrB}{\EuScript{B}}

\newcommand{\scrT}{\EuScript{T}}
\newcommand{\scrU}{\EuScript{U}}
\newcommand{\rk}{\mathrm{rk}}
\newcommand{\vd}{\operatorname{vd}}
\newcommand{\cdbar}{\overline{\partial}}

\newcommand{\bQ}{\mathbb{Q}}
\newcommand{\bN}{\mathbb{N}}
\newcommand{\bP}{\mathbb{P}}
\newcommand{\bC}{\mathbb{C}}
\newcommand{\bR}{\mathbb{R}}
\newcommand{\bZ}{\mathbb{Z}}
\newcommand{\bA}{\mathbb{A}}

\newcommand{\bT}{\mathbb{T}}
\newcommand{\bK}{\mathbb{K}}

\newcommand{\ev}{\mathrm{ev}}
\newcommand{\tw}{\mathrm{tw}}

\newcommand{\cB}{\mathcal{B}}

\newcommand{\cG}{\mathcal{G}}
\newcommand{\cH}{\mathcal{H}}
\newcommand{\cV}{\mathcal{V}}
\newcommand{\cQ}{\mathcal{Q}}
\newcommand{\cY}{\mathcal{Y}}
\newcommand{\cZ}{\mathcal{Z}}
\newcommand{\cI}{\mathcal{I}}

\newcommand{\cK}{\mathcal{K}}
\newcommand{\cC}{\mathcal{C}}

\newcommand{\bk}{\mathbb{K}}

\newcommand{\PSS}{\mathrm{PSS}}
\newcommand{\SSP}{\mathrm{SSP}}

\newcommand{\orb}{\mathrm{orb}}
\newcommand{\val}{\mathrm{val}}
\newcommand{\depth}{\mathrm{depth}}

\newcommand{\reg}{\mathrm{reg}}

\newcommand{\cP}{\mathcal{P}}
\newcommand{\cF}{\mathcal{F}}
\newcommand{\cR}{\mathcal{R}}

\newcommand{\cW}{\mathcal{W}}

\newcommand{\bfh}{\mathbf{h}}
\newcommand{\bfr}{\mathbf{r}}

\newcommand{\bfm}{\mathbf{m}}
\newcommand{\bfa}{\mathbf{a}}

\begin{document}

\title{Orbifold Hamiltonian Floer theory for global quotients}
\date{\today}
\author{Cheuk Yu Mak, Sobhan Seyfaddini, and Ivan Smith}
\maketitle

\begin{abstract} 
\noindent We construct bulk-deformed orbifold Hamiltonian Floer theory for a global quotient orbifold $Y = [X / \Gamma]$, where $\Gamma$ is a finite group acting faithfully by symplectomorphisms on a smooth closed symplectic manifold $X$.  The moduli spaces define an `ordered marked Flow category', which we equip with a coherent presentation via derived orbifolds. The relevant global charts for moduli spaces of orbifold Floer cylinders are built from spaces of domains which are curves in $\bP(V)/\Gamma$ for a faithful $\Gamma$-representation $V$.
\end{abstract}

\tableofcontents

\section{Introduction}

Let $(X,\omega_X)$ be a closed symplectic manifold, $[\omega_X] \in H^2(X,\mathbb{Q})$ and $\Gamma$ a finite group acting faithfully by symplectomorphisms of $X$. Let $Y = [X/\Gamma]$ denote\footnote{
$Y$ and $[X/\Gamma]$ denote the orbifold quotient; we write $X/\Gamma$ for its underlying coarse space.} the quotient symplectic orbifold (i.e. differentiable Deligne-Mumford stack) and let $\frak{b}$ be a bulk-deformation class, i.e. a class in the cohomology of the inertia orbifold (see Theorem \ref{t:main} for a precise description).  In this paper, we construct bulk-deformed orbifold Hamiltonian Floer cohomology $HF^*(H_Y;\frak{b})$ for a non-degenerate smooth Hamiltonian function $H_Y: Y \times S^1 \to \bR$, and the associated spectral invariants 
\[
c^{\frak{b}}(x,-): C^{\infty}(Y \times S^1) \to \bR
\]
associated to classes $x$ in the $\mathfrak{b}$-deformed orbifold (Chen-Ruan) quantum cohomology  $QH^*_{orb}(Y;\mathfrak{b})$.  

The sequel \cite{MSS2} to this paper applies these spectral invariants, in the special case of symmetric product orbifolds, to questions in Hamiltonian dynamics related to density of periodic orbits, and the `smooth closing lemma'. In particular, we give there sufficient conditions for a sequence of spectral invariants on symmetric product orbifolds to satisfy a `Weyl law'.
\medskip

Fix a field  $\bk \supset \bQ$ and 
let $\Pi$ be a discrete subgroup of $\mathbb{R}$ containing $\frac{1}{|\Gamma|}\omega_X(H_2(X;\mathbb{Z}))$.
Consider the Novikov field
\[\Lambda_{\Pi}:=\left\{\sum_{i=1}^{\infty} a_iT^{b_i}| a_i \in \bk,  b_i \in \Pi,  \lim_i b_i=\infty\right\}.\]
with valuation 
\[
\val:\Lambda_{\Pi} \to \mathbb{R} \cup \{\infty\}, \quad \val\left(\sum a_iT^{b_i} \right)=\min\{b_i\}. 
\]
Let $\Lambda_{>0}:=\val^{-1}(\mathbb{R}_{>0} \cup \{\infty\})$ 
and $\Lambda_{\ge 0}:=\val^{-1}(\mathbb{R}_{ \ge 0} \cup \{\infty\})$
which consist of elements of $\Lambda_{\Pi}$ with positive and non-negative valuation respectively.

Recall that an orbifold $Y$ has an inertia stack $IY$, which splits into a union of components called `twisted sectors' (\cite{OrbifoldBook}).  For a global quotient $Y = [X/\Gamma]$, the twisted sectors are indexed by the conjugacy classes of $\Gamma$, with the conjugacy class $(g)$ of $g$ giving the sector $X^g / C(g)$ with $X^g \subset X$ the fixed point locus of $g$ and $C(g) \subset \Gamma$ the centraliser of $g$.\footnote{Different representatives of $(g)$ give isomorphic sectors.}
Let $\|\Gamma\|$ be the set of conjugacy classes of $\Gamma$ and $\|\Gamma\|^{\circ}=\|\Gamma\| \setminus \{id\}$.
The orbifold (Chen-Ruan) cohomology is, up to a grading shift,
\[
H_{orb}(Y;\bK)=\oplus_{(g) \in \|\Gamma\|} H(X^g / C(g);\bK)
\]
(see Section \ref{s:orbicoh} and Equation \ref{eqn:orbifold_cohomology_decomposition} for the discussion of the grading shift).

\begin{theo}\label{t:main}
Let $Y = [X / \Gamma]$ be a global quotient orbifold, and 
\begin{align}\label{eq:bulkb}
\frak{b}=\sum_{(g) \in \|\Gamma\|^{\circ}} \frak{v}_{(g)} PD[X^g/C(g)] \in H_{orb}(Y;\Lambda_{>0}),  \quad \frak{v}_{(g)} \in \Lambda_{>0}
\end{align}
where $PD$ stands for the Poincar\'e dual.

Then for each $H_Y \in  C^{\infty}(Y \times S^1;\bR)$ that is non-degenerate (in the sense that $1$ is not an eigenvalue of the linearised return map for any $1$-periodic orbit of $H_Y$) and each $\omega_Y$-compatible almost complex structure $J_Y$,

\begin{enumerate}
    \item[(A)] there is a $\mathbb{Z}/2\mathbb{Z}$-graded bulk-deformed Floer cohomology group $HF(H_Y;\frak{b})$ linear over $\Lambda$, independent of $J_Y$ up to canonical isomorphism, independent of $H_Y$ up to additive isomorphism, and intrinsic to $(Y,H_Y)$ (i.e. not depending on the presentation of $Y$ as a global quotient);
    
    \item[(B)] there is an orbifold pair of pants product
    \[
HF(H_Y;\frak{b}) \otimes HF(H'_Y;\frak{b}) \longrightarrow HF(H_Y\# H'_Y;\frak{b})
\]
    which depends on the bulk $\frak{b}$, and an isomorphism 
\begin{align}\label{eq:PSS}
QH^*_{orb}(Y;\frak{b}) \longrightarrow HF(H_Y;\frak{b})
\end{align}
which is compatible with the product structure (i.e. entwines the bulk-deformed orbifold quantum product with the pair of pants product).
\end{enumerate}

\end{theo}

In the statement (B) above, $H_Y \# H'_Y$ is referred to as the {\it composition} of the Hamiltonians $H_Y, H'_Y$ and is defined via the formula
\begin{align}\label{eq:comp}
H_Y \# H'_Y(t,x):= H_Y(t,x) + H'_Y(t, (\varphi^t_{H_Y})^{-1}(x)).
\end{align}
The Hamiltonian flow of $ H_Y \# H'_Y$ is the composition of the flows of $H_Y$ and $H'_Y$, i.e.\ $\varphi^t_{H_Y\# H'_Y} = \varphi^t_{H_Y} \circ \varphi^t_{H'_Y}$.

\begin{remark}
    We only treat the special case in which every term of the bulk insertion $\frak{b}$ (Equation \ref{eq:bulkb}) is a multiple of the fundamental class of some twisted sector of the inertia orbifold $IY$.  Our technology is sufficient to treat more general bulks (allowing contribution from the trivial sector), but this case already reveals the interesting features of the construction, and suffices for all the applications in the sequel. 
\end{remark}

The proof of Theorem \ref{t:main}(A), which occupies most of the paper, is concluded in Section \ref{s:Haminv}.
Theorem \ref{t:main}(B) is proved in Section \ref{s:additive}, \ref{s:compatibleprod} and \ref{s:product} (see Equation \eqref{eqn:product}, Lemma \ref{l:prodcommute} and \ref{p:proPss}).

The construction of $HF(H_Y;\frak{b})$ is involved, but the basic idea follows a standard template.
Generators are formally capped oriented Hamiltonian orbits of $H_Y$ and the differentials count zeroes (weighted by $\frak{b}$) of a system of homotopy coherent multi-valued perturbations of global Kuranishi chart lifts of the moduli spaces of Floer cylinders between formally capped oriented Hamiltonian orbits. More precisely, in Section \ref{s:orbifoldFloer}, we define the objects of an ``ordered marked Floer flow category" $\cC(H_Y)$ (see Definition \ref{d:object}). Every object $c_x$ has a $\mathbb{Z}/2\mathbb{Z}$-grading (we call the grading parity, see Remark \ref{r:parityde}, Definition \ref{d:parity} and Remark \ref{r:underlyingparity}) and a $\frak{b}$-deformed action (Equation \eqref{eqn:action of capping}).
The group $\Pi$ acts on $c_x$ by formal recapping and changes the action accordingly (Equation \eqref{eq:formalaction}).
The Floer cochain complex is the $\mathbb{Z}/2\mathbb{Z}$-graded $\Lambda_{\Pi}$-module generated by the objects of  $\cC(H_Y)$ modulo the relation
$T^ac_x=a \cdot c_x$ for any $a \in \Pi$ (see Definition \ref{d:chaincomplex}(1) and \ref{d:chaincomplex2}).
For any pair of objects $c_x, c_y$, the morphism space $\cM(c_x,c_y)$ is the moduli space of possibly broken representable orbifold Floer cylinders from $c_x$ to $c_y$ (see Definition \ref{d:orbicylinder} and Equation \eqref{eq:orbicylinder}).

The underlying combinatorial data governing the stratification of $\cM(c_x,c_y)$ and how boundary strata are fibre products of other moduli spaces is explained in Section \ref{s:flow}, and we call these data an \emph{ordered marked flow category}.
These data are not sufficient to extract the differential.
In Section \ref{s:multivalue}, we define a \emph{homotopy coherent global chart lift} of an ordered marked flow category (Definition \ref{d:coherenthomotopy} and \ref{d:global lift}, see also \cite[Section 3.2]{Hirschi-Hugtenburg} for a similar formulation).
Together with the abstract smoothing theory (Lemma \ref{lem:smoothing}) and coherent multi-valued perturbations (\cite{CMS}, cf. Lemma \ref{lem:mvp_exists}, \ref{l:pertubationexist}), we can count the zeroes of the multi-valued perturbations (see Definition \ref{d:chaincomplex}(2)).
But instead of counting the zeroes directly, we weight the count by $\mathfrak{b}$ to get our $\mathfrak{b}$-deformed Floer cochain complex (see Remark \ref{r:deform} and Definition \ref{d:chaincomplex2}).
The construction of a single global chart for a moduli space of closed curves as relevant to orbifold Gromov-Witten theory is explained in Section \ref{s:globalchartclosed}.
It underlies the main idea of the construction of a global chart of a moduli space of orbifold Floer cylinders, which is explained in Section \ref{s:globalchart_cylinder}.
The construction of a homotopy coherent global chart lift is explained in Section \ref{s:zigzag} and \ref{s:system}.
We remark that representable maps from an orbifold curve $\scrC$ to $Y$ are the same as $\Gamma$-equivariant maps from a \emph{smooth} curve $\Sigma$, a $\Gamma$-principal bundle over $\scrC$, to $X$ (see Theorem \ref{t:stablemap}). This allows us to work with smooth manifolds most of the time.

The existence of the isomorphisms \eqref{eq:PSS} is sufficient to construct spectral invariants.  

\begin{theo}\label{t:spectral}
Consider a bulk \[
\mathfrak{b}:=\sum_{(g)} \frak{v}_{(g)} PD([X^g/C(g)]) \in QH_{orb}(Y,\mathfrak{b};\Lambda_{>0}) := H_{orb}(Y;\Lambda_{>0}).\] 
For any class $x$ in the $\mathfrak{b}$-deformed orbifold quantum cohomology\footnote{This and the previous equality are additive identifications.} $QH_{orb}(Y,\mathfrak{b}):=H_{orb}(Y;\Lambda_{\Pi})$, there is an associated spectral invariant $c^{\frak{b}}(x,-):C^{\infty}(Y \times S^1; \bR) \to \mathbb{R}$ with the following properties.  

\begin{enumerate}
   \item (Symplectic invariance) $c^{\frak{b}}(x,H_Y\circ \psi) = c^{\frak{b}}(x,H_Y)$ for any $\psi \in \Symp(Y, \omega_Y)$;
    \item (Spectrality) for any $H_Y$, $c^{\frak{b}}(x,H_Y)$ lies in the action spectrum $\Spec(Y,H_Y)$;
\item (Hofer Lipschitz) for any $H_Y, H'_Y$
$$  \int_{0}^1 \min  ((H_Y)_t - (H'_Y)_t) dt \leq c^{\frak{b}}(x,H_Y) - c^{\frak{b}}(x,H'_Y) \leq  \int_{0}^1 \max\,  ((H_Y)_t - (H'_Y)_t) dt;$$
\item (Monotonicity) if $(H_Y)_t \leq (H'_Y)_t$ then $c^{\frak{b}}(x,H_Y) \leq c^{\frak{b}}(x,H'_Y)$;
\item (Homotopy invariance) if $H_Y, H'_Y$ are mean-normalized and determine the same point of the universal cover $\widetilde{\Ham}(Y,\omega_Y)$, then $c^{\frak{b}}(x,H_Y) = c^{\frak{b}}(x,H'_Y)$;
\item (Shift) $c^{\frak{b}}(x,H_Y + s(t)) = c^{\frak{b}}(x,H_Y) +  \int_0^1 s(t)\, dt$;
\item (Subadditivity) for any $x,x' \in QH_{orb}(Y,\mathfrak{b})$ and $H_Y,H'_Y$, $c^{\frak{b}}(x \ast_{\frak{b}} x',H_Y \# H'_Y) \leq c^{\frak{b}}(x,H_Y) + c^{\frak{b}}(x',H'_Y)$, where $\ast_{\frak{b}}$ is the bulk-deformed orbifold quantum product in $QH_{orb}(Y,\mathfrak{b})$.
\end{enumerate}
\end{theo}

The action spectrum in (2), introduced  in Definition \ref{d:action}, is the set of values of the action on `formally capped' orbits of the Hamiltonian.

The spectral invariant is defined and Theorem \ref{t:spectral} is proved in Section \ref{s:spectralprop}.

Our construction of orbifold Hamiltonian Floer cohomology is highly influenced by the fundamental works of Bai-Xu \cite{BX} and Rezchikov \cite{Rez}, see also \cite[Appendix]{Abouzaid-Blumberg}, who construct Hamiltonian Floer cohomology on a general symplectic manifold. Nonetheless, the orbifold setting throws up a number of new features, some of which we summarise here:

\begin{enumerate}
    \item We construct global charts for moduli spaces of orbifold Floer cylinders; the underlying spaces of domains are spaces of orbifold curves embedded in $\bP(V)/\Gamma$ where $V$ is a faithful $\Gamma$-representation. (These can be understood as $\Gamma$-principal bundles in $\bP(V)$ via the technology of admissible covers \cite{ACV}.) Relative to the case of stable maps to a manifold, a $\Gamma$-equivariant homotopy class of maps from a Riemann surface to $\bP(V)$ plays the role of the underlying homology class, and the representation $V$ plays the role of the degree of the map. Boundary strata are given by nodal curves in $\bP(V)/\Gamma$, where each component is in a $\bP(V_i)/\Gamma \subset \bP(V)/\Gamma$ that is dual to a surjective morphism of $\Gamma$-representations $V^* \to V_i^*$. We are able to avoid any deep analysis of the boundary strata by a zig-zag trick (cf. Section \ref{s:zigzag}) but the cost is that our global lift is only homotopy coherent as opposed to strictly compatible as in \cite{BX}.

    \item As understood by Cho and Hong \cite{Cho-Hong} in orbifold Morse theory, the moduli spaces of compactified gradient flow lines are not manifolds with corners: broken flow-lines may admit several different gluings. We interpret this
    in terms of choices in identifying sets of  asymptotic markers; in particular, we do obtain moduli spaces which are manifolds with corners when considering boundary strata of admissible covers labelled with suitable additional data at the nodes (cf. Section \ref{s:globalchart_cylinder} and Lemma \ref{l:realblowup}). As an important intermediate step, we clarify what  data an admissible cover of a broken orbifold cylinder with asymptotic markers should carry (cf. Section \ref{s:orbifoldFloer}). 
    
    \item The global chart lift of the orbifold Floer flow category is still constructed inductively in action, but the zig-zag relating the boundary of a chart to a chart on the boundary involves a homotopy of sections, which means the differentials in the chain complex eventually count zeroes of multivalued sections in some homotopy cube, cf. Section \ref{sec:ordered marked}. 
    \item The presence of positive-valuation bulk insertions means we encounter infinitely many moduli spaces for fixed asymptotics, and modifies the relationship between degrees of curves and their integralised actions, over which we induct. Furthermore, the presence of automorphisms of objects (isotropy groups of orbits) and of ordered interior marked points modifies the combinatorial structure of boundary breaking, which leads to the notion of an \emph{ordered marked flow category} which is a slight variation on the usual definition  (see Definition \ref{d:omfc}).
    \item The dynamical applications in \cite{MSS2} involve spectral invariants, in particular, we require the orbifold Floer product, which plays no role in establishing Arnol'd-type bounds as in \cite{BX,Rez}.
\end{enumerate}
\begin{remark}
    The construction of orbifold Gromov-Witten theory for general orbifolds is the subject of ongoing work of McLean and Ritter. Following seminal work of Ross and Thomas \cite{RT}, their theory is built on moduli spaces of domains comprised of curves in weighted projective spaces, rather than in global quotients of projective spaces.
\end{remark}

\begin{remark}\label{r:parityde}
    Orbifold cohomology $H_{orb}(Y)$ (sometimes called Chen-Ruan cohomology) is $\mathbb{Q}$-graded in general. 
    But there is a parity decomposition $H_{orb}(Y)=H_{orb}^{even}(Y) \oplus H_{orb}^{odd}(Y)$, with respect to which the product structure is supercommutative.
    Likewise, bulk-deformed orbifold quantum cohomology $QH_{orb}(Y,\mathfrak{b})$  is $\mathbb{Q}$-graded with an additional parity decomposition. 
    In this paper, we are going to use a simplified set-up where bulk-deformed orbifold quantum cohomology and Hamiltonian Floer cohomology will be viewed as ungraded theories, but we retain the parity decomposition to make sense of the supercommutativity of the quantum product so the resulting theory is $\mathbb{Z}/2\mathbb{Z}$-graded by parity.
\end{remark}

\begin{remark}
    The theory without bulk insertions and without the pair of pants product becomes  significantly simpler because one only needs to count orbifold cylinders without interior marked points.  Such a cylinder mapping  to $Y$ can be re-interpreted as a strip mapping to $X$ (with an identification of boundary strips under the action of a suitable group element). 
This viewpoint was studied in \cite{MMRM} and \cite{GZ21}.
\end{remark}

\noindent \textbf{Acknowledgements.} 
The authors are grateful to Mohammed Abouzaid, Amanda Hirschi and Mark McLean for several helpful conversations on global charts, and in particular to Mark McLean for suggestions on global charts for global quotients. We are also grateful to Dhruv Ranganathan for lessons on twisted stable maps. We are especially grateful to Roman Krutowski who identified a serious error in our original treatment, and to the anonymous referee who provided numerous insightful comments leading to improvements, corrections and clarifications. It will be obvious to the reader that the paper owes a significant debt to the works \cite{BX,Rez}. We also thank Michael Hutchings and the UC Berkeley Department of Mathematics for their warm hospitality in December 2022, when this project was initiated.

C.Y.M is partially supported by the Royal Society University Research Fellowship.
I.S. is partially supported by UKRI Frontier Research Grant EP/X030660/1 (in lieu of an ERC Advanced Grant).
S.S. is partially supported by ERC Starting Grant number 851701.

\section{Orbifold Quantum cohomology and Morse cohomology}

\subsection{Orbifold cohomology}\label{s:orbicoh}

\emph{We define orbifold cohomology for a global quotient, where the inertia stack has a particularly simple description.}

Let $\Gamma$ be a finite group acting faithfully on a closed connected symplectic manifold $X$ by symplectomorphisms, and let $Y := [X/\Gamma]$ denote the orbifold quotient, viewed as a groupoid \cite{OrbifoldBook, Lerman}. 
Let $J_Y$ be a $\omega_Y$-tamed almost complex structure on $Y$. This is equivalent to the data of a $\Gamma$-equivariant $\omega_X$-tamed almost complex structure $J_X$ on $X$.
We will write $Y^{\mathrm{reg}} \subset Y$ for the `regular locus', i.e. the locus where the isotropy group is trivial.  If $g\in \Gamma$ we write $X^g$ for the fixed-point set.  Define
\[
H^*(X,\Gamma) = \oplus_{g\in \Gamma} H^*(X^g)
\]
If $\Delta: \Gamma \times X \to X \times X$ is the action map $(g,x) \mapsto (gx,x)$ then this is equivalently $H^*(\Delta^{-1}(\Delta_X))$.

There is a rational grading on $H^*(X,\Gamma)$ in which the connected component $Z\subset X^g$ has $H^*(Z)$ shifted by the value $2a(g,Z)$. Here the `age' $a(g,Z) = \sum_{j=1}^{\dim_{\mathbb{C}} X} r_j$ where, for $z\in Z$, $g$ acts on $(T_zX,J_X)$ with weights $\exp(2i\pi r_j)$, $r_j \in \Q \cap [0,1)$. 
The age $a(g,Z)$ is independent of the choice of $z \in Z$, as the notation suggests.
One checks 
\begin{align}
a(g,Z) + a(g^{-1},Z) = \mathrm{codim}_{\bC}(Z)=\dim_{\bC}(X)-\dim_{\bC}(Z). \label{eq:agesum}
\end{align}
When $g$ and $g^{-1}$ are conjugate, $a(g,Z)=a(g^{-1},Z)$ so the age is half-integral and the grading on $H^*(X,\Gamma)$ is then integral.  The identity has age zero.

On top of the rational grading, there is a parity decomposition of $H^*(X,\Gamma)$ (see \cite[Definition 1.8]{FG03}) given by
\[
H^{even}(X,\Gamma) := \oplus_{g\in \Gamma} \oplus_{k \in \mathbb{Z}} H^{2k}(X^g), \quad \text{ and } \quad H^{odd}(X,\Gamma)=\oplus_{g\in \Gamma} \oplus_{k \in \mathbb{Z}} H^{2k+1}(X^g).
\]
For $h \in \Gamma$ and $\alpha \in H^*(X^g)$, we define $h \cdot \alpha:= (h^{-1})^*\alpha \in H^*(X^{hgh^{-1}})$. This defines a $\Gamma$ action on $H^*(X,\Gamma) $ which preserves the rational grading and the parity decomposition. 
As a graded vector space, the \emph{orbifold cohomology} of $Y$ is  (see \cite{Chen-Ruan, FG03, LS03})
\begin{equation} \label{eqn:orbifold_cohomology_decomposition}
H^*_{orb}(Y) := H^{*+2\mathrm{age}}(X,\Gamma)^\Gamma \cong \oplus_{(g) \in \|\Gamma\|} H^{*+2a(g,Z)}(X^g/C(g))
\end{equation}
where $\|\Gamma\|$ denotes the set of conjugacy classes in $\Gamma$, the corresponding summand is described via a choice of representative $g$ for the conjugacy class, $C(g)$ is the centralizer of $g$, and the grading shifts depend on the decomposition of $X^g$ into its set of connected components $\{Z\}$.  This is (a regrading of) the cohomology of the inertia orbifold $IY$ of $Y$. The non-identity summands in \eqref{eqn:orbifold_cohomology_decomposition} are called `twisted sectors'.
The even and odd parts are defined by  
\begin{align}\label{eq:parity}
H^{even}_{orb}(Y):=H^{even}(X,\Gamma)^\Gamma, \quad \text{ and } \quad H^{odd}_{orb}(Y):=H^{odd}(X,\Gamma)^\Gamma.
\end{align}

\begin{remark} The cyclic group $\langle g\rangle \subset C(g)$ acts trivially on $X^g$, so the orbifold $X^g/C(g)$ is not reduced.  One often passes to the `rigidified inertia stack' with components $X^g / (C(g)/\langle g\rangle)$. \end{remark}

\begin{example}
    For any $(g) \in \|\Gamma\|$, the Poincar\'e dual of the fundamental class, $PD([X^g/C(g)])$, lies in $H^{even}_{orb}(Y)$.
    Therefore, the $\mathfrak{b}$ in Equation \eqref{eq:bulkb} is in $H^{even}_{orb}(Y;\Lambda_{\Pi})$.
\end{example}

\begin{example}
If $Y = \Sym^n(M)$, the $n$-fold symmetric product of a symplectic manifold $M$, then $X=M^n$ and $\Gamma = \Sym_n$. 
In this case, $\|\Sym_n\|$ is indexed by cycle types, and the twisted sectors are products of smaller symmetric products of $M$ (however,  as orbifolds these may now have non-trivial generic stabiliser). Since a cycle is conjugate to its inverse, the grading on $H^*_{orb}(Y)$ is integral.
\end{example}

The associative ring structure is given by the `Chen-Ruan' product, which can be seen as the degree zero part of the quantum product (i.e. contributions from constant representable orbifold maps, see Section \ref{s:orbifoldmaps}, but with possibly non-trivial obstruction bundles). The degree shifts in the definition ensure that this product respects the grading.

\subsection{Representable orbifold maps and admissible covers}\label{s:orbifoldmaps}

\emph{In this section we give some background on orbifold stable maps.}

An orbifold prestable curve $\scrC$ comprises a prestable curve $C$ (the underlying coarse moduli space of the stacky curve $\scrC$)  together with a stack structure at the nodes and marked points. Concretely, at a marked point $z \in C$ we have an integer label $m_z = m \geq 1$ and  a disc neighborhood $U$ of $z$ which is uniformized by the branched covering map
$z \mapsto z^{m}$, i.e.  there is an open subset $V \subset \mathbb{C}$
invariant with respect to the $G_{z}:=\mathbb{Z} / m\bZ$ action $z \mapsto z\exp(2i \pi/m)$, 
together with a surjective $G_{z}$-invariant holomorphic map $\pi: V \to U$, such that  the induced map $V/G_z \to U$ is bijective.  (When $m_z=1$, $z$ is a smooth point.) We will always assume that the stack structure at a node $p$ is \emph{balanced}, meaning that  if $\mu_m$ is the group of $m$-th roots of unity with generator $t$ and if $\xi = e^{2i\pi/m}$ then a node on $\scrC$ is locally analytically equivalent to
\[
\{zw=0\} / \mu_m \qquad  (z,w) \cdot t = (\xi z, \xi^{-1} w)
\]
The balancing condition is necessary for the node to be smoothable.

We will be considering holomorphic maps from orbifold curves to orbifolds. 
To linearise the $\cdbar$-operator at an orbifold map, one must pull back the tangent bundle, but in general a map of orbifolds does not yield a well-defined pullback map on orbifold vector bundles (the issue arises for maps contained in the complement of the regular locus). The extra information required to define the pullback is a choice of a `compatible system' in the sense of \cite[Definition 4.4.1]{Chen-RuanGW}, \cite[Definition 16.12]{Cho-Poddar}. A compatible system comprises an indexing set $I$ and a covering of the domain and an open neighbourhood of the range of the given orbifold map by charts indexed by $I$ (in particular the charts are in bijection in domain and range), satisfying a number of compatibilities under inclusion maps and with respect to prescribed local uniformizers.  A  compatible system induces a group homomorphism between the isotropy groups $G_z \to G_{v(z)}$ for $z \in \scrC$ and $v:\scrC \to Y$, 
see Definition 16.11 (2)(b) of \cite{Cho-Poddar} and the subsequent paragraph, or the paragraph before \cite[Definition 4.4.1]{Chen-RuanGW}. 
Under an isomorphism of compatible systems, this group homomorphism is changed by post-composition with an automorphism of the target $G_{v(z)}$, in particular the  
(non-)injectivity of this homomorphism depends only on the isomorphism class of a compatible system.

The precise definition of compatible system will not play a role in the sequel, so we defer such to the references.  It is possible that a smooth orbifold map has no compatible system, or has more than one isomorphism class of compatible systems.  (Two compatible systems for an orbifold map $u$ are isomorphic precisely when they give rise to isomorphic pullback orbifold bundles $u^*E$ for all orbifold bundles $E$ on the target, cf. \cite[Lemma 16.1]{Cho-Poddar}.) 
A  smooth orbifold map that has a compatible system is called {\it good}.

When $Y = [X/\Gamma]$ is a global quotient, there is another description of good orbifold maps in terms of `admissible covers', which we will use systematically. In this context, good maps are called representable maps in the literature.
\begin{theo}[\cite{OrbifoldBook}, Theorem 2.45]\label{t:stablemap}
A representable orbifold map $v:\scrC \to Y$ is the same data as:
\begin{enumerate}
\item a (possibly disconnected) nodal Riemann surface $\Sigma$ with $\Gamma$-action so that 
\item $\pi_{\Sigma}:\Sigma \to \Sigma/\Gamma = \scrC$ is a $\Gamma$-principal bundle, branched at the orbifold points of $\scrC$ with ramification orders given by the orbifold orders of the marked points, and so that
\item the $\Gamma$-action is balanced at all nodes of $\Sigma$, together with
\item a $\Gamma$-equivariant map $u:\Sigma \to X$.
\end{enumerate}
\end{theo}

The $\Gamma$-equivariant map $\Sigma \to X$ determines a map $\scrC \to Y$.
The assumption that every component of $\Sigma$ has smooth normalisation (i.e. is  not an orbifold) guarantees that $\scrC \to Y$ is a representable orbifold map.

 Conversely, given a representable orbifold map $v:\scrC \to Y$, we can obtain the curve $\Sigma$ from the pullback square
\[
\xymatrix{
\Sigma \ar[r]^u \ar[d]^{\pi_{\Sigma}} & X \ar[d] \\ \scrC \ar[r]^v & Y
}
\]
The fact that the map $\scrC \to Y$ is representable ensures that every component of the fibre-product $\Sigma$ is smooth and not an orbifold, and the $\Gamma$-action on $\Sigma$ is balanced.

\begin{remark}\label{rmk:automorphism_of_cover}
An automorphism of a representable stable orbifold map $v:\scrC \to Y$ (or equivalently, a $\Gamma$-equivariant map $u:\Sigma \to X$) is a $\Gamma$-equivariant bundle isomorphism $f:\Sigma \to \Sigma$ covering the identity map of $\scrC$ such that $u(f(z))=u(z)$ for all $z \in \Sigma$.

For example, when $Y=[pt/\Gamma]$, the automorphism group of every map $v:\scrC \to Y$ contains the centre $Z(\Gamma)$.
\end{remark}

\begin{remark} When $Y$ is algebraic, a representable orbifold map is exactly a representable 1-morphism of the underlying Deligne-Mumford stacks (i.e. a representable map is a morphism of orbifolds viewed as groupoids).  See e.g. \cite[Section 2.4]{OrbifoldBook}. If $v:\scrC \to Y$ is any smooth orbifold map with $\scrC$ irreducible and  $v^{-1}(Y^{\reg})$ connected and dense, then $v$ is good (ie. representable) with a unique choice of  isomorphism class of compatible system. \end{remark} 

\begin{remark}\label{rmk:log_can_is_ample} For a smooth Riemann surface $C$ with marked points $D\subset C$, $(C,D)$ is stable if and only if $\omega_C(D)$ is ample.  The same result holds for an orbifold curve $\scrC$, taking the orbifold log canonical bundle. See \cite[Definition 2.1]{Cheong-CiocanFontanine-Kim} (in the case $\varepsilon>2$). \end{remark}

\subsection{Pseudoholomorphic representable orbifold stable maps}\label{s:algebraicstack}

\emph{We define the moduli space of orbifold stable maps and give its virtual dimension.}

A $J_Y$-holomorphic stable map $v$ from a pre-stable orbifold curve $\scrC$ to an orbifold $Y$ 
is a representable orbifold map
such that
\begin{enumerate}
\item $v: C \to Y$ is a $J_Y$-holomorphic map on the underlying coarse space $C$ of $\scrC$, i.e. $v$ is $J_Y$-holomorphic on each irreducible component;
\item the set of automorphisms of $v$ is finite.
\end{enumerate}
Under Theorem \ref{t:stablemap}, the first condition is equivalent to the associated $\Gamma$-equivariant map $u:\Sigma \to X$ being $J_X$-holomorphic.

Abramovich-Vistoli \cite{AV} prove that there is a compact moduli space $\cK_{g,h}(Y,\beta)$ of stable orbifold maps where the discrete data comprises:  the genus $g$, number of marked points $h$, and homology class $\beta$ on the target. 
The moduli space $\cK_{g,h}(Y,\beta)$ is a proper Deligne-Mumford stack and its coarse moduli space has a metrizable Gromov topology.

\begin{remark}
If $Y = [X/\Gamma]$ is a global quotient orbifold, then $\beta \in H_2(Y;\mathbb{Q})= H_2(X;\mathbb{Q})^{\Gamma}$, where the equality holds because 
$1/|\Gamma|$ exists in $\mathbb{Q}$. 
\end{remark}

If $v: \scrC\to Y$ is a representable stable orbifold map, with compatible system $\xi$ and $z$ is a marked point,  the local orbifold fundamental group has a well-defined positive generator $t$ (the image of $\{e^{2i\pi \theta}\}_{\theta \in [0,1/|G_z|]}$ in $V\subset \bC$ under the local uniformising map described previously), and the image of $t$ under  the homomorphism $\rho_{\xi,z}: G_z \to G_{v(z)}$ determined by the compatible system $\xi$ determines an element $\rho_{\xi,z}(t) \in G_{v(z)}$.  The conjugacy class of this element is independent of the choice of $\xi$, and so the evaluation map at $z$ naturally lifts to the inertia orbifold $IY$:
\begin{align*}
&\ev_{z_i}: \cK_{g,h}(Y,\beta) \longrightarrow IY, \ 1\leq i \leq h \\
&\ev_{z_i}(v)=(v(z_i),\rho_{\xi,z}(t)) \in X^{\rho_{\xi,z}(t)}/C(\rho_{\xi,z}(t)) \subset IY.
\end{align*}
Because of the balancing condition at nodes, it is also natural to introduce a \emph{twisted evaluation map}
\[
\ev^{\tw}_{z_i} :  \cK_{g,h}(Y,\beta) \longrightarrow IY, \ 1\leq i \leq h
\]
which is defined by
\[
\ev^{\tw}_{z_i} = \iota\circ \ev_{z_i}
\]
where $\iota: IY \to IY$ is the canonical involution of the inertia stack given by $(y,(g)) \mapsto (y, (g^{-1}))$ for $(g) \in \|\Gamma\|$.

The map 
\[
\cK_{g,h}(Y,\beta) \to  \|\Gamma\|^h \text{ given by } v \mapsto (\rho_{\xi,z_1}(t), \dots, \rho_{\xi,z_h}(t))
\]
is locally constant so we have a decomposition
\begin{align}\label{eq:closedmodulidecom}
\cK_{g,h}(Y,\beta)= \sqcup_{\bfg \in \|\Gamma\|^h} \cK_{g,h, \bfg}(Y,\beta)
\end{align}
into a union of open and closed substacks.

\begin{prop}[\cite{Chen-Ruan, AV}]\label{p:vdim}
The space $\cK_{g,h,\bfg}(Y,\beta)$ carries a virtual fundamental cycle of $\bC$-dimension $\langle c_1(TY),\beta\rangle + (\dim_{\bC}(Y)-3)(1-g) + h - \sum_{i=1}^h a(v,z_i)$, where $a(v,z_i):=a(\rho_{\xi,z_i}(t), \ev_{z_i}(v))$ denotes the age grading.
\end{prop}

\begin{remark}\label{r:indexformula}
The formula for the virtual dimension in \cite{Chen-Ruan} (Proposition 4.2.2, Lemma 3.2.4) is obtained as follows: for $(v:\scrC \to Y) \in \cK_{g,h}(Y,\beta)$, the complex vector bundle $v^*TY$ admits a canonical desingularization $|v^*TY|$ over the desingularization $C$ of $\scrC$ \cite[Proposition 4.2.2]{Chen-Ruan}.
The sheaf of holomorphic sections of $v^*TY$ over $\scrC$ agrees with the sheaf of  holomorphic sections of $|v^*TY|$ over $C$.
As a result, one can compute the virtual dimension as $\langle c_1(|v^*TY|), C\rangle + (\dim_{\bC}(Y)-3)(1-g) + h$, where the term $h$ comes from varying the marked points.
Finally, a local computation shows that $\langle c_1(|v^*TY|), C\rangle=  \langle c_1(v^*TY), \scrC\rangle - \sum_{i=1}^h a(\rho_{\xi,z_i}(t), \ev_{z_i}(u))$ and hence the result follows.
\end{remark}

\begin{remark}\label{r:breaking}
The breaking of a holomorphic sphere into $v$ and $v'$ joined along an orbifold node is a (complex) codimension one phenomenon. To see this, fix components $\cM(v) \subset  \cK_{g,h+1}(Y,\beta)$ and $\cM(v') \subset  \cK_{g',h'+1}(Y,\beta')$ with the property that the (twisted) evaluation maps 
$ev_{z_1}: \cM(v) \to IY$
and $ev_{z_1'}^{\tw}: \cM(v') \to IY$ have image in the same component $Z$ of a twisted sector of $IY$.  Then  the virtual complex dimension of the fibre product of these components over the evaluation maps is given by
\begin{align*}
&\vdim(v)+\vdim(v')- \dim(Z)\\
=&\, \langle c_1(TY),\beta\rangle + (n-3)(1-g) + h+1 - \sum_{i=1}^{h+1} a(v,z_i) \\ 
& + \langle c_1(TY),\beta'\rangle + (n-3)(1-g') + h'+1 - 
 \sum_{i=1}^{h'+1} a(v',z_i')- \dim(Z) \\
=&\, \langle c_1(TY),\beta+\beta' \rangle + (n-3)(1-(g+g'))+(h+ h')+(n-1) \\ &-  \sum_{i=1}^h a(v,z_i) - \sum_{i=1}^{h'} a(v',z_i')- \dim(Z)\\
=&\, \vdim(\cK_{g+g',h+h',\bfg}(Y,\beta+\beta'))+(n-1)-a(v,z_1)-a(v',z_1')-\dim(Z)
\end{align*} 
where $\bfg$ is the concatenation of $(\rho_{\xi,z_2}(t), \dots, \rho_{\xi,z_h}(t))$ and $(\rho_{\xi',z_2}(t), \dots, \rho_{\xi',z_{h'}}(t))$ from $v$ and $v'$. 
An element in the fibre product is smoothable only if $a(v,z_1)+a(v',z_1')=n-\dim(Z)$ (cf. \eqref{eq:agesum}) in which case the virtual dimension equals  $\vdim(\cK_{g+g',h+h',\bfg}(Y,\beta+\beta'))-1$.

This example shows that the degree of $|(v\# v')^*TY|$ is not the sum of the degree of $|v^*TY|$ and $|(v')^*TY|$ if the node has a nontrivial orbifold isotropy group. That is, the degree of $|(v\# v')^*TY|$ is not additive  under breaking.
\end{remark}

In fact, we have a decomposition more refined than \eqref{eq:closedmodulidecom}. Let $g=0$.
Given $v \in \cK_{0,h}(Y,\beta)$, we have the associated $u:\Sigma \to X$ and $\pi_{\Sigma}:\Sigma \to \scrC$.
Suppose that $\scrC$ is irreducible.
We choose a base point and a collection of pairwise disjoint simple loops from the base point such that the $i^{th}$ loop wraps around the $i^{th}$ marked point and not the other marked points.  
The monodromy of the loops for the bundle $\pi_{\Sigma}$ give us an element $\bfg=(g_1,\dots,g_h) \in \Gamma^h$ such that $g_1 \dots g_h=id$.

Two different collections of pairwise disjoint simple loops will give us two different $\bfg, \bfg'$ that are related by a pure Hurwitz equivalence (pure because the marked points are ordered); this is the equivalence relation generated by 
\[
(g_1,\ldots,g_h) \sim (g_1,\ldots,g_{i-1}, g_ig_{i+1}g_i^{-1},g_i,g_{i+2},\ldots,g_h) \qquad 1\leq i \leq h-1
\] 
which comes from the standard action of the braid group as automorphisms of a free group.

Conversely, given a $\bfg'$ in the same pure Hurwitz orbit as $\bfg$, we can find a collection of pairwise disjoint simple loops realizing $\bfg'$.
Changing the base point doesn't change the pure Hurwitz orbit.
In this way, we associate to $v$ a pure Hurwitz orbit.
When $\scrC$ is reducible with components $\scrC_1, \dots, \scrC_l$ so $\Sigma=\Sigma_1 \vee \dots \vee \Sigma_l$, we cannot choose disjoint simple loops as before. What we can do instead is smooth the nodes in $\scrC$ and nodes in $\Sigma$ simultaneously (preserving the $\Gamma$-principal bundle structure).
We emphasize that $\Sigma$, as a nodal Riemann surface, is determined by the representable orbifold map $v$ (i.e. how $\Sigma_i$ and $\Sigma_j$ are joined together is determined by $v$ for all $i,j$).
From a local model, one can see that there is an $S^1$-family of  choices of smoothing for each node in $C$, irrespective of whether  the node has a non-trivial isotropy group or not (cf. Example \ref{e:realblowup}).
After smoothing, we can associate to it a pure Hurwitz orbit as above, and it is independent of the choice of smoothing (because the association to a pure Hurwitz orbit is a locally constant map). Let $\cH_h$ be the set of pure Hurwitz orbits in $\Gamma^h$.
Then we have a decomposition
\begin{align}\label{eq:closedmodulidecom2}
\cK_{0,h}(Y,\beta)= \sqcup_{\bfm \in \cH_h} \cK_{0, \bfm}(Y,\beta).
\end{align}
into a union of open and closed substacks.

The definition of the orbifold quantum product  later in \eqref{eqn:quantum_product} relies only on  the existence of the  virtual class $ [\cK_{0,3+h}(Y;\beta)]^{vir} $. Such has been constructed by Chen-Ruan \cite{Chen-RuanGW} and Abramovich-Vistoli \cite{AV}.  The symplectic definition can be simplified by using global Kuranishi charts rather than Kuranishi atlases. In Section \ref{s:globalchartclosed} we discuss such a construction, which we will then adapt for the construction of orbifold Hamiltonian Floer theory.

\subsection{Orbifold quantum product}\label{s:orbQH}

\emph{We quickly review the definition of the bulk-deformed orbifold quantum cohomology algebra.}

Let $\frak{b} \in H^{even}_{orb}(Y; \Lambda_{>0}) = H^{even}(IY;\Lambda_{>0})$ (recall the parity decomposition in \eqref{eq:parity}). 
Restrict attention to the spaces $\cK_{0,3+h}(Y;\beta)$ where $h$ marked points $\{z_1\ldots,z_h\}$ carry `bulk' insertions, two $z_0,z_0'$ are inputs and one $z_{\infty}$ is an output. (Any of the marked points may be stacky.)   We define the $\frak{b}$-bulk deformed orbifold quantum cohomology $QH_{orb}(Y,\frak{b})$ to be the group $H_{orb}(Y;\Lambda_{\Pi})$ with the product
\begin{equation} \label{eqn:quantum_product}
u \ast v = \sum_{\beta} (u \ast v)_{\beta} T^{\omega_Y(\beta)}
\end{equation}
with 
\begin{equation}  \label{eqn:quantum_product_formula}
(u \ast v)_{\beta} =  PD \left( (ev_{z_{\infty}})_* (ev_{z_0}^*u \cdot ev_{z_0'}^*v\cdot \sum_{h \geq 0} \frac{1}{h!} \,\prod_{i=1}^h ev_{z_i}^*\frak{b} \cap  [\cK_{0,3+h}(Y;\beta)]^{vir} \right)
\end{equation}
with $PD$ denoting Poincar\'e duality in $H^*_{orb}(Y; \Lambda_{\Pi})$, noting that we are working in characteristic zero and $\Lambda_{\Pi}$ is a field. The particular choices of insertion for $\frak{b}$ give constraints to the data $\bfm$ (see Equation \eqref{eq:closedmodulidecom2}). Later we will take $\frak{b}$ to be a linear combination of the fundamental class of non-trivial sectors.
The virtual class appearing in the formula can be taken to be the one constructed in \cite{AV}.
However, for the purpose of comparing with the Hamiltonian Floer theory, we are going to use a different virtual class that will be constructed in \ref{s:globalchartclosed} (see Proposition \ref{prop:chart for closed curves}, \ref{p:doubly-closed}). We do not attempt to compare our virtual class to the one in  \cite{AV}.

\begin{prop}\label{p:QHwelldefined}
The bulk-deformed quantum product is well-defined (i.e. independent of $J_Y$). The resulting algebra is unital, associative, and supercommutative with respect to the parity decomposition.
\end{prop}

\begin{proof}
The first statement holds because of the well-definition of the virtual class (i.e. independent of the choice of global chart presentation for the moduli space, including cobordism invariance under changing $J_Y$). The other properties are then standard.
For the supercommutativity with respect to the parity decomposition, see \cite{FG03}.
\end{proof}

\subsection{Orbifold Morse cohomology\label{sec:Morse}}
\emph{In this section we give a Morse theory description of orbifold (quantum) cohomology, for later use when comparing with Hamiltonian Floer cohomology.}

Let $X$ be oriented, let $\Gamma$ be a finite group acting on $X$ preserving the given orientation, and  let $Y = [X/\Gamma]$. A generic $\Gamma$-invariant function $f: X \to \R$ is  Morse \cite{Was}, but there may in general be no $\Gamma$-equivariant Riemannian metric $g$ making $(f,g)$ a Morse-Smale pair (failure of equivariant transversality).  A recent result of \cite{Bao-Lawson} circumvents this issue. Suppose in the situation above $p\in X$ is a critical point of $f$, with stabiliser group $H \leq \Gamma$. There is a canonical splitting
\[
T_p X = (T_pX)^H \oplus (T_pX)^{\perp} \qquad (T_pX)^{\perp} = \{ v \in T_p X \, | \, \sum_{h\in H} dh|_p(v) = 0\}
\]
into the $H$-invariant subspace and its complement; the latter would agree with the $g$-orthogonal complement to $(T_pX)^H$ with respect to any equivariant Riemannian metric $g$. We say that $p$ is stable if  the Hessian of $f$ at $p$ is positive definite on $(T_pX)^{\perp}$, and $f$ is stable if all critical points are stable.  Then

\begin{lemma}[Bao-Lawson, \cite{Bao-Lawson}]\label{l:stableMorse}
For any $\Gamma$-equivariant Morse function $f$, there is a $C^0$-small perturbation $\hat{f}$ of $f$, with perturbation supported near $\mathrm{Crit}(f)$,  with the properties that
\begin{itemize}
\item $\hat{f}$ is a $\Gamma$-equivariant stable Morse function (in particular, a smooth function)
\item for any $\Gamma$-equivariant Riemannian metric $g$ and any $k$, there are arbitrarily small $\Gamma$-equivariant $C^k$-perturbations $\hat{g}$ of $g$, with perturbation supported in the complement of any prescribed neighbourhood of the critical points, making $(\hat{f},\hat{g})$ a Morse-Smale pair.
\end{itemize}
\end{lemma}
Note that $\hat{f}$ has more critical points than $f$, and the perturbation of the function is necessarily not $C^2$-small.  The key point is that after perturbation at each critical point $p$ the stabiliser group of $p$ pointwise fixes the descending manifold $W^u(p;f)$ of $f$ at $p$. The second part of the Lemma (cf. \cite[Lemma 7.7]{Bao-Lawson}) shows that one can assume that the metric $\hat{g}$  is diffeomorphic to  the standard Euclidean metric in small neighbourhoods of the critical points of $\hat{f}$, whilst maintaining that $(\hat{f},\hat{g})$ is a Morse-Smale pair. We will work under these hypotheses throughout (but to simplify notation will now drop the typographic hats).

Descending to the quotient, and writing $(\bar{f},\bar{g})$ for the induced function and metric on $Y$ arising from $(f,g)$, we thus have a Morse-Smale pair on the orbifold $Y$. The Morse complex in this setting was studied by Cho and Hong \cite{Cho-Hong}.

 If $x\in \mathrm{Crit}(f)$, then the isotropy group $\Gamma_x$ acts on the unstable manifold $W^u(x;f)$. We say that the critical point $x$ is \emph{orientable} if $\Gamma_x$ acts on $W^u(x;f)$ by orientation-preserving diffeomorphisms. One can check that for a critical point $y$ of $\bar{f}$, one preimage in $X$ is orientable if and only if all are. We write $\Crit(\bar{f})^+$ for the subset of orientable critical points.

  We define a chain complex
\begin{equation} \label{eqn:morse}
C^*_{Morse}(Y;f) := (\oplus_{y \in \Crit(\bar{f})^{+} }\, \mathbb{K}\cdot \langle y\rangle)
\end{equation}
to be generated by 
the \emph{oriented} critical points of $\bar{f}$.  The differential of an oriented critical point $y$ is defined by a weighted sum of critical points:
\[
\partial(y) = \sum_{z} n(y,z) z
\]
Here $z$ ranges through oriented critical points, and
\begin{align}\label{eq:Morse_weight}
n(y,z) = \sum_{u} \epsilon(u) |\Gamma_z| / |\Gamma_u|
\end{align}
where $u$ is an isolated flow line between $y$ and $z$, $\Gamma_z$ is the stabiliser group of $z$, $\Gamma_u$ is the isotropy group of points in $u$ (which is constant along $u$), and $\epsilon(u) \in \{\pm 1\}$ is the orientation sign of the trajectory. It is a theorem of \cite{Cho-Hong} that, with the given weights, $\partial^2 = 0$.

\begin{remark} \label{rmk: weights work}
    Given two trajectories $u,v$ meeting at a point $y$, the set of smoothings is indexed by the double coset space $ \mathrm{stab}(u) \backslash \Gamma_y / \mathrm{stab}(v)$. The orientation sign of a trajectory arising from gluing $u$ and $v$ depends only on the pair $\{u,v\}$ and not the particular smoothing. If $u$ has output $y$ and $v$ has output $z$, then
    \[
   \frac{ |\Gamma_y| \cdot |\Gamma_z| }{|\mathrm{stab}(u)|\cdot |\mathrm{stab}(v)| } \ = \ \sum \frac{|\Gamma_z|}{|\mathrm{stab}(u\# v)|}
    \]
    where we sum over the set of smoothings. This is why $\partial^2=0$ when the weights (asymmetrically) take account of the stabiliser group at the output.
\end{remark}

\begin{remark}\label{r:global}
    Let $\Crit(f)^+ \subset \Crit(f)$ denote the subset of oriented critical points. This carries an action of $\Gamma$.  The subcomplex
\[
C^*_{Morse}(X;f)^{\Gamma,+} := (\oplus_{x \in \Crit(f)^+} \mathbb{K}\langle x\rangle)^\Gamma
\]
is quasi-isomorphic to $C^*_{Morse}(Y;\bar{f})$: a sum of critical points is $\Gamma$-invariant only when they are individually oriented (see \cite{Cho-Hong}).
Indeed, if we denote the sum of the critical points over $y$ and $z$ by $\tilde{y}$ and $\tilde{z}$ respectively, then the weight $|\Gamma_z| / |\Gamma_{u}|$ can be simply understood as the total contribution from lifts of $u$ to the coefficient of $\tilde{z}$ in $\partial(\tilde{y})$.
\end{remark}

\begin{remark}\label{r:smoothings}
    Under the Morse-Floer isomorphism between the orbifold Floer theory of a $C^2$-small Morse function and orbifold Morse theory, the
    set of possible smoothings in Remark \ref{rmk: weights work} corresponds to the choices of identifying asymptotic markers on Floer cylinders, cf. Section \ref{s:orbifoldFloer}. 
\end{remark}

Recall that $\|\Gamma\|$ is the set of conjugacy classes in $\Gamma$. For $(g) \in \|\Gamma\|$, we choose a representative $g$ in the conjugacy class.
We have the fixed locus $X^{g} \subset X$ which carries an action of the centraliser group $C(g)$. The map $f$ induces a $C(g)$-invariant Morse function on $X^g$ for each $g$. We define the \emph{orbifold Morse complex}
\begin{align}\label{eq:orbiMorsecomplex}
C^*_{orb}(Y;f) :=  \oplus_{(g)\in \|\Gamma\|} C^*_{Morse}([X^g/C(g)];\bar{f})
\end{align}
The resulting cohomology is by definition the orbifold Morse cohomology, and computes the cohomology of the inertia stack of $Y$.

There is an orbifold Morse product, which can mix up the different inertia sectors (unlike the differential). The product  counts $Y$-shaped trajectories of gradient flow lines asymptotic to oriented critical points.  Each such should also be a weighted count, with
\[
\mu^2(x,y) = \sum_z n(x,y;z) z; \qquad n(x,y;z) = \sum_{u} \epsilon(u) |\Gamma_z| / |\mathrm{stab}(u)|
\]
where we sum over all the $Y$-configurations with given inputs and outputs, and again the weight depends only on the output. The same argument as in Remark \ref{rmk: weights work} then shows that this induces a chain map.

\section{Revisiting the quantum product using global charts}\label{s:globalchartclosed}

\subsection{Global Kuranishi charts}

\emph{In this section we define `global Kuranishi charts without boundary' and the virtual class of such a chart.}

Let $M$ be a compact metric  orbispace (see \cite[Section 3]{Pardon22},  \cite{HenGep07} and \cite{Moe02} for more discussion on orbispaces) where the isotropy group of every point is finite. A (topological and oriented) \emph{global Kuranishi chart} (without boundary) for $M$ is a quintuple $\bT=(G,\scrT,E,s,\phi)$ where $G$ is a compact Lie group, $\scrT$ is an oriented topological $G$-manifold (possibly with boundary) on which $G$ acts preserving the orientation and with finite stabilisers, $E \to \scrT$ is an oriented $G$-bundle, $s: \scrT \to E$ is a $G$-equivariant section with compact zero-set (contained in $\scrT \backslash \partial \scrT$ in the case when $\partial \scrT \neq \emptyset$), and $\phi: s^{-1}(0)/G \to M$ is  an isomorphism of orbispaces.  We will often suppress $\phi$ from the notation, since in the applications to spaces of holomorphic curves it will be `the identity'. 

The \emph{virtual dimension} of a global chart is $\dim(\scrT) - \dim(G) - \rk(E)$.   A global chart expresses $M$ as a closed subset of the orbifold $[\scrT/G]$, globally cut out as a section of the orbibundle $[E/G] \to [\scrT/G]$. We will call $\scrT$ the \emph{thickening} and $E$ the \emph{obstruction bundle}.

\begin{operation}\label{o:stablization}
There are various operations on global charts:
\begin{enumerate}
\item germ equivalence: replace $(G,\scrT,E,s)$ by $(G,\scrU,E|_\scrU, s|_{\scrU})$ for a $G$-invariant open subset $s^{-1}(0) \subset \scrU \subset \scrT$;
\item bundle stabilisation: replace $(G,\scrT,E,s)$ by 
$(G,Tot(F), p^*(E \oplus F), s \oplus \Delta)$
 where
 $F$ is a $G$-equivariant bundle over $\scrT$, $Tot(F)$ is the total space of $F$,
 $p: F \to \scrT$ is the projection,  and $s \oplus \Delta(x,f):=((x,f), s(x),f)$ for $x \in \scrT$ and $f \in F_x$, is the tautological diagonal section;
\item group induction: replace $(G,\scrT,E,s)$ by $(G', \scrT \times_G G', p^*E,p^*s)$ where $G \to G'$ is an inclusion of a Lie subgroup and $p: \scrT \times_G G' \to \scrT$ is projection;
\item free quotient: replace $(G,\scrT,E,s)$ by $(G/H, \scrT/H, E/H, s)$ where $H\subset G$ is a normal subgroup acting freely on $\scrT$ and $\scrT/H \to \scrT/G$ is viewed as a $G/H$-bundle.
\end{enumerate}
\end{operation}

\begin{remark} \label{rmk:essential_operations}
It is important to think of the underlying orbifold $[\scrT/G]$ and orbibundle $[E/G] \to [\scrT/G]$; from that viewpoint, bundle stabilisation adds a further orbibundle summand without changing the vanishing locus of the section, whilst induction and free quotients do not change the orbifold at all, but only its presentation as a global quotient. 
\end{remark}

    In general, an inclusion of topological manifolds need not have any analogue of a normal bundle. This indicates that the embeddings arising from changing the presentation of a global chart are rather special. We will incorporate this into a definition of embeddings of global charts which is adapted to our purposes, and is directly modelled on the geometry of the stabilisation operation introduced above.

\begin{definition}\label{d:embedding}
Given two global charts $\bT=(G,\scrT,E,s)$ and $\bT'=(G',\scrT',E',s')$ with the same virtual dimension, an \emph{embedding} from
$\bT$ to $\bT'$ consists of
\begin{enumerate}
    \item a $G$-invariant open subset $s^{-1}(0) \subset \scrU \subset \scrT$, and a $G'$-invariant open subset $(s')^{-1}(0) \subset \scrU' \subset \scrT'$;
    \item an orbifold vector bundle $F \to \scrU/G$ and an embedding $\scrU'/G' \to \mathrm{Tot}(F)$ to an open neighbourhood of the zero-section;
    \item an isomorphism of orbibundles $E'/G'  \to E/G \oplus F$ over $\scrU'/G'$ under which $s'$ is identified with $s \oplus \Delta$, where $\Delta$ is the canonical diagonal section of $p^*F \to \mathrm{Tot}(F)$.
\end{enumerate}
\end{definition}

\begin{remark}\label{rmk:not_global_quotient}
    Note that the `normal bundle' $F$ is an orbibundle which need not arise as a global quotient. This is necessary to allow one to compose embeddings.
\end{remark}

Two embeddings $\bT \to \bT'$ are isomorphic if all the data agrees after restricting to smaller open neighbourhoods of $s^{-1}(0)/G$ and $(s')^{-1}(0)/G'$ and perhaps composing with isomorphisms of the orbibundles. 
The composition of embeddings is well-defined up to isomorphism. We will also call embeddings of global charts stabilisation maps.  We emphasise that a stabilisation includes the data of an isomorphism of orbibundles in (3) above, and not just the existence of such.




    

\begin{definition}\label{d:itStab}
An iterative stabilisation of a global chart $\bT$ (with stabilising bundle $F$) is a global chart $\bT'$ with an embedding $\bT \to \bT'$ with normal bundle $F$. Two  global charts are equivalent if they admit a common iterative stabilisation.
    \end{definition}


\begin{remark}
An iterative stabilisation may have many steps. An embedding only remembers the outcome after composing the various bundle stablizations and germ restrictions, in a sufficiently small open neighborhood of the zero locus.    
\end{remark}

\begin{remark}\label{rmk:itStab}
    The requirement that $s'=s \oplus \Delta$ will be used when we build homotopy coherent perturbations in a later section (see Lemma \ref{l:pertubationexist}). In particular, a choice of perturbation of  $s$, and the condition $s' = s \oplus \Delta$, will canonically determine a perturbation of the extension $s'$. Note that writing down $\Delta$ relies on the identification of the stabilised chart $\scrU'/G'$ with a neighbourhood in an orbibundle over $\scrU/G$.
\end{remark}

\begin{remark}
    An alternative approach would be to define a class of `tailored' orbifold embeddings, which are induced from locally flat embeddings of topological manifolds which admit normal microbundles (in general a locally flat embedding only admits a normal microbundle after stabilisation of the codomain by $\bR^s$ for some $s>0$). 
\end{remark}

Fix a global chart $\bT:=(G,\scrT,E,s)$ for a compact space $M$ and let $Y = [\scrT/G]$. Let $r=\rk(E)$ and $\vd$ denote the virtual dimension.  There is a natural composition
\begin{equation} \label{eqn:vfc}
H^i(Y) \to H^{i+r}(E,E^{\sharp}) \to H^{i+r}(Y, Y\backslash s^{-1}(0)) \to H^{i+r}_{ct}(Y) \to H_{\vd-i}(Y) \to \bk.
\end{equation}
where the maps are respectively Thom isomorphism, pullback by the section $s$, the natural map since $s^{-1}(0)$ is compact (since $G$ and $M$ are), Poincar\'e duality in $H^*(\bullet;\bk)$ for the topological orbifold $Y$, and push-forward to a point $\bk = H_0(pt;\bk)$. The maps in \eqref{eqn:vfc} are all defined at chain level, given a cocycle representative for the Thom class.  Following Pardon \cite{Pardon}, one then defines:

\begin{definition}\label{defn:vfc}
 The \emph{virtual fundamental class} $[\bT]$ of the chart is the image of $1 \in H^0(Y;\bk)$ in $H_{\vd}(Y;\bk)$.  \end{definition}

When the virtual dimension is zero, we will usually just consider the image of the virtual class under the push-forward to the point. The virtual class plays the role of the non-existent fundamental class of $M$, viewed as an equivariant Euler class $e_G(E) \in H^*_{G,ct}(\scrT;\bk) = H^*_{ct}(Y;\bk)$.  More precisely, the virtual fundamental class is the Poincar\'e dual of $e_G(E)$.

\begin{lemma}
The virtual class is preserved (in an obvious sense) under the operations on global charts listed previously. 
\end{lemma}

\begin{proof} Follows from elementary properties of the Euler class, cf. \cite{AMS1}. \end{proof}

\subsection{Global charts  for moduli spaces of stable maps}

\emph{In this section we give global charts  for moduli spaces of stable maps to global quotient orbifolds.}

We now discuss global charts  for orbifold quantum cohomology of a global quotient  $Y = [X/\Gamma]$, where $\Gamma$ is a finite group acting faithfully by symplectomorphisms on a closed symplectic manifold $(X,\omega_X)$.  The construction here is adapted from \cite{AMS1, AMS2}, whose notation we follow for comparison purposes, and is similar in spirit to the final section of \cite{AGOT}. Whereas in \cite{AMS1,AMS2} the spaces of domains underlying the global charts are spaces of stable maps to projective space, we will now be considering spaces of orbifold stable maps to orbifolds $[\bP(V)/\Gamma]$, where $V$ is a faithful $\Gamma$ representation and  $\bP(V)$ is its projectisation.

We first give a brief overview of the appearance of $[\bP(V)/\Gamma]$.
We are interested in stable representable orbifold maps from a genus $0$ orbifold curve $\scrC$ (with coarse space $C$) with $h$ marked points to $Y$. Fix a pure Hurwitz orbit $\bfm \in \cH_h$ (cf. Equation \eqref{eq:closedmodulidecom2}).
It determines a $\Gamma$-twisted cover $\Sigma \to C$.
Let $\epsilon$ be a $\Gamma$-equivariant homotopy class of maps from $\Sigma $ to $ X$ (i.e. an isomorphism class of $\Gamma$-equivariant maps from $\Sigma$ to $X$ up to $\Gamma$-equivariant homotopy). 
We will build a global chart for the moduli space $\cM_{0,\bfm}(Y,\epsilon)$ of stable $\Gamma$-equivariant $J_X$-holomorphic maps $u:\Sigma \to X$ with Hurwitz orbit $\bfm$ and in the class $\epsilon$.
To do this, we will fix a Hermitian line bundle $L$ on $X$ once and for all and a positive integer $k>3$. Given $u:\Sigma \to X$, we consider a Hermitian line bundle 
$L_u=(u^*L^{\otimes 3} \otimes \omega_{log})^{\otimes k}$ where $\omega_{log}$ is the log-canonical line bundle of $\Sigma$.
We can show that $H^1(L_u)=0$
so the isomorphism type of $H^0(L_u)$ as a $\Gamma$ representation only depends on $\epsilon$ but not on a specific $u$.
We will take $V$ to be the $\Gamma$-representation such that $V^*$ is isomorphic to $H^0(L_u)$ as $\Gamma$-representations.
The global chart $(G,\scrT,E,s)$ for $\cM_{0,\bfm}(Y,\epsilon)$ will be defined such that $\scrT$ is a space of maps whose domains are elements in a space of orbifold stable maps to the orbifold $[\bP(V)/\Gamma]$.
In other words, compared with the construction in \cite{AMS1,AMS2}, $\epsilon$ now plays the role of the homology class, whilst $V$ plays the role of the degree.

For each isomorphism class of finite dimensional faithful $\Gamma$-representations, we fix once and for all a $\Gamma$-representation representing the isomorphism class. Denote this collection of representations by $\Theta$.

\subsubsection{Moduli of framed curves}\label{ss:framecurves}

Fix a $\Gamma$-equivariant $\omega_X$-tamed almost complex structure $J=J_X$ on $X$.

For each faithful $\Gamma$ representation $V$ in $\Theta$, the diagonal scaling commutes with $\Gamma$ so we can consider the orbifold
$[\bP(V)/\Gamma]$.
The centralizer $Z_{\Gamma}(V)$ of $\Gamma$ in $GL(V)$ descends to an action on $[\bP(V)/\Gamma]$.

\begin{remark}\label{r:normalizer}
    The normalizer of $\Gamma$ in $GL(V)$ also acts on $[\bP(V)/\Gamma]$; however, it is the action of the centraliser that will be relevant for us, cf. Proposition \ref{prop:chart for closed curves}.
\end{remark}

Fix a pure Hurwitz orbit $\bfm \in \cH_h$. For each (orbifold curve class) $\beta \in H_2([\bP(V)/\Gamma], \mathbb{Q})$\footnote{ We use $\epsilon$ to denote an equivariant homotopy class and $\beta$ is reserved for a homology class.}, we have a stack $\mathcal{K}_{0,\bfm}([\bP(V)/\Gamma],\beta)$ (see Section \ref{s:algebraicstack} and Equation \eqref{eq:closedmodulidecom2}).
Each $(v:\scrC \to [\bP(V)/\Gamma]) \in \mathcal{K}_{0,\bfm}([\bP(V)/\Gamma],\beta)$ comes with an equivariant map $u:\Sigma \to \bP(V)$ and the associated virtual $\Gamma$-representation $I(u):=H^0(u^*\mathcal{O}(1))-H^1(u^*\mathcal{O}(1))$.

\begin{lemma}\label{l:constantrep}
    Let $vRep(\Gamma)$ be the set of finite dimensional virtual representations of $\Gamma$ up to isomorphisms.
    The map $\mathcal{K}_{0,\bfm}([\bP(V)/\Gamma],\beta) \to vRep(\Gamma)$ given by $u \mapsto I(u)$ is  locally constant.
\end{lemma}

\begin{proof}
This follows from  equivariant index theory \cite{Mat71} (see also \cite[Appendix]{Atiyah},  \cite{Kuiper}, \cite{Janich}) and discreteness of finite dimensional $\Gamma$-representations.
\end{proof}

Therefore, by Lemma \ref{l:constantrep}, we can decompose the stack $\mathcal{K}_{0,\bfm}([\bP(V)/\Gamma],\beta)$ further into
\[
\mathcal{K}_{0,\bfm}([\bP(V)/\Gamma],\beta)=\sqcup_{R \in vRep(\Gamma)} \mathcal{K}_{0,\bfm}([\bP(V)/\Gamma],\beta,R)
\]
In fact, most $\mathcal{K}_{0,\bfm}([\bP(V)/\Gamma],\beta,R)$ are empty because $R$ constrains $\beta$.
We write 
\[
\mathcal{K}_{0,\bfm}([\bP(V)/\Gamma],R):= \cup_{\beta \in H_2([\bP(V)/\Gamma], \mathbb{Q})} \mathcal{K}_{0,\bfm}([\bP(V)/\Gamma],\beta,R).
\]

 Let $V^*$ be the dual of $V$ as  a $\Gamma$-representation.
We consider the stack  $\mathcal{K}_{0,\bfm}([\bP(V)/\Gamma],V^*)$.  
As a summary, it consists of maps $(u:\Sigma \to \bP(V))$  such that
\begin{enumerate}
\item $\Sigma$ is a  not necessarily connected nodal curve with a faithful $\Gamma$-action;
\item $\Sigma/\Gamma =C$ is a connected genus $0$ pre-stable curve with $h$ marked points, which includes the ramification points;
\item the $\Gamma$-action is balanced at the nodes of $\Sigma$;
\item there is a $\Gamma$-equivariant holomorphic map $u: \Sigma \to \bP(V)$ with $H^0(u^*\mathcal{O}(1))-H^1(u^*\mathcal{O}(1)) \simeq V^*$;
\item the induced map $v: [\Sigma / \Gamma] = \scrC \to [\bP(V) / \Gamma]$ is stable.
\end{enumerate}
The stack quotient $[\Sigma/\Gamma] = \scrC$ of $\Sigma$ by $\Gamma$ is a balanced twisted curve in the sense of \cite{AV,ACV}, and $C$ is its underlying coarse space.
If $u':\Sigma' \to [\bP(V)/\Gamma]$ is another such $\Gamma$-equivariant map, we say that $u'$ is isomorphic to $u$ if there is an isomorphism of stacks $\phi:[\Sigma/\Gamma] \to [\Sigma'/\Gamma]$ such that $v' \circ \phi =v$.
Automorphisms of $v:\scrC \to [\bP(V) / \Gamma]$ are defined in the same way as in Remark \ref{rmk:automorphism_of_cover}.
Such stacks are classically constructed e.g. in \cite{AV,ACV}, see also  \cite{Costello}.

We let $\scrF:=\scrF_{0,\bfm}(V^*) \subset \mathcal{K}_{0,\bfm}(\bP(V)/\Gamma,V^*)$ denote the open locus where
\begin{align}\label{eq:condition}
\begin{cases}
& u:\Sigma \to \bP(V) \text{ is an embedding}\\
& H^1(u^*T\bP(V))^{\Gamma} = 0 \\
& \text{the automorphism group of the map $v$ is trivial.}
\end{cases}
\end{align}

The first condition is equivalent to $u^*\mathcal{O}(1)$ being very ample on $\Sigma$, which is in turn equivalent to $H^1(\Sigma, u^*\mathcal{O}(1)\otimes I_Z)=0$ for any length $2$ subscheme $Z$ of $\Sigma$, where $I_Z$ is the ideal sheaf of $Z$.

\begin{lemma}\label{l:modulidomain}
The subspace $\scrF_{0,\bfm}(V^*)$ is smooth.
\end{lemma}

\begin{proof}
It is proved in \cite[Theorem 3.0.2, Section 3.0.3]{ACV} that the deformation of $[\Sigma/\Gamma] \to B\Gamma$ is unobstructed.
Since $H^1(u^*T\bP(V))^{\Gamma}=H^1(v^*T\bP(V)/\Gamma)$, the condition $H^1(u^*T\bP(V))^{\Gamma} = 0$ above implies that the deformation of $[\Sigma/\Gamma] \to [\bP(V)/\Gamma]$ is also unobstructed for $u \in \scrF_{0,\bfm}(V^*)$, so $\scrF_{0,\bfm}(V^*)$ is a smooth Deligne-Mumford stack and hence an orbifold.
The condition which requires the automorphism of $v$ to be trivial ensures that $\scrF_{0,\bfm}(V^*)$ is actually a manifold.
\end{proof}

\begin{remark}\label{r:isotropy}
 Since $u:\Sigma \to \bP(V)$ is an equivariant embedding, the isotropy group of a point $z \in \Sigma$ is the same as the isotropy group of $u(z) \in \bP(V)$ because $u$ sends the $\Gamma$-orbit of $z$ equivariantly isomorphically to the  $\Gamma$-orbit of $u(z)$.
 It implies that the isotropy group of $u(z)$ is cyclic because the isotropy group of $z$ is the isotropy group of the corresponding point on $\scrC$, which is cyclic.

 If the $\Gamma$ action on $\bP(V)$ is not faithful, then there is no equivariant embedding $u:\Sigma \to \bP(V)$ so we only need to consider faithful $\Gamma$ representations.
\end{remark}

For a holomorphic orbifold $Y$, the boundary of the stack $\mathcal{K}_{g,h}(Y,\beta)$ is by definition the locus of twisted stable maps with domains containing at least one node.  

\begin{lemma}\label{lem:normal crossing}
    The boundary of $\scrF_{0,\bfm}(V^*)$ is a normal crossing divisor.
\end{lemma}

\begin{proof}
     The following argument, which we learned from Dhruv Ranganathan, applies much more generally, to the corresponding locus in $\mathcal{K}_{0,h}(Y,\beta)$ for any complex algebraic orbifold $Y$. We introduce some characters:
     \begin{itemize}
         \item $\mathcal{M}_{0,h}$ is the Deligne-Mumford space of curves, which (in genus zero) is smooth and has normal crossing boundary;
         \item $\mathfrak{M}_{0,h}$ is the Artin stack of prestable curves;
         \item $\mathfrak{M}^{tw}_{0,h}$ is the stack of twisted prestable curves;
         \item $\mathfrak{M}^{tw, \diamond}_{0,h}$ is the corresponding stack of curves decorated by curve classes in the target orbifold summing to the given class $\beta$.
     \end{itemize}
     (The decorations $\diamond$ give the analogue of the stacks $\mathfrak{M}(\tau)$ of prestable curves labelled by a modular graph in \cite{Behrend}.) 
     There are natural forgetful maps
     \[
     \mathcal{K}_{0,h}(Y,\beta) \to \mathfrak{M}^{tw, \diamond}_{0,h} \to \mathfrak{M}^{tw}_{0,h} \to \mathfrak{M}_{0,h} \to \mathcal{M}_{0,h}.
     \]
By definition, a normal crossing divisor on a smooth Deligne-Mumford stack is a divisor which has normal crossings seen in an atlas, i.e. when pulled back by any local \'etale map from a scheme.  Note that all the maps in the sequence above satisfy that the preimage of the boundary is the boundary.  The map from prestable curves to stable curves is an iterated prestable curve fibration, hence it is log smooth in the sense of \cite[Definition 3.3.3]{CLS}, which means that it pulls back the boundary (which is the normal crossing divisor defining the log structure) to a normal crossing divisor.

For any fixed Hurwitz orbit $\bfm$, there is an open and closed substack $\mathfrak{M}^{tw}_{0,\bfm} \subset \mathfrak{M}_{0,h}^{tw}$ and the map $\mathfrak{M}^{tw}_{0,h} \to \mathfrak{M}_{0,h}$, restricted to such a component, is given by a root stack construction, cf. \cite[Remark 1.10]{Olsson} and \cite[Remark 11.7]{Abramovich_etal}. It is standard that taking root stacks along a normal crossing divisor yields a normal crossing divisor upstairs, see \cite[Section 3]{Bergh-Rydh} for instance. Adding the discrete decoration data to pass to $\mathfrak{M}^{tw,\diamond}_{0,h}$ is passing to an  \'etale cover, so also preserves the normal crossings property (and now the domain is separated).

By definition, the map 
     \[
     \mathcal{K}_{0,h}(Y,\beta) \to \mathfrak{M}^{tw, \diamond}_{0,h}
     \]
     is smooth as a map of stacks on the locus where $H^1(v^*TY)=0$, and the preimage of a normal crossing divisor under a smooth map again has normal crossings. Finally, passing to the automorphism-free locus $\scrF_{0,\bfm}(V^*)$ yields a scheme, so one can check the normal crossing property on the open set itself rather than on some \'etale cover.
\end{proof}

\begin{remark}\label{r:highergenus}
    The same argument applies in higher genus, except in that case Deligne-Mumford space $\mathcal{M}_{g,n}$ is a smooth Deligne-Mumford stack whose underlying coarse space has boundary which is locally the \emph{quotient} of a normal crossing divisor.  
    \end{remark}

We also have the space 
\begin{align}\label{eq:double}
\scrF' \subset \mathcal{K}_{0,\bfm}((\bP(V) \times \bP(V))/\Gamma,(V^*,V^*))
\end{align}
of $\Gamma$-equivariant stable maps $u: \Sigma \to \bP(V) \times \bP(V)$, both of whose projections yield elements of $\scrF:=\scrF_{0,\bfm}(V^*)$.
Let $\scrF'$ be the underlying manifold; there is then a natural diagonal $\Delta: \scrF \to \scrF'$. 
The normal bundle of the diagonal at a point $\Delta(u)$ is given by $H^0(u^*T\bP(V))^{\Gamma}$.
Therefore, a neighbourhood $\scrU_{\Delta}$ of the diagonal $\Delta(\scrF) \subset \scrF'$ is modelled over the total space of the vector bundle $H \to \scrF$ with fibre $H^0(u^*T\bP(V))^{\Gamma}$. 
Fix such an identification 
\begin{align}\label{eq:Phi}
\Phi: Tot(H) \to \scrU_\Delta.     
\end{align}

Let $\scrC$ be the universal domain curve\footnote{See \cite[Section 4.1]{AGOT} for some discussion of the universal domain curve.} over $\scrF$. So for $\phi \in \scrF$, its fibre  is of the form $\scrC|_{\phi} = \Sigma_{\phi}$ with a balanced $\Gamma$ action. Let $\Sigma_{\phi}^{\circ}$ denote the complement of the marked points and nodes in that curve. We let $\scrC^\circ \to \scrF$ be the complement in $\scrC$ of the nodes and marked points.
Note that the action on $\bP(V)$ by $Z_{\Gamma}(V)$ induces an action on $\scrF$ by postcomposition.
The postcomposition extends to an action on $\scrC$ equivariantly over  $\scrF$  because on one hand, it extends locally, and on the other, elements in $\scrF$ have trivial automorphisms so there is a unique way to extend it and this is automatically consistent globally.
The $\Gamma$ action and $Z_{\Gamma}(V)$ action on $\scrC$ commute and form a $Z_{\Gamma}(V) \times \Gamma$ action on $\scrC$.

\subsubsection{A polarizing line bundle and another $\Gamma$ representation}

We consider smooth $\Gamma$-equivariant maps $u:  \scrC|_{\phi} \to X$.  Now suppose that $L\to X$ is a complex line bundle
 with Hermitian connection $\nabla$ with curvature $-2i\pi\Omega$, for a $\Gamma$-invariant integral $2$-form $\Omega$ such that $\Omega(\cdot, J_X \cdot)$ is non-negative.
If $v: [\Sigma/\Gamma] \to Y$ is an orbifold stable map, with admissible cover $u: \Sigma \to X$, then $u^*L$ carries a $\Gamma$-action, so $H^0(u^*L)$ and $H^1(u^*L)$ are $\Gamma$-representations. Similar to Lemma \ref{l:constantrep}, we have

\begin{lemma}\label{l:equivariantInd}
    Suppose that $u_t:\Sigma_t \to X$, for $t \in [0,1]$, is a family of $\Gamma$-equivariant maps from closed Riemann surfaces.
    Then $H^0(u_t^*L)-H^1(u_t^*L)$ as a virtual $\Gamma$-representation is independent of $t \in [0,1]$. 
\end{lemma}



Denote the relative log canonical line bundle of the universal curve $\scrC$ over $\scrF$ by $\omega_{\scrC/\scrF}$.

\begin{lemma}\label{l:regular}
  
Let $\phi \in \scrF$ and $u:\scrC|_{\phi} \to X$ be a $\Gamma$-equivariant $J_X$-holomorphic map.
  For $k \gg 0$ and  $L_u := (\omega_{\scrC/\scrF} \otimes u^*L^{\otimes 3})^{\otimes k}$, we have $H^1(L_u)=0$. 
\end{lemma}
 
\begin{proof}
    Using Remark \ref{rmk:log_can_is_ample}, we see that $L_u$ has strictly positive degree on every component of $\scrC|_{\phi}$. Indeed, the log canonical bundle can only have non-positive degree on unstable components, where the degree is $\geq -2$; such components are not contracted, so $u^*L$ has degree $\geq 1$ and the given tensor product has strictly positive degree. The result then follows, for large $k$, by Riemann-Roch.
\end{proof} 

Let $\epsilon$ be a homotopy class of  $\Gamma$-equivariant maps $\scrC_{\phi} \to X$ up to $\Gamma$-equivariant homotopy.

 \begin{corol}\label{c:representation}
  In Lemma \ref{l:regular}, the $\Gamma$-representation $H^0(L_u)$ is the same for
  any $\phi \in \scrF$ and any $\Gamma$-equivariant $J_X$-holomorphic map $u:\scrC|_{\phi} \to X$ in the class $\epsilon$.
 \end{corol}

  \begin{proof}
      By Lemma \ref{l:equivariantInd}, the virtual $\Gamma$-representation $[H^0(L_u)]-[H^1(L_u)]$ depends only on $\epsilon$. Lemma \ref{l:regular} says that  $H^1(L_u)=0$.
  \end{proof}

\subsubsection{Global chart}\label{ss:global-chart}

We have fixed the number of marked points $h$, the pure Hurwitz orbit $\bfm$ up to simultaneous conjugation and an equivariant line bundle $L$. We also fix an equivariant homotopy class $\epsilon$ and let $k >3$. Now we are going to construct a global chart for $\cM_{0,\bfm}(Y,\epsilon)$.

Let $V$ be the $\Gamma$ representation in the collection $\Theta$ such that $H^0(L_u) \simeq V^*$ as $\Gamma$ representations for some (and hence any) $\Gamma$-equivariant $u:\Sigma \to X$ in the class $\epsilon$ (Corollary \ref{c:representation}). 
For this $V$, we define $\scrF$ by Equation \eqref{eq:condition} and denote the universal domain curve by $\scrC$.
Then we pick
a $GL(V)$-invariant (in particular, $Z_{\Gamma}(V) \times \Gamma$-invariant) Hermitian metric  on $\omega_{\scrC/\scrF}$ cf.  \cite[Lemma 4.11]{AMS2}.

We need to introduce the following notion of finite-dimensional approximation schemes \footnote{There is a more general definition introduced after \cite[Definition 4.35]{AMS2} which we don't need.} to define obstruction bundles and obstruction sections later on.

\begin{definition}

Let $G$ be a compact Lie group, $N \to Q$ be a smooth $G$-equivariant vector bundle over a (not necessarily compact) smooth $G$-manifold $Q$, $K \subset Q$ a $G$-invariant closed subset and $C^{\infty}_K(N)$ be the space of compactly supported smooth sections of $N$ that are vanishing near $K$.
A finite-dimensional $G$-approximation scheme for $N$ is a sequence $\lambda_{\mu}: V_{\mu} \to C^{\infty}_K(N)$ of finite-dimensional $G$-representations, indexed by $\mu \in \bN$ and nested (meaning that $V_{\mu} \subset V_{\mu+1}$ and $\lambda_{\mu+1}|_{V_\mu} = \lambda_\mu$), with the property that $\cup_{\mu} \lambda_{\mu}(V_{\mu})$ is dense in $C^{\infty}_K(N)$ with respect to the $C^{\infty}_{loc}$-topology. 
\end{definition}
Such always exist \cite[Lemma 4.2]{AMS2}. It is often helpful to assume that the $V_{\mu}$ are equipped with invariant inner products, so one can make sense of the orthogonal complement of a subrepresentation.

Let $C^{\infty}_c(\omega_{\scrC/\scrF} \otimes_{\bC} TX)$ be the space of  smooth sections of the $Z_{\Gamma}(V) \times \Gamma$-bundle $\omega_{\scrC/\scrF} \otimes_{\bC} TX$ over $\scrC \times X$ that are vanishing near the nodes and marked points.

Let $U_{\Gamma}(V):=Z_{\Gamma}(V) \cap U(V)$ be the subgroup of unitary transformations in $Z_{\Gamma}(V)$.
We take a finite-dimensional approximation scheme  $\{W_{\mu}\}_{\mu \in \bN}$ for $C^{\infty}_c(\omega_{\scrC/\scrF} \otimes_{\bC} TX)$ by $U_{\Gamma}(V) \times \Gamma$ representations.
We furthermore assume that the  $\Gamma$-trivial-subrepresentations $V_\mu = W_{\mu}^{\Gamma}$ form a finite-dimensional approximation scheme of $C^{\infty}_c(\omega_{\scrC/\scrF} \otimes_{\bC} TX)^{\Gamma}$.

Now we can define a global Kuranishi chart $(G,\scrT,E,s)$ of $\cM_{0,\bfm}(Y,\epsilon)$. The group $G$ is $U_{\Gamma}(V)$.

Let $\{f_1,\dots,f_{rk(V)}\}$ be an ordered basis of $V^*$.
The thickening $\scrT$ comprises tuples $( \phi,u,e,F)$ where 
\begin{enumerate}
\item $\phi \in \scrF$, 
\item $u: \scrC|_{\phi} \to X$ is a smooth $\Gamma$-equivariant map representing the class $\epsilon$
\item  $e\in V_{\mu}$ (i.e. the $\Gamma$-trivial subrepresentation of $W_{\mu}$);
\item $F$ is an isomorphism $H^0(L_u) \to V^*$ as $\Gamma$-representations such that the matrix $H_F$ of inner products $(H_F)_{ij}:=\langle F^{-1}(f_i), F^{-1}(f_j) \rangle$ has positive eigenvalues
\end{enumerate}
and satisfying the further two conditions:
\begin{equation} \label{eqn:CR}
\cdbar_J(u)|_{\Sigma^\circ_{\phi}} + \lambda_{\mu}(e)(\iota_{\phi},u) = 0
\end{equation}
where $(\iota_\phi,u): \Sigma_{\phi}^{\circ} \to \scrC^\circ \times X$, and
\begin{equation} \label{eqn:diagonal_condition}
(\phi,\phi_F) \in \scrU_{\Delta} \subset \scrF'
\end{equation}
where $\phi_F: \Sigma_{\phi} \to \bP(V)$ is the map defined by the $\Gamma$-equivariant isomorphism $F$. In other words, we define $\phi_F(z):=[F^{-1}(f_1)|_z:  \dots : F^{-1}(f_{rk(V)})|_z] \in \bP(V)$.

The action of $G$ on $\scrT$ is given by $ (\phi, u,e,F)\cdot g:=(\phi \cdot g, u \cdot g, e \cdot g, F \cdot g)$. To spell it out, the action on $\phi$ comes from the action on $\scrF$. The map $(u \cdot g)(z)$ is defined by $u(z \cdot g^{-1})$, where for $z \in \Sigma_{\phi \cdot g}$, we have $z \cdot g^{-1} \in \Sigma_{\phi}$. The action on $e$ comes from the  $U_{\Gamma}(V)$ action on $V_{\mu}$. The action on $F$ comes from pre-composing the isomorphism
$H^0(L_{u \cdot g}) \to H^0(L_u)$.
The $G$ action on $\scrT$  has finite stabilisers (although the $\Gamma$-equivariant unitary action on $\scrF$ factors through a projective $\Gamma$-equivariant unitary action of $\bP U_{\Gamma}(V)$, the diagonal circle subgroup acts non-trivially on the framing data $F$).

The group of $\Gamma$-equivariant isomorphisms of $V$ is $Z_{\Gamma}(V) \subset GL(V)$.
Let $\mathfrak{z}$ be the Lie algebra of $Z_{\Gamma}(V) \subset GL(V)$ and $\scrH$ be the Lie subalgebra of $\mathfrak{z}$ consisting of Hermitian matrices.
Since we have chosen a $\Gamma$-invariant Hermitian metric, the  $\Gamma$ action on $V$ lies in $U(V)$. 
Therefore, $Z_{\Gamma}(V)$ is closed under taking Hermitian transpose and hence we have the following Cartan decomposition theorem:
\begin{prop}[\cite{BeyondLie}, Proposition 1.143]\label{p:polar}
    The map $U_{\Gamma}(V) \times \scrH \to Z_{\Gamma}(V)$ given by 
    \[
    (k,X) \mapsto k\exp(X)
    \]
    is a homeomorphism.
\end{prop}
Take a matrix $A \in Z_{\Gamma}(V)$ and regard it as an element in $GL(V)$.
The polar decomposition theorem for $GL(V)$ allows us to write $A=k_A \exp(h_A)$ for some $k_A \in U(V)$ and some Hermitian matrix $h_A$. Proposition \ref{p:polar} proves that $k_A \in U_{\Gamma}(V)$ and $h_A \in \scrH$.
Therefore, the matrix $H_F$ from the isomorphism $F$ above lies in $\exp(\scrH)$.

We will show that $\scrT$ (more precisely, an open subset of it) is a manifold in Lemma \ref{l:manifold}. Assuming it for now, we define the \emph{obstruction bundle} $E \to \scrT$ to be the vector bundle with fibres $H \oplus V_{\mu} \oplus \scrH$.  The section $s$ is defined by (recall the definition of $\Phi$ from \eqref{eq:Phi})
\[
(\phi,u,e,F) \mapsto (\Phi^{-1}(\phi,\phi_F), e, \exp^{-1}(H_F)).
\]
Where $s=0$, we have:
\begin{itemize}
\item  $e=0$ so \eqref{eqn:CR} is the usual Cauchy-Riemann equation, 
\item $\exp^{-1}(H_F) = 0$ implies that the isomorphism $F$ is unitary;
\item $\Phi^{-1}(\phi,\phi_F)=0$ shows that $\phi=\phi_F$ so the choice of $\phi$ is determined by the frame.
\end{itemize}
Taken together, this shows that  $s^{-1}(0)$ consists of elements of the form
\[
(\phi_F,u,0,F)
\]
where $u \in \cM_{0,h, \bfm}(Y, \epsilon)$, $F:H^0(L_u) \to V^*$ is a $\Gamma$-equivariant unitary isomorphism such that $\phi_F$ is an element of $\scrF$.

\begin{lemma}\label{l:manifold}
The space $\scrT$ is a topological manifold in a neighbourhood of $s^{-1}(0)$ for sufficiently large $k$ and $\mu$.
\end{lemma}

\begin{proof} Consider the same space $\scrT^{pre}$ without the condition that $u$ be $\Gamma$-equivariant and in which $e\in W_{\mu}$. This is a topological manifold for sufficiently large $k,\mu$  by \cite{AMS1,AMS2}. Now $\scrT \subset \scrT^{pre}$ is the fixed locus of a locally linear finite group action on a manifold $\scrT^{pre}$, so is again a manifold. \end{proof}

\begin{prop} \label{prop:chart for closed curves}
For sufficiently large $k$ and $\mu$ the data $\bT = (G,\scrT,E,s)$  above defines a global chart for the moduli space $\cM_{0,h, \bfm}(Y, \epsilon)$. In other words, we have a natural homeomorphism $s^{-1}(0)/G \simeq \cM_{0,h, \bfm}(Y, \epsilon)$ as orbispaces.
\end{prop}

\begin{proof} 

The result then follows from the arguments in \cite[Theorem 4.30]{AMS1} given the characterisation of orbifold stable maps as maps of admissible covers. We recall the argument here for the sake of completeness.

The degree of $L_u$ on every irreducible component of $\Sigma$ increases strictly when $k$ increases. Therefore, when $k$ is sufficiently large, we have $\phi_F \in \scrF$ for all $u \in \cM_{0,h, \bfm}(Y, \epsilon)$ and all $\Gamma$-equivariant unitary isomorphisms $F:H^0(L_u) \to V^*$.
Therefore the forgetful map $s^{-1}(0) \to \cM_{0,h, \bfm}(Y, \epsilon)$
\[
(\phi_F,u,0,F) \mapsto u
\]
is surjective.

Given $u \in \cM_{0,h, \bfm}(Y, \epsilon)$, we can construct an element in $s^{-1}(0)$ by choosing a $\Gamma$-equivariant unitary isomorphism $F$ and take $\phi$ to be $\phi_F$.
Two different choices of $F$ differ by a $\Gamma$-equivariant unitary isomorphism (i.e. an element in $G=U_{\Gamma}(V)$)\footnote{This is why we consider $Z_{\Gamma}(V)$ instead of the normalizer of $\Gamma$, cf Remark \ref{r:normalizer}. Indeed, we are interested in the group of automorphism of $V^*$ rather than the largest group acting on $\bP(V)/\Gamma$.}.
Therefore, $s^{-1}(0)/G \simeq \cM_{0,h, \bfm}(Y, \epsilon)$ as topological spaces.

To compare the isotropy groups, recall that an isomorphism of $u:\Sigma \to X$ in  $\cM_{0,h, \bfm}(Y, \epsilon)$ is a $\Gamma$-equivariant map $f: \Sigma \to \Sigma$ such that $u \circ f=u$.
Pulling back a $\Gamma$-equivariant isomorphism  $F:H^0(L_u) \to V^*$ along $f$ will give us an element $g \in U_{\Gamma}(V)$ such that $\phi_F \circ f=\phi_{F \cdot g}$. On the other hand, an isomorphism of $(\phi_F,u,0,F)$ is an isomorphism $f:\Sigma \to \Sigma$ of $u$ together with $g \in U_{\Gamma}(V)$ such that $\phi_F \circ f=g \circ \phi_{F}$.
Since $g \circ \phi_{F}=\phi_{F \cdot g}$, the isotropy groups in $s^{-1}(0)/G$ and in $ \cM_{0,h, \bfm}(Y, \epsilon)$ are canonically identified.
\end{proof}

 The construction of $\bT$ involves many choices, for example, the Hermitian line bundle $L$, finite dimensional approximation scheme, $k$, $\mu$, etc. 
We also use the presentation of $Y$ as $[X/\Gamma]$. 
Using the doubly framed trick \cite{AMS2}, one can show that the equivalence class (see Definition \ref{d:itStab}) of the global chart is independent of these choices.
As an illustration, we explain how to use the doubly framed trick to show that it is independent of the presentation of $Y$.

\begin{prop}\label{p:doubly-closed}
    The equivalence class of the global chart is independent of the choice of presentation of $Y$ as a global quotient.
\end{prop}

\begin{proof}
Suppose that $Y=[X/\Gamma]=[X'/\Gamma']$ where $X,X'$ are smooth manifolds and $\Gamma,\Gamma'$ are finite groups.
Any element $u \in \cM_{0,h, \bfm}(Y, \epsilon)$ defines a representable orbifold morphism $v:\scrC \to Y$.
By Theorem \ref{t:stablemap}, it gives the corresponding $\Gamma'$-equivariant map $u':\Sigma' \to X'$ and the associated pure Hurwitz type $\bfm'$ and equivariant homotopy class $\epsilon'$.
Homotoping $u$ equivariantly will induce a homotopy of $v$ and an equivariant homotopy of $u'$ so either of $(\bfm,\epsilon)$ and $(\bfm',\epsilon')$ determines the other. We want to compare the global charts of $\cM_{0,h, \bfm}(Y, \epsilon)=\cM_{0,h, \bfm'}(Y, \epsilon')$.

Let $L,L'$ be the respective equivariant Hermitian line bundles on $X$ and $X'$. As in Section \ref{ss:global-chart}, let $k >3$. 
We have $H^0(L_u) \simeq V^*$ and $H^0(L'_{u'}) \simeq (V')^*$, and the corresponding spaces of framed curves\footnote{The $\scrF'$ here is the space of maps to $\mathcal{K}_{0,\bfm'}([\bP(V')/\Gamma'],(V')^*)$ satisfying the analog of \eqref{eq:condition}, \emph{not} the space labelled $\scrF'$ introduced in the paragraph after Remark \ref{r:highergenus}; we hope this will cause no confusion! } $\scrF$ and $\scrF'$.
Let $V_\mu, W_{\mu}, V'_{\mu'}, W'_{\mu'}$ be the respective finite dimensional approximation schemes.
Let the associated global charts from Proposition \ref{prop:chart for closed curves} be $\bT$ and $\bT'$, respectively.

We consider the subspace $\scrF^{master} \subset \mathcal{K}_{0,(\bfm,\bfm')}([\bP(V)/\Gamma] \times [\bP(V')/\Gamma'],(V^*,(V')^*))$ consisting of elements whose projection to the first and second factors lie in $\scrF$ and $\scrF'$ respectively.

We now describe a `master' global chart $\bT^{master}$, which will be a stabilisation of both $\bT$ and $\bT'$.
Let $G=U_{\Gamma}(V)  \times U_{\Gamma'}(V')$.
The thickening space consists of $(\phi, v, e,e',F,F')$ such that
\begin{itemize}
    \item $\phi \in \scrF^{master}$ (denote the domain curve and its associated $\Gamma$ and $\Gamma'$ admissible covers by $\scrC_{\phi}$, $\Sigma_{\phi}$ and $\Sigma'_{\phi}$ respectively),
    \item $e \in V_{\mu}$ and $e' \in V'_{\mu'}$,
    \item $v:\scrC_{\phi} \to Y$ is a representable orbifold morphism descending from $u:\Sigma_{\phi} \to X$ (in the class $\epsilon$) and $u':\Sigma'_{\phi} \to X'$  (in the class $\epsilon'$) satisfying the perturbed holomorphic equation (i.e. $(dv)^{0,1}_{J/\Gamma}+\lambda_{\mu}(e)+\lambda'_{\mu'}(e')=0$),
    \item $F:H^0(L_u) \to V^*$ and $F':H^0(L'_{u'}) \to (V')^*$ are equivariant isomorphisms such that the associated matrices $H_F$, $H_{F'}$ have positive eigenvalues. Here $H_F$ and $H_{F'}$ are defined with respect to the fibrewise metric on the universal curves over $\scrF$ and $\scrF'$ (i.e. we are using the projection $\scrF^{master} \to \scrF, \scrF'$ here).
    They induce the associated equivariant maps $\phi_F: \Sigma_{\phi} \to \bP(V)$ and $\phi_{F'}: \Sigma'_{\phi} \to \bP(V')$,
    \item $\phi_F$ and $\phi_{F'}$ lie in neighborhoods of the respective diagonals so that $\Phi^{-1}(\phi_F,\phi)$ and $(\Phi')^{-1}(\phi_{F'},\phi)$ are well-defined (here we abuse notation and identify $\phi$ with its projection to $\scrF$ and $\scrF'$ respectively).
\end{itemize}
  The obstruction bundle has fibre $H \oplus V_{\mu} \oplus \scrH \oplus  H' \oplus V'_{\mu'} \oplus \scrH' $. Here $H$ and $H'$ are the fibres of the normal bundle of the diagonal inside the square of  $\scrF$, and  $\scrF'$, respectively. The vector spaces $\scrH$ and $\scrH'$ are subspaces of Hermitian matrices associated to $U_{\Gamma}(V)$ and $U_{\Gamma'}(V')$.
 The obstruction section is the obvious map.

We leave the readers to check that $\bT^{master}$ is a stabilisation of both $\bT$ and $\bT'$ (cf. \cite[Section 4.9]{AMS2}). 
\end{proof}

Using the evaluation maps from the master global chart $\bT^{master}$, one can show that the quantum multiplication \eqref{eqn:quantum_product_formula} is independent of choices. Proposition \ref{p:QHwelldefined} remains valid using the virtual fundamental class (Definition \ref{defn:vfc}) of the global Kuranishi chart in Proposition \ref{prop:chart for closed curves}.

Since our global chart construction agrees with the one in \cite{AMS2} when $\Gamma$ is trivial, another consequence of Proposition \ref{p:doubly-closed} is that if the $\Gamma$ action on $X$ is free, in which case $Y$ is a smooth manifold, then the resulting global chart in Proposition \ref{prop:chart for closed curves} is equivalent to the global chart constructed by \cite{AMS2}.

\section{Flow category generalities}\label{s:flowcat}

    This section explains our general mechanism for extracting a chain complex from a flow category $\cC$ where objects carry non-trivial automorphisms (see Definition \ref{d:chaincomplex} and Proposition \ref{p:dsquare=0}). In the sequel, we will associate an ordered marked flow category $\cC(H_Y)$ to a non-degenerate Hamiltonian function $H_Y: Y \times S^1 \to \bR$ such that  
     it is an invariant of the orbifold $Y = [X/\Gamma]$. 

\subsection{Ordered marked flow category}\label{s:flow}

\emph{We introduce ordered marked flow categories and their enrichment in global Kuranishi charts.
Since we consider broken cylinders with ordered interior marked points, the combinatorics of the corresponding stratifications of moduli spaces by broken curves will differ from the usual combinatorics of a flow category. We first set up the combinatorics and then discuss the enrichment in global Kuranishi charts.
}

We borrow the notation from \cite{BX}. Let $\bA$ be a poset with relation $\leq$, equipped with its canonical topology where a set $U$ is open whenever $\alpha \in U, \alpha \leq \beta \,  \Rightarrow \,  \beta \in U$.  Write $\partial^{\alpha} \bA = \{\beta \in \bA \, | \, \beta \leq \alpha\}$.  A topological orbispace $T$ is $\bA$-stratified if it is equipped with a continuous function $s:T \to \bA$ with finite range (more precisely, $s$ is a map from the coarse moduli space of $T$ to $\bA$); we write $T_{\alpha} := s^{-1}(\alpha)$ and $\partial^{\alpha}T := \sqcup_{\beta \leq \alpha} T_{\beta}$. Since $s$ is continuous, $\partial^{\alpha}T$ is a closed subset of $T$.

The input to our ordered marked flow category construction will comprise: 
\begin{enumerate}
\item a countable poset $\cP$, with an action $\cA:\cP \to \mathbb{R}$ and an integralised action $\cA_{\mathbb{Z}}: \cP \to \bZ$  which respects order, so $p <  q \Rightarrow \cA(p) < \cA(q) \text{ and } \cA_{\mathbb{Z}}(p) < \cA_{\mathbb{Z}}(q)$;
\item  
a free action of a countable group $\Pi$ on $\cP$ with $\cP / \Pi$ finite, and two maps 
$\cA_\Pi: \Pi \to \bR$, $\cA_{\Pi,\mathbb{Z}}: \Pi \to \bZ$   with
$\cA(a\cdot p) = \cA(p) - \cA_{\Pi}(a)$ and
$\cA_{\mathbb{Z}}(a\cdot p) = \cA_{\mathbb{Z}}(p) - \cA_{\Pi,\mathbb{Z}}(a)$ for $a \in \Pi$ and $p \in \cP$.
\end{enumerate}

\begin{remark}
The integralised action $\cA_{\bZ}$ will be used to run induction for the construction of a global chart lift of the flow category, while the action $\cA$ will be used to define the action filtration of the resulting chain complex. When we define the bulk-deformed orbifold Floer cohomology, $\cA_{\Pi,\mathbb{Z}}$ will be a multiple of $\cA_{\Pi}$.
\end{remark}

Recall that $\|\Gamma\|$ is the set of conjugacy classes in $\Gamma$ and 
 $\|\Gamma\|^{\circ}=\|\Gamma\| \setminus \{id\}$.
For $p,q \in \cP$, $k \in \mathbb{N}$ and $(g_1), \dots,(g_k) \in \|\Gamma\|^{\circ}$ we introduce a poset (which will not depend on the sequence of labels $\{(g_i)\}$)
\begin{align}
\bA_{pq}^{((g_1), \dots, (g_k))} = \left\{ (pr_1\cdots r_lq, h_1,\dots,h_k) \, | \, r_i \in \cP, \ p < r_1 <  \cdots < r_l < q, h_j \in \{0,\dots,l\}\right\}  \label{eq:poset}
\end{align}
with 
\[
(ps_1\ldots s_mq,h_{1,s},\dots,h_{k,s}) \leq (p r_1\ldots r_l q, h_{1,r}, \dots,h_{k,r})
\]
if and only if there is an injective monotonic increasing function $f:\{1,\dots,l\} \to \{1,\dots,m\}$ such that
\begin{enumerate}
\item $r_i=s_{f(i)}$ for all $i \in \{1,\dots,l\}$, and
\item $f(h_{j,r}) \le h_{j,s} \le f(h_{j,r}+1)$ for all $j \in \{1,\dots,k\}$, where $f(0):=0$ and $f(l+1):=m$.
\end{enumerate}
Note that $\bA_{pq}^{((g_1), \dots, (g_k))}=\bA_{pq}^{((g_1'), \dots,(g_k'))}$ for any $(g_j), (g_j') \in \|\Gamma\|^{\circ}$.
Let
\begin{align}\label{eq:Ak}
\bA_{pq}^{k} = \sqcup_{((g_1), \dots, (g_k)) \in (\|\Gamma\|^{\circ})^k} \bA_{pq}^{((g_1), \dots, (g_k))}.
\end{align}

\begin{remark}\label{r:compatibleH}
    In orbifold Hamiltonian Floer theory, $(pr_1\cdots r_lq, h_1,\dots,h_k) \in \bA_{pq}^{((g_1), \dots, (g_k))}$ corresponds to the stratum which consists of a possibly broken representable Floer cylinder  $v_i$ (possibly with some sphere bubbles) from the Hamiltonian orbit $r_i$ to $r_{i+1}$ for all $i$ ($r_0:=p$ and $r_{l+1}:=q$) such that the $j^{th}$-ordered marked point is an orbifold marked point labelled by conjugacy class $g_j$ (i.e. the evaluation map goes to the twisted sector $[X^{g_j}/C(g_j)]$) and is on the broken cylinder $v_{h_j}$ or the sphere bubbles attached to it.

    The partial ordering is defined such that $(ps_1\ldots s_mq,h_{1,s},\dots,h_{k,s}) \leq (p r_1\ldots r_l q, h_{1,r}, \dots,h_{k,r})$ if and only if the stratum associated to $(ps_1\ldots s_mq,h_{1,s},\dots,h_{k,s})$ is contained in the stratum associated to $(p r_1\ldots r_l q, h_{1,r}, \dots,h_{k,r})$.
\end{remark}

For simplicity, we will denote $(pr_1\cdots r_lq, h_1,\dots,h_k)$ by $(p\bfr q, \bfh)$. 
If we want to remember which $\bA_{pq}^{((g_1),\dots,(g_k))}=:\bA_{pq}^{\bfg}$ the element  $(p\bfr q, \bfh)$ lies in, we denote it by 
$(p\bfr q, \bfh, \bfg)$.
There is a maximal element $(pq,\mathbf{0})$ in each $\bA_{pq}^{\bfg}$ (we denote it by $pq$ for simplicity when it is clear).
The poset $\bA_{pq}^{\bfg}$ is homogeneous, so elements have a well-defined `depth'.
More precisely, the depth of $\alpha=(pr_1\cdots r_lq, h_1,\dots,h_k)$ is $\depth(\alpha)=l$ and $(pq,\mathbf{0})$ is the unique element with depth $0$.

For any $\alpha=(p\bfr q, \bfh) \in \bA_{pq}^\bfg$ of depth $l$ and $i \in \{0,\dots,l\}$, we define
\begin{align}\label{eq:k_i}
  k_i(\alpha):=\#   \{j| h_j=i\}, \text{ and } \bfg_i(\alpha)=((g_{j}): h_j=i)
\end{align}
Note that $\sum_{i=0}^l k_i(\alpha)=k$.

\begin{example}
    If $\bfh=(0,1,0,0,1)$ and $\bfg=((g_1),(g_2),(g_3),(g_4),(g_5))$, then $k_0(\alpha)=3$, $k_1(\alpha)=2$, $\bfg_0(\alpha)=((g_1),(g_3),(g_4))$ and $\bfg_1(\alpha)=((g_2),(g_5))$.
\end{example}

Let $\alpha=(prq, \bfh)$ be an element with depth $1$ and $\bfh=(h_1,\dots,h_k)$ with $h_j \in \{0,1\}$.
Consider its associated boundary $\partial^{\alpha} \bA_{pq}^\bfg$.
There is an isomorphism of homogeneous posets
\begin{align}\label{eq:posetface}
\bA_{pr}^{\bfg_0(\alpha)} \times \bA_{rq}^{\bfg_1(\alpha)} \cong \partial^{(prq,\bfh)} \bA_{pq}^{\bfg}
\end{align}
given by
\[
((p\bfr_0 r, \bfh_0), (r\bfr_1 q, \bfh_1)) \mapsto (p\bfr_0 r \bfr_1 q, \bfh')
\]
where if we write $\bfh_0=(h_{0,1}, \dots,h_{0,k_0(\alpha)})$, $\bfh_1=(h_{1,1}, \dots,h_{1,k_1(\alpha)})$ and $\bfh'=(h'_1,\dots,h'_k)$, then 
$h'_j=h_{0,j}$ if $h_j=0$ and $h'_j=h_{1,j}+|\bfr_0|+1$ if $h_j=1$.

\begin{example}
    If $\bfh=(0,1,0,0,1)$ and $\bfg=((g_1),(g_2),(g_3),(g_4),(g_5))$,  then under \eqref{eq:posetface} we have 
    \[
    ((pr, (0,0,0)), (rr_1q, (1,0))) \ \mapsto \ (prr_1q, (0,2,0,0,1)).
    \]
\end{example}

\begin{remark}\label{r:nchooser_A}
Notice that if $\beta=(prq, \tilde{\bfh})\in \bA_{pq}^\bfg$ satisfies $\bfg_i(\alpha)=\bfg_i(\beta)$ for $i=0,1$, then there is also an isomorphism from $\bA_{pr}^{\bfg_0} \times \bA_{rq}^{\bfg_1}$ to $\partial^{\beta} \bA_{pq}^\bfg$ as well.
Since there are $\frac{k!}{k_0!k_1!}$ many monotonic increasing injective functions from $\{1, \dots, k_0\}$ to $\{1, \dots, k=k_0+k_1\}$,  $\bA_{pr}^{\bfg_0} \times \bA_{rq}^{\bfg_1}$ appears in $\frac{k!}{k_0!k_1!}$ many depth one strata of $\bA_{pq}^{k}$ (see Equation \eqref{eq:Ak}).
A sanity check on the combinatorics: there are $(|\Gamma|-1)^k2^k$ many depth one strata in total in $\bA_{pq}^{k}$, among which $(|\Gamma|-1)^{k_0} (|\Gamma|-1)^{k_1}\frac{k!}{k_0!k_1!}$ are covered by $\bA_{pr}^{\bfg_0} \times \bA_{rq}^{\bfg_1}$ for some $\bfg_0 \in (\|\Gamma\|^\circ)^{k_0}$ and $\bfg_1 \in (\|\Gamma\|^\circ)^{k_1}$.
Indeed, $(|\Gamma|-1)^k2^k=\sum_{k_0+k_1=k} (|\Gamma|-1)^{k_0} (|\Gamma|-1)^{k_1}\frac{k!}{k_0!k_1!}$.
\end{remark}

Let $\bA_{pq}:=\cup_{k \in \mathbb{N}}\bA_{pq}^k$ viewed as a poset with no partial ordering relation between $\bA_{pq}^k$ and $\bA_{pq}^{k'}$ if $k \neq k'$.

As a generalization of \cite[Definition 3.7]{BX}, we introduce: 

\begin{definition}\label{d:omfc}
An {\bf ordered marked flow category} $\cC$ with object set $\cP$ consists of the following data:
\begin{enumerate}
\item A group $\Gamma_p$ for every object $p \in \cP$. We call $\Gamma_p$ the automorphism group of $p$ or isotropy group of $p$.
\item For each $p,q \in \cP$, the morphism space from $p$ to $q$ is a compact $\bA_{pq}$-orbispace $\cM_{pq}$  such that 
\[
\cM_{pq} \neq \emptyset \Rightarrow p< q \qquad \text{ so } \cM_{pp} = \emptyset.
\]
We write $\cM_{pq}$ as a disjoint union of compact $\bA_{pq}^k$-spaces $\cM_{pq}^k$ for $k \in \mathbb{N}$, and write $\cM_{pq}^k$ as a disjoint union of compact $\bA_{pq}^{\bfg}$-spaces $\cM_{pq}^\bfg$  for $\bfg \in (\|\Gamma\|^\circ)^k$.
We also impose the following properness condition on $\cA$ and $\cA_{\mathbb{Z}}$ (cf. Remark \ref{r:action_de} for our convention of differential later on):
\begin{align}\label{eq:action_proper}
    \text{ for every }r \in \mathbb{R} \text{ and }q \in \cP, \cup_{p, \cA(p)>r} \cM_{pq} \text{ and } \cup_{p, \cA_{\mathbb{Z}}(p)>r} \cM_{pq} \text{ are compact.}
\end{align}

\item A continuous map $\vdim:\cM_{pq} \to \mathbb{Z}$. Write $\cM_{pq}^k \cap \vdim^{-1}(m)$ and $\cM_{pq}^\bfg \cap \vdim^{-1}(m)$ as $\cM_{pq}^{k,m}$ and $\cM_{pq}^{\bfg,m}$ respectively, so $\cM_{pq}$ is a disjoint union of $\cM_{pq}^{k,m}$ (and also of $\cM_{pq}^{\bfg,m}$).  By compactness of $\cM_{pq}$, all but finitely many $\cM_{pq}^{k,m}$ are empty.

\item Representable orbispace morphisms (called evaluation maps) $\cM_{pq}^k \to [pt/\Gamma_p]$ and $\cM_{pq}^k \to [pt/\Gamma_q]$. In particular, this implies that the isotropy groups of elements in $\cM_{pq}^k$ are subgroups of both $\Gamma_p$ and $\Gamma_q$.

\item 
For any depth-one $\alpha=(prq,\bfh, \bfg)$, there is a boundary face map that is compatible with the $\bA$-stratification
\begin{align}\label{eq:codim1}
\xymatrix{
\cM_{pr}^{\bfg_0(\alpha)} \times_{[pt/\Gamma_r]} \cM_{rq}^{\bfg_1(\alpha)} \ar[rr]^{ \simeq} \ar[d] &  &\partial^{(prq,\bfh)}\cM_{pq}^\bfg \ar[d] \\
\bA_{pr}^{\bfg_0(\alpha)} \times \bA_{rq}^{\bfg_1(\alpha)} \ar[rr]^{\simeq} & & \partial^{(prq,\bfh)}\bA_{pq}^\bfg
}
\end{align}
where 
$\cM_{pr}^{\bfg_0(\alpha)} \times_{[pt/\Gamma_r]} \cM_{rq}^{\bfg_1(\alpha)}$ is the fibre product over the evaluation maps (see \cite[Section 2.3]{Moe02} for the definition of fibred product\footnote{Strictly speaking, \cite{Moe02} considers orbifolds rather than orbispaces and uses Lie groupoids rather than topological groupoids. Orbispaces are topological groupoids for which the source and target maps are surjective continuous maps of topological spaces. The definition of a fibre product carries over to this setting.}) and
the bottom arrow is \eqref{eq:posetface}.
Moreover, we require that 
\begin{align}\label{eq:vdimsum}
\vdim(u_0)+\vdim(u_1)=\vdim([u_0,u_1])-1
\end{align}
for any $u_0 \in \cM_{pr}^{\bfg_0(\alpha)}$ and $u_1 \in \cM_{rq}^{\bfg_1(\alpha)}$, where  $[u_0,u_1]$ is defined to be the image of $(u_0,u_1)$ under the top arrow. In other words, we have $\sqcup_{m_0+m_1=m-1} \cM_{pr}^{\bfg_0(\alpha),m_0} \times_{[pt/\Gamma_r]} \cM_{rq}^{\bfg_1(\alpha),m_1}  \simeq \partial^{(prq,\bfh)}\cM_{pq}^{\bfg,m}$.
\item The boundary face maps are strictly associative
\item The group $\Pi$ acts strictly, via stratified homeomorphisms 
\[
\xymatrix{
\cM_{pq} \ar[r] \ar[d] & \cM_{ap,aq} \ar[d] \\
\bA_{pq}  \ar[r] & \bA_{ap,aq}
}
\]
for $a\in \Pi$  and $p,q\in \cP$.
\end{enumerate}
\end{definition}

\begin{example} \label{ex:Morse flow cat definitions}
For Morse theory on $[X/\Gamma]$:
\begin{itemize}
\item $\cP$ is a set of critical points $p$ of a Morse function $f: X/\Gamma \to \bR$. In particular, $\cP$ is in bijection with the set of $\Gamma$-orbits $\tilde{p}$ of critical points of the $\Gamma$-invariant Morse function $\tilde{f}: X \to \bR$ (the lift from $f$);
\item $\Gamma_{p}$ is the isotropy group of $p$;
\item $\Pi$ is trivial;
\item $\cA$ is the usual action of $p$ (i.e. $\cA(p)=f(p)$) and
$\cA_{\mathbb{Z}}$ is an integral approximation of (a multiple of the) action of $p$ (see Section \ref{s:interalaction} for the discussion of the Floer setting, which specialises to the Morse setting);
\item $\cM_{pq}^k= \emptyset$ if $k>0$, and is the moduli of gradient trajectories from $p$ to $q$ otherwise 
\item $\bA_{pq}$ indexes possible breaking of Morse flow-lines between critical points $p$ and $q$.
\end{itemize}
\end{example} 

\begin{example}\label{ex:Ham Floer flow cat definitions}
For bulk-deformed orbifold Hamiltonian Floer theory (see Section \ref{s:orbifoldFloer}):
\begin{itemize}
\item $\cP$ is a set of formal capped Hamiltonian orbits in an orbifold $Y = [X/\Gamma]$ with partial order given by $p < q$ if and only if the compactified moduli space of Floer solutions $\cM_{pq} \neq \emptyset$; 
\item $\Gamma_p$ is the isotropy group of $p$;
\item $\Pi$ is a discrete subgroup of $\mathbb{R}$ containing $\frac{1}{|\Gamma|}\omega_X(H_2(X;\mathbb{Z}))$, acting by formal recapping; 
\item $\cA$ and $\cA_{\Pi}$ are the usual actions, while
$\cA_{\mathbb{Z}}$ and $\cA_{\Pi,\bZ}$ are `integralised' actions defined by an integral scaling of a rational approximation to $\cA$ and $\cA_{\Pi}$, respectively
(see Section \ref{s:interalaction}); 
\item $\bA_{pq}^k$ indexes the possible breakings and hence combinatorics for the moduli space of stable Floer cylinders with constraints at $k$ ordered marked points $\cM^k_{pq}$ between two given orbits (for which the formal cap at $q$ is the sum of that at $p$ with the cylinder).
\end{itemize}
\end{example}

\begin{remark}
    Since we are going to define bulk-deformed Hamiltonian Floer theory, the virtual dimension of the moduli $\cM^k_{pq}$ will depend on the bulk constraints and is not determined by $p$ and $q$ even though they are capped orbits.
    Therefore, we need the function $\vdim$ as part of the data in the Definition \ref{d:omfc}.
\end{remark}

\subsection{Global charts for $\bA$-spaces and fibre products}

\emph{Informally, global charts without boundary serve as models for solution spaces to elliptic equations which one expects to be smooth compact manifolds. Following  \cite{BX,Rez} one can extend this set-up to the case of a solution space which should be a manifold with boundary and corners. We will require a further extension to families of global charts parametrized over an auxiliary topological manifold.
}

For an $\bA$-orbispace $T$ with finite isotropy group at every point, a global Kuranishi chart $\bT = (G,\scrT,E,s, \phi)$ for $T$ consists of a compact Lie group $G$, an oriented topological $\bA$-manifold $\scrT$ (i.e. an $\bA$-stratified manifold with boundary and corners such that $\partial^{\alpha} \scrT$ is a manifold with boundary and corners for all $\alpha \in \bA$), an oriented $G$-bundle $E$, a $G$-equivariant section $s$, and an isomorphism of $\bA$-orbispaces $\phi:s^{-1}(0)/G \simeq T$. Operation \ref{o:stablization} can be defined in the same way for a global Kuranishi chart $\bT$.
Definition \ref{d:embedding} and \ref{d:itStab} carry over as well.

One additional feature for a global Kuranishi chart $\bT$ of an 
$\bA$-orbispace $T$ is that it induces a collection of Kuranishi charts $\partial^\alpha \bT$ for the spaces $\partial^{\alpha}T$, by restriction.

\begin{remark} \label{rmk:remember_orbifold}
Despite our notation, which reflects the underlying geometric construction, we view a global chart $(G,\scrT,E,s)$ as defining the triple $(\scrT/G,E/G,s)$ of a section of an orbibundle over an orbifold, and do not specifically  retain the information of the Lie group presenting the orbifold as a global quotient.  Thus, the `group enlargement' axiom of \cite{AMS1} on global charts is \emph{trivial} in this set-up, since the underlying orbifold $\scrT/G$ is literally identical.  The moves on global charts in this setting are thus germ-equivalence (restricting to an open neighbourhood of $s^{-1}(0)/G \subset \scrT/G$) and stabilisation of the chart by an orbibundle (above in Operation \ref{o:stablization}, the bundle $F \to \scrT/G$).
\end{remark}

We need a slight generalisation of the terminology of a global chart to allow disjoint union.
More precisely, let $\bT=(G,\scrT,E,s)$ and $\bT'=(G',\scrT',E',s')$ be global charts of $T$ and $T'$ respectively.
 Strictly speaking, we cannot view  $\bT \sqcup \bT'$ as a global chart of $T \sqcup T'$ because we cannot write $\bT \sqcup \bT'$ in the form of $(G'',\scrT'',E'',s'')$.
However, if we use Remark \ref{rmk:remember_orbifold} and define $\bT=(\scrT/G,E/G,s)$ and $\bT'=(\scrT'/G',E'/G',s')$ using orbifolds and orbibundles, then $\bT \sqcup \bT'$ makes sense as a global chart of $T \sqcup T'$ because we can take disjoint union of orbibundles
over orbifolds.
In what follows, we will give a global chart of $\cM_{pq}$.
Since $\cM_{pq}$ is a disjoint union of $\cM_{pq}^k$ and $\cM_{pq}^k$ is a disjoint union of $\cM_{pq}^{k,m}$\footnote{When we do Hamiltonian Floer theory, $\cM_{pq}^{k,m}$ will break further into a disjoint union over different possible monodromy $\bfm$ and classes $\epsilon$.}, we use the convention from now on that a global chart of a compact space $T$ (e.g. $\cM_{pq}$ or $\cM_{pq}^k$) is always understood as a disjoint union of global charts of its components. 

We also need to introduce
fibre products of global charts for later use.

Suppose that $\bT_0=(G_0,\scrT_0,E_0,s_0)$ and $\bT_1=(G_1,\scrT_1,E_1,s_1)$ are global Kuranishi charts and $H$ is a finite group.
If $e_i:\scrT_i/G_i \to [pt/H]$ is a representable orbifold map for $i=0,1$, then we can form the fibre product global chart 
$\bT_0 \times_{[pt/H]} \bT_1:=(G_0 \times G_1, \scrT_0 \times H \times \scrT_1, E_0 \oplus E_1, s_0 \oplus s_1)$
where $G_0 \times G_1$ acts on $\scrT_0 \times H \times \scrT_1$ by
\begin{align}\label{eq:fibreproduct}
 (x_0,h,x_1)\cdot(g_0,g_1) =(x_0 \cdot g_0, (e_1)_*(g_1)h((e_0)_*(g_0))^{-1}, x_1 \cdot g_1)
\end{align}
where $(e_i)_*:G_i \to H$ is the  group homomorphism induced by $e_i$.
The action on $E_0 \oplus E_1$ is defined similarly.
The orbifold $(\scrT_0 \times H \times \scrT_1)/(G_0 \times G_1)$ is exactly the same as $\scrT_0/G_0 \times_{[pt/H]} \scrT_1/G_1$, the fibre product of the given orbifolds (cf. \cite[Section 2.3]{Moe02}).

\begin{remark}\label{r:assocom}
Taking fibre products of global charts is associative. That is 
\[
(\bT_0 \times_{[pt/H_{01}]} \bT_1) \times_{[pt/H_{12}]} \bT_2=\bT_0 \times_{[pt/H_{01}]} (\bT_1 \times_{[pt/H_{12}]} \bT_2).
\]
Moreover, there is a canonical isomorphism
\[
(\bT_0 \times_{[pt/H_{01}]} \bT_1) \times_{[pt/H_{12}]} \bT_2 \simeq (\bT_1 \times_{[pt/H_{01}]} \bT_0) \times_{[pt/H_{12}]} \bT_2
\]
when we change the order as indicated.
\end{remark}

\subsection{Ordered marked flow categories enriched over global charts\label{sec:ordered marked}}

\emph{ It will be important to allow collections of charts which present the strata of an $\bA$-space which are more general than those obtained by restriction from a single chart, but incorporate suitable homotopies of sections.}


For any $\beta \le \alpha$, let $\partial_{\beta}^{\alpha} \bA:=\{\gamma \in \bA|  \beta \le \gamma \le \alpha\}$ be equipped with the induced topology and $d_{\alpha\beta}:=\depth(\beta)-\depth(\alpha)$.
Note that $\partial_{\beta}^{\alpha} \bA$ is homeomorphic to $\{0,1\}^{d_{\alpha\beta}}$ (with product poset topology on $\{0,1\}^{d_{\alpha\beta}}$).
Such a homeomorphism necessarily sends $\beta$ to $(0,\dots,0)$ and $\alpha$ to $(1,\dots,1)$.
We fix such a homeomorphism and then forget the topology on $\partial_{\beta}^{\alpha} \bA$ and identify $\partial_{\beta}^{\alpha} \bA$  as a subset of $\mathbb{R}^{d_{\alpha\beta}}$.
Let $I_{\alpha\beta}:=[0,1]^{d_{\alpha\beta}}$ be the convex hull of $\partial_{\beta}^{\alpha} \bA$ in $\mathbb{R}^{d_{\alpha\beta}}$.
We stratify $I_{\alpha\beta}$ by its manifold with boundary and corner structure.
In other words, a stratum of $I_{\alpha\beta}$ is the convex hull of $\partial^{\gamma}_{\delta} \bA$ for some $\gamma,\delta$ such that $\alpha \ge \gamma \ge \delta \ge  \beta$.
We denote this stratum by $\partial^{\gamma}_{\delta} I_{\alpha\beta}$.
There is a canonical isomorphism $\partial^{\alpha}_{\gamma}I_{\alpha\beta}=I_{\alpha,\gamma}$
and $\partial^{\gamma}_{\beta}I_{\alpha\beta}=I_{\gamma,\beta}$.

We call $I_{\alpha\beta}$ the \emph{homotopy cube} from $\alpha$ to $\beta$.


\begin{definition}
Given a global Kuranishi chart $\bT_{\beta}=(G_{\beta},\scrT_{\beta},E_{\beta},s_{\beta})$, an $I_{\alpha\beta}$-parametrized family of sections starting from $s_{\beta}$ is an $I_{\alpha\beta}$-family of $G_{\beta}$-equivariant section $s_{\beta}^I$ which equals $s_{\beta}$ at $\{\beta\} \subset I_{\alpha\beta}$. 
\end{definition}

Given two global charts $\bT=(G,\scrT,E,s)$, $\bT'=(G',\scrT',E',s')$ with $I_{\alpha\beta}$-parametrized family of sections $s^I$ and $(s')^I$, an embedding from $(G,\scrT,E,s^I)$ to $(G',\scrT',E',(s')^I)$ is an $I_{\alpha\beta}$-family of embeddings from 
$(G,\scrT,E,s^I|_x)$ to  $(G',\scrT',E',(s')^I|_x)$ for $x \in I_{\alpha\beta}$ (cf. Definition \ref{d:embedding}).

\begin{definition}\label{d:coherenthomotopy}
A \emph{system of chart presentations with coherent homotopies}\footnote{ \cite{BX} defines a system of chart presentations as a collection of global charts such that $\partial^{\beta} \bT_{\alpha}$ is a stabilisation of $\bT_{\beta}$ for all $\beta \le \alpha$. This notion is too restrictive for our setting.} for an $\bA_{p,q}$ space $T$ comprises
\begin{enumerate}
\item  for any $\alpha \in \bA_{p,q}$,
a global chart $\bT_{\alpha} = (G_{\alpha},\scrT_{\alpha}, E_{\alpha}, s_{\alpha})$ presenting the stratum $\partial^{\alpha} T$, and an $I_{pq,\alpha}$-parametrized family $\bT_{\alpha}^I=(G_{\alpha},\scrT_{\alpha},E_{\alpha},s_{\alpha}^I)$ of sections of $E_{\alpha}$ starting from $s_{\alpha}$

\item  for any $\beta \leq \alpha$,  an embedding  
\begin{align}\label{eq:emb}
\partial^{\beta}\bT^I_{\alpha} \to \bT_{\beta}^I|_{\partial_{\alpha}^{pq}I_{pq, \beta}}
\end{align}
of global charts with $I_{pq,\alpha}$-parametrized family of sections,
\item  for any $\gamma \leq \beta \leq \alpha$, the embeddings 
 $\partial^{\beta}\bT^I_{\alpha} \to \bT_{\beta}^I|_{\partial_{\alpha}^{pq}I_{pq, \beta}}$ and 
 $\partial^{\gamma}\bT^I_{\beta} \to \bT_{\gamma}^I|_{\partial_{\beta}^{pq}I_{pq, \gamma}}$
 restrict to embeddings
$\partial^{\gamma}\bT^I_{\alpha} \to \partial^{\gamma}\bT_{\beta}^I|_{\partial_{\alpha}^{pq}I_{pq, \beta}}$ 
and $\partial^{\gamma}\bT^I_{\beta}|_{\partial_{\alpha}^{pq}I_{pq, \beta}} \to \bT_{\gamma}^I|_{\partial_{\alpha}^{pq}I_{pq, \gamma}}$, respectively.
We insist that the composition of them is isomorphic to the embedding  $\partial^{\gamma}\bT^I_{\alpha} \to \bT_{\gamma}^I|_{\partial_{\alpha}^{pq}I_{pq, \gamma}}$.

\end{enumerate}    
\end{definition}

\begin{figure}[ht]
\begin{center}
\begin{tikzpicture}

\draw[semithick,fill] (0,0) circle (0.1);
\draw (0.75,0.5) node {$I_{pq,pq}$};

\draw[semithick,fill] (-1,0) circle (0.1);
\draw[semithick,fill] (-5,0) circle (0.1);
\draw[semithick] (-5,0) -- (-1,0);
\draw (-3,0.5) node {$I_{pq,plq}$};

\draw[semithick,fill] (0,-1) circle (0.1);
\draw[semithick,fill] (0,-5) circle (0.1);
\draw[semithick] (0,-1) -- (0,-5);
\draw (0.75,-3) node {$I_{pq,prq}$};

\draw[semithick] (-1,-1) -- (-5,-1) -- (-5,-5) -- (-1,-5) -- (-1,-1);
\draw (-3,-3) node {$I_{pq,prlq}$};
\draw (1.2,-1) node {$\partial^{pq}_{pq} I_{pq,prq}$};
\draw (1.2,-5.2) node {$\partial^{prq}_{prq} I_{pq,prq}$};

\draw (4,-2) node {$plq$}; 
\draw (5,-2) node {$\leq$};
\draw (6,-2) node {$pq$};
\draw (4,-3) node[rotate=90] {$\leq$};
\draw (6,-3) node[rotate=90] {$\leq$};
\draw (4,-4) node {$prlq$}; 
\draw (5,-4) node {$\leq$};
\draw (6,-4) node {$prq$};

\draw (-6.5,-2) node {$\partial^{pq}_{plq}$};
\draw[semithick,dashed,rounded corners,->] (-6,-2) -- (-4,-2) -- (-4,-1.25);

\draw (-2.5,-6) node {$\partial^{pq}_{prq}$};
\draw[semithick,dashed,rounded corners,->] (-2.5,-5.5) -- (-2.5,-4.25) -- (-1.25,-4.25);

\draw[semithick] (0,-6) -- (0,-12);
\draw[semithick] (-5,-9) -- (5,-9);

\draw (2.5,-7.5) node {$\mathbf{T^I_{pq,pq}}$};
\draw (2.5,-10.5) node {$\mathbf{T^I_{pq,prq}}$};
\draw (-2.5,-7.5) node {$\mathbf{T^I_{pq,plq}}$};
\draw (-2.5,-10.5) node {$\mathbf{T^I_{pq,prlq}}$};

\draw (4.5,-8.25) node {$\partial^{prq}T$};
\draw[semithick,dashed,rounded corners,->] (4,-8.25) -- (3.5,-8.25) -- (3.5,-8.75);
\draw (4.5,-9.75) node {$\partial^{pq}_{pq}I$};
\draw[semithick,dashed,rounded corners,->] (4,-9.75) -- (3.5,-9.75) -- (3.5,-9.25);

\draw (-4.5,-9.75) node {$\partial^{pq}_{plq}I$};
\draw[semithick,dashed,rounded corners,->] (-4,-9.75) -- (-3.5,-9.75) -- (-3.5,-9.25);

\draw (0.75,-6.5) node {$\partial^{plq}T$};
\draw[semithick,dashed,rounded corners,->] (0.75,-6.75) -- (0.75,-7.25) -- (0.25,-7.25);
\draw (-0.75,-6.5) node {$\partial^{pq}_{pq}I$};
\draw[semithick,dashed,rounded corners,->] (-0.75,-6.75) -- (-0.75,-7.25) -- (-0.25,-7.25);

\draw (0.75,-11.5) node {$\partial^{prlq}T$};
\draw[semithick,dashed,rounded corners,->] (0.75,-11.25) -- (0.75,-10.75) -- (0.25,-10.75);
\draw (-0.75,-11.5) node {$\partial^{pq}_{prq}I$};
\draw[semithick,dashed,rounded corners,->] (-0.75,-11.25) -- (-0.75,-10.75) -- (-0.25,-10.75);

\end{tikzpicture}
\end{center}
\caption{Boundary strata and boundary charts}\label{fig:internalcoherent}
\end{figure}

\begin{definition}\label{d:global lift}
    For an ordered marked flow category $\cC$ over a poset $\cP$ a  global chart lift comprises 
\begin{enumerate}
\item for every $p,q \in \cP$, $\bfg \in (\|\Gamma\|^\circ)^k$ and $m \in \mathbb{Z}$, a system of chart presentations with coherent homotopies of $\cM_{pq}^{\bfg,m}$ as an $\bA^\bfg_{pq}$-space. That is, for any $\alpha \in \bA^\bfg_{pq}$ with $\depth(\alpha)=l$, an $(m-l)$-dimensional global chart of $\partial^{\alpha}(\cM_{pq}^{\bfg,m})$
\[
\bT_{pq,\alpha}^{\bfg,m} = (G_{pq,\alpha}^{\bfg,m}, \scrT_{pq,\alpha}^{\bfg,m}, E_{pq,\alpha}^{\bfg,m}, s_{pq,\alpha}^{\bfg,m})
\]
and a $I_{\alpha,pq}$-parametrized family $\bT_{pq,\alpha}^{\bfg,m,I}$ of sections of $\bT_{pq,\alpha}^{\bfg,m}$
such that Definition \ref{d:coherenthomotopy}(2) and (3) are satisfied.

\item for all $p$, $\scrT_{pp}^{\bfg,m}=\emptyset$ if $(\bfg,m) \neq (\emptyset,0)$, $\scrT_{pp}^{\emptyset,0}$ is a singleton and $E_{pp}^{\emptyset,0}=0$,
\item  for any $\alpha = (p r_1\ldots r_s q,\bfh) \in \bA_{pq}$, $\beta_i \in \bA_{r_i,r_{i+1}}$ for $i=0,\dots,s$ and $\beta$ the concatenation of $(\beta_i)_{i=0}^s$ refining $\alpha$, an embedding
\begin{align}\label{eq:productboundary}
\bT_{pr_1,\beta_0}^{\bfg_0(\alpha),I}\times_{[pt/\Gamma_{r_1}]} \cdots \times_{[pt/\Gamma_{r_s}]} \bT_{r_sq,\beta_s}^{\bfg_s(\alpha),I} \to 
\bT_{pq,\beta}^{\bfg,I}|_{\partial^{\alpha}_{\beta}I_{pq,\beta}}
\end{align}
of global charts  with $(\partial^{\alpha}_{\beta}I_{pq,\beta}=I_{\alpha,\beta}=I_{pr_1,\beta_0} \times \dots \times I_{r_sq,\beta_s})$-family of sections; 
\item Similar to Definition \ref{d:coherenthomotopy}(3), we can restrict the embedding \eqref{eq:productboundary} to a boundary stratum of a global chart $\partial^{\gamma} \bT$ or to a boundary stratum of a homotopy cube. We insist that all possible compositions of embeddings arising from restrictions of \eqref{eq:emb} and \eqref{eq:productboundary} with the same domain and target are isomorphic;
\item We insist that the entire package is strictly $\Pi$-equivariant.
\end{enumerate}
\end{definition}

In particular, when $\alpha=\beta=(prq,\bfh)$, $\alpha_0=(pr,\mathbf{0})$ and $\alpha_1=(rq,\mathbf{0})$, then 
there is an embedding
\[
\bT_{pr}^{\bfg_0(\alpha)} \times_{[pt/\Gamma_r]} \bT_{rq}^{\bfg_1(\alpha)} \to \bT_{pq,(prq,\bfh)}^{\bfg,I}|_{\partial^{(prq,\bfh)}_{(prq,\bfh)} I_{pq,(prq,\bfh)}}
\]
(cf. Definition \ref{d:omfc}(5)).

\begin{figure}[ht]
\begin{center}
\begin{tikzpicture}

\draw[semithick,fill] (0,0) circle (0.1);
\draw (0.75,0) node {$I_{rq,rq}$};
\draw (0.75,3) node {$I_{pq,prq}$};
\draw (1.5,1) node {$\partial^{prq}_{prq}I = I_{prq,prq}$};

\draw[semithick,fill] (0,1) circle (0.1);
\draw[semithick,fill] (0,5) circle (0.1);
\draw[semithick,fill] (-1,0) circle (0.1);
\draw[semithick,fill] (-5,0) circle (0.1);
\draw[semithick,fill] (-5,1) circle (0.1);
\draw[semithick,fill] (-1,1) circle (0.1);

\draw[semithick] (-5,0)--(-1,0);
\draw[semithick] (0,1)--(0,5);
\draw[semithick] (-5,1) -- (-5,5) -- (-1,5) -- (-1,1) -- (-5,1);

\draw (4,0) node {$rlq$};
\draw (5,0) node {$\leq$};
\draw (6,0) node {$rq$};
\draw (4,3) node {$prlq$};
\draw (5,3) node {$\leq$};
\draw (6,3) node {$prq$};

\draw (-3,3) node {$I_{pq,prlq}$};
\draw (-3,-0.5) node {$I_{rq,rlq}$};

\draw (-7,1.75) node {$I_{prq,prlq} = \partial^{prq}_{prlq}I$};
\draw[semithick,dashed,rounded corners,->] (-5.5,1.75) -- (-4,1.75) -- (-4,1.25);


\draw (-5,-9) -- (5,-9);
\draw[semithick,fill] (-5,-9) circle (0.1);
\draw[semithick,fill] (0,-9) circle (0.1);

\draw[thick, ->] (0,-8.5) -- (0,-6.5);
\draw (1.25,-7.5) node {$\times_{[pt/\Gamma_r]} T_{pr,pr}$};
\draw (3,-8.5) node {$T^I_{rq,rq}$};
\draw (-3,-8.5) node {$T^I_{rq,rlq}$};

\draw (3,-4.5) node {$T^I_{pq,prq}$};
\draw (-3,-4.5) node {$T^I_{pq,prlq}$};

\draw (-6.5,-5) node {$\partial^{prq}_{prlq} I$};
\draw[semithick,dashed,rounded corners,->] (-5.75,-5) -- (-4,-5) -- (-4,-5.75);
\draw (6.25,-5) node {$\partial^{prq}_{prq} I$};
\draw[semithick,dashed,rounded corners,->] (5.5,-5) -- (4.25,-5) -- (4.25,-5.75);

\draw (-5,-6) -- (5,-6);
\draw (-5,-6) -- (-5,-2);
\draw (0,-6) -- (0,-2);
\draw (-5,-3) -- (5,-3);
\draw[semithick,fill] (-5,-6) circle (0.1);
\draw[semithick,fill] (0,-6) circle (0.1);

\end{tikzpicture}
\end{center}
\caption{Fibre product structure at the boundary}\label{fig:externalcoherent}
\end{figure}

\begin{remark}
    Our definition of flow category enriched over Kuranishi charts is more involved than that in \cite{BX} because the restriction of our global chart to a boundary stratum is not the stabilisation of the fibre product of the global charts for the components of the broken Floer cylinder.
    Instead, we will show that there is a larger thickening and obstruction bundle to which both of them stabilize, but the stabilized obstruction sections are different. Therefore, we need a homotopy of sections in the larger space to relate the Kuranishi charts. Similarly, we don't have strict associativity of the Kuranishi charts when we restrict to deeper strata, but only homotopy coherent associativity (cf. Lemma \ref{l:pertubationexist}).
\end{remark}

\begin{remark}[(Analogue of Remark \ref{r:nchooser_A})]\label{r:nchooser}
Notice that if both $\bfg_i(prq,\bfh,\bfg)=\bfg_i(prq,\bfh',\bfg)$ for $i=0,1$, then we have a strict isomorphism from $\bT_{pr}^{\bfg_0} \times_{[pt/\Gamma_r]} \bT_{rq}^{\bfg_1}$ to each of 
$ \bT_{pq,(prq,\bfh)}^{\bfg}$ and $ \bT_{pq,(prq,\bfh')}^{\bfg}$.
Since there are $\frac{k!}{k_0!k_1!}$  monotonic increasing (injective) functions from $\{1, \dots, k_0\}$ to $\{1, \dots, k=k_0+k_1\}$,  $\bT_{pr}^{\bfg_0} \times_{[pt/\Gamma_r]} \bT_{rq}^{\bfg_1}$ appears in $\frac{k!}{k_0!k_1!}$ many depth one strata of $\bT_{pq}^{k}=\sqcup_{\bfg \in (\|\Gamma\|^{\circ})^k} \bT_{pq}^{\bfg}$.
\end{remark}

\begin{remark}\label{r:bulksmooth}
    Even though our main goal is to set up the orbifold theory, our  framework of ordered marked flow categories also yields a definition of bulk-deformed Hamiltonian Floer theory on smooth symplectic manifolds. 
\end{remark}

\begin{example}[A zig-zag]\label{eq:zig-zag}
Assume $k=0$ for simplicity.
    Suppose that we have a global chart $\bT_{pq}$ of $\cM_{pq}$. It induces a global chart $\partial^{prq}\bT_{pq}$ of $\partial^{prq}\cM_{pq}$.
    On the other hand, suppose we also have global charts $\bT_{pr}$ and $\bT_{rq}$ of $\cM_{pr}$ and $\cM_{rq}$ so that we can form the fibre product Kuranishi chart $\bT_{pr} \times_{[pt/\Gamma_r]} \bT_{rq}$ of $\cM_{pr} \times_{[pt/\Gamma_r]} \cM_{rq} \simeq \partial^{prq}\cM_{pq}$.

    If, moreover, we have a $[0,1]$-parametrized family of sections of a Kuranishi chart $(\bT_t)_{t \in [0,1]}$ such that $\bT_0$ is an iterative stablization of $\bT_{pr} \times_{[pt/\Gamma_r]} \bT_{rq}$ and $\bT_1$ is an iterative stablization of $\partial^{prq}\bT_{pq}$, then we can take $\bT_{pq,prq}^I=(\bT_t)_{t \in [0,1]}$.
    The assumption that $\bT_1$ is an iterative stablization of $\partial^{prq}\bT_{pq}$ corresponds to Definition \ref{d:coherenthomotopy} item 2, whilst  $\bT_0$ being an iterative stablization of $\bT_{pr} \times_{[pt/\Gamma_r]} \bT_{rq}$ corresponds to Definition \ref{d:global lift} item 3.
    \end{example}

\begin{example}[Homotopy coherent associativity]\label{eg:homotopyassociativity}
    Assume $k=0$ for simplicity.
    Suppose that we have a global chart $\bT_{pq}$ inducing Kuranishi charts $\partial^{prq} \bT_{pq}$, $\partial^{plq} \bT_{pq}$ and $\partial^{prlq} \bT_{pq}$.
    Suppose we have also chosen $\bT^I_{pq,prq}$ and $\bT^I_{pq,plq}$ so that $\partial^{prq} \bT_{pq}$, $\partial^{plq} \bT_{pq}$ embed to the respective ends.
    Then we need to choose $\bT^I_{pq,prlq}$ that is compatible with both $\bT^I_{pq,prq}$ and $\bT^I_{pq,plq}$.
     The compatibility condition is that $\partial^{prlq}\bT^I_{pq,pq}$, $\partial^{prlq}\bT^I_{pq,prq}$ and $\partial^{prlq}\bT^I_{pq,plq}$ embed to the respective boundary strata of $\bT^I_{pq,prlq}$, and the embedding from $\partial^{prlq}\bT^I_{pq,pq}$ to $\bT^I_{pq,prlq}|_{\partial^{pq}_{pq} I_{pq,prlq}}$ is isomorphic to the composition of embeddings that factor through $\partial^{prlq}\bT^I_{pq,prq}$ and $\partial^{prlq}\bT^I_{pq,plq}$ (see Figure \ref{fig:internalcoherent}).
    This corresponds to Definition \ref{d:coherenthomotopy}(2)(3).

    On the other hand, we need the other boundary strata of $T^I_{pq,prlq}$ to be the stabilisation of the respective fibre products (see Figure \ref{fig:externalcoherent}).
    This corresponds to Definition \ref{d:global lift}(3).
    \end{example}

\subsection{Smoothing theory\label{subsec:smoothing}} \emph{We discuss abstract smoothing theory following \cite{AMS1, BX, Rez}.}


A microbundle of rank $n$ over a topological space $X$ comprises a topological space $E$ and a diagram $\{X \stackrel{s}{\to} E \stackrel{p}{\to} X\}$ satisfying 
\begin{itemize}
    \item $p\circ s = \id$;
    \item for each $x \in X$ there are neighbourhoods $U_x \subset X$ of $x$ and $V_x \subset E$ of $s(x)$ and a
homeomorphism $h_x : U_x \times  \bR^n \to V_x$ for which $p \circ h_x = pr_{U_x}$ and $h_x |_{U_x\times \{0\}} = s$. Here $pr_{U_x}$ denotes the projection map to the factor $U_x$.
\end{itemize}
These are viewed as germs near the image $s(X)$; a morphism $E_1 \to E_2$ of microbundles $\{X \stackrel{s_i}{\to}  E_i \stackrel{p_i}{\to} X\}$ is a map $E_1 \supset U_1 \to E_2$ defined on an open neighbourhood $U_1$ of $s_1(X)$ and entwining the maps $s_i,p_j$, viewed up to the equivalence of agreement on a perhaps smaller neighbourhood of $s_1(X)$.

A topological manifold has a tangent microbundle $T_{\mu}(X) = \{X \stackrel{\Delta}{\to}  X\times X \stackrel{pr_1}{\to} X\}$.  
If $\pi: V \to X$ is a vector bundle, its associated microbundle is $\{X \stackrel{0}{\to} V \stackrel{\pi}{\to} X\}$. 

There is a classifying space $BTop(n)$ for rank $n$ microbundles, with a canonical map $BO(n) \to BTop(n)$. A vector bundle lift of a microbundle $E$ over $X$ is a vector bundle $V_E$ over $X$, together with an equivalence between the associated microbundle for $V_E$ and $E$ itself. Equivalently, this is a diagram
\[
\xymatrix{
& BO(n) \ar[d] \\ 
X \ar[r]_-{E} \ar@{.>}[ur]& BTop(n)
}
\]
where the horizontal map is that classifying $E$. For a smooth manifold, the classical tangent bundle $TX$ is a vector bundle lift of $T_{\mu}(X)$.  Increasing $n$ gives a natural notion of a stable vector bundle lift.  If $X$ is a $G$-space, a $G$-microbundle comprises a $G$-space $E$ fitting into the corresponding diagram with $s,p$ both $G$-equivariant continuous maps. There is an obvious  notion of a $G$-vector bundle lift of a $G$-microbundle. 

If $X$ is a topological manifold, a \emph{stable smoothing} of $X$ is a smooth structure on $X \times \bR^s$ for some $s >0$. Lashof's equivariant smoothing theorem \cite{Lashof} says:

\begin{theo}
    Let $G$ be a compact Lie group acting on a topological manifold $X$ with finitely many orbit types. There is a bijection between homotopy classes of $G$-vector bundle lift of $T_{\mu}X$ and of $G$-equivariant stable smooth structures on $X$ up to concordance.
\end{theo}

In particular, a choice of bundle lift of $T_{\mu}X$ gives a smooth structure on $X \times V$, for some $G$-representation $V$, canonical up to further stabilisation and concordance.

In Floer theory, global charts are moduli spaces of perturbed holomorphic curves which live naturally over moduli spaces of domains. Classical elliptic theory shows that the moduli space of maps with a fixed domain is, in good cases and for sufficiently generic data, a submanifold of a Banach manifold. This makes the following relevant, cf. \cite[Sections 4.4-4.5]{AMS1} for details of the notation and proof: 

\begin{prop}\label{prop:stable_smoothing}
    Let $\scrT$ be a topological $G$-manifold which admits the following structure: there is a smooth $G$-manifold $\scrF$, and a $C^1_{loc}$ $G$-equivariant topological submersion $\pi: \scrT \to \scrF$. Then $T_{\mu}\scrT$ admits a $G$-vector bundle lift, hence $\scrT$ admits a stable smoothing.
\end{prop}

The key point is that the $C^1_{loc}$-structure gives rise to a canonical vertical tangent bundle $T^{vt}(\pi)$, and the underlying bundle of the vector bundle lift of $T_{\mu}\scrT$ is $T^{vt}(\pi) \oplus \pi^* T\scrF$.

\begin{lemma} \label{lem:smoothing}
Suppose that $\cC$ admits a global chart lift with the property that 
\begin{enumerate}
    \item for the system of chart presentation of $\cM_{pq}$ and each $\alpha \le pq$, there is a  topological submersion $\pi_{pq,\alpha}: \scrT_{pq,\alpha} \to \partial^\alpha \scrF_{pq}$ over a smooth $\partial^\alpha\bA_{pq}$-manifold such that $\pi_{pq,\alpha}$ is relatively smooth, in particular there is a well-defined fibrewise $C^1$-structure;
    \item the maps $\pi_{pq,\alpha}$ are compatible with the boundary decompositions.
\end{enumerate}
Then $\cC$ admits another global chart lift with coherent homotopies in which the thickenings $\scrT'_{pq}$ are smooth manifolds with corners, and the sections $s_{pq}$ are smooth.  
\end{lemma}

\begin{proof}
We define $\hat{\pi}_{pq,\alpha}:\scrT_{pq,\alpha}^I =\scrT_{pq,\alpha} \times I_{pq,\alpha} \to \partial^\alpha\scrF_{pq}$ to be the composition of the projection to the first factor and $\pi_{pq,\alpha}$.
Given the indicated structure, the maps $\hat{\pi}_{pq,\alpha}$ are compatible with the boundary decompositions so
one can  construct `outer collarings' of the $\bA_{pq}$-spaces $\scrT_{pq}$ and then use Lashof's smoothing theory as in \cite{AMS1,BX,Rez} inductively to find systems of smooth chart presentations. The required extension of Lashof's theory to manifolds  stratified by posets and in particular to manifolds with corners is established in  \cite[Appendix B]{BX} by a direct construction and \cite[Section 4.2, 4.3]{Rez} by an argument relying on $(\bZ/2)^i$-equivariant smoothing theory on an iterated doubling of a manifold with  corner strata of depth at most $i$.  The upshot is that, after further stabilising the obstruction bundles in the global charts if necessary, we can assume that all the global charts are smooth. 

The hypotheses imply that the vector bundle lifts for the chart microbundles extend over the coherent systems of homotopies.  We inductively smooth  all charts stabilised by the appropriate $I_{pq,\alpha}$-factors, without ever changing the sections. This yields a sytem of smooth global charts, with continuous sections. We then separately smooth the sections by the mollification argument of \cite{AMS2}. 
\end{proof}


\begin{example}
Following on from Example \ref{ex:Morse flow cat definitions}, for a pair $(f,g)$ of a Morse function $f$ and Riemannian metric $g$ satisfying somewhat stronger conditions than the usual Morse-Smale conditions (in particular, making $g$ standard in neighbourhoods of the critical points), Wehrheim \cite{KW} showed that all the Morse moduli spaces $\cM_{xy}$ are smooth manifolds with corners of the expected dimension.  It follows that the Morse flow category admits a global chart lift with $\scrT_{xy} = \cM_{xy}$ and $E_{xy} = \{0\}$ for all pairs. 
\end{example}

\subsection{The chain complex of an ordered marked flow category}\label{s:multivalue}

\emph{In this section we explain how to use multivalued perturbations to obtain a chain complex over  
$\Lambda_{\Pi}$ from a flow category with a global chart lift. At this point we use smoothness of the global charts, to have a notion of transversality of such a multi-valued perturbation.} 

We review some facts on weighted branched manifolds from \cite{CMS}. We take an `equivariant' viewpoint; a more intrinsic description in terms of quotient orbifolds is given in \cite{McDuff06}.

Let $B$ be a smooth manifold carrying an action of a compact Lie group $G$ with finite stabilisers.  A \emph{weighted branched $d$-submanifold} of $B$ is a function $\lambda: B \to \bQ_{\geq 0}$ satisfying
\begin{itemize}
\item $\lambda(gx) = \lambda(x)$ for all $g \in G$;
\item any $x\in B$ is contained in an open neighbourhood $U$ which contains finitely many $(d+\dim(G))$-dimensional relatively closed submanifolds $M_1,\ldots, M_m \subset U$, and $\lambda_1,\ldots,\lambda_m \in \bQ_{\geq 0}$, such that $\lambda = \sum_{x\in M_i} \lambda_i$ throughout $U$.
\end{itemize}
The support is $\{x \, | \, \lambda(x) > 0\}$.  The subset of branch points, where $\lambda$ restricted to the support is not locally constant near $x$, is contained in the locus of points $x$ where there are two local branches $M_i$ at $x$ with $x\in (M_1\cap M_2) \backslash int_{M_i}(M_1\cap M_2)$ for $i=1,2$.  An ordinary submanifold defines a weighted branched submanifold with empty branch set, by taking the weighting function identically equal to $1$ on the submanifold, which becomes the support. 

Since the action of $G$ on $B$ has finite stabilisers there is a subbundle $\frak{g} \subset TB$ of tangents to the orbits. We have the Grassmannian
\[
Gr_d(TB/\frak{g}) = \{ (x,F) \, | \, x \in B, F \in Gr_{d+dim(G)}(T_xB), \frak{g} \subset F\}
\]
A weighted branched $d$-manifold $\lambda: B \to \bQ_{\geq 0}$ defines a unique weighted branched $d$-manifold $T\lambda: Gr_d(TB/\frak{g}) \to \bQ_{\geq 0}$ characterised by the property
\[
T\lambda(x,F) = \sum_{T_x M_i = F} \lambda_i
\]
for any system $\{(M_i,\lambda_i)\}$ of local branches of $\lambda$ at $x$ and any $x\in B$.  We say weighted branched manifolds $\lambda$ and $\lambda'$ are \emph{transverse} if \[
F, F' \subset T_xB \ \mathrm{and} \ T\lambda(x,F) > 0, T\lambda(x,F') > 0 \ \Rightarrow \ F \pitchfork F'  \ \mathrm{transverse \ in} \, T_xB.
\]
There is then an intersection $\lambda\cdot\lambda' : B \to \bQ_{\geq 0}$, weighted branched of the expected dimension $d+d'-dim(B)$.

Now suppose $\pi: E \to B$ is a finite-dimensional fibre bundle, $G$-equivariant for an action with finite stabilisers. A \emph{multivalued section} of $\pi$ is a weighted branched $\sigma: E \to \bQ_{\geq 0}$ which is $G$-equivariant and which satisfies that for every $x \in B$, there is some open $U$, sections $s_1,\ldots, s_m : U \to E$ and rational numbers $\sigma_1,\ldots, \sigma_m \geq 0$ such that
\begin{equation}\label{eqn:multivalued_section}
\sum \sigma_i = 1; \ \sigma(x,e) = \sum_{s_i(x) = e} \sigma_i \ \forall x \in U.
\end{equation}
In this setting, we have a typical transversality theory:
\begin{lemma} \label{lem:mvp_exists}
If $\pi: E \to B$ is a $G$-fibre bundle with a section $s: B \to E$, then $\pi$ admits multivalued sections which are transverse to $s(B)$ and contained in an arbitrarily small $C^0$-neighbourhood of $s(B)$. Moreover, given a locally closed submanifold $Z$ of $B$ and a multivalued section which is transverse to $s(B)$ over $Z$, then there are multivalued sections transverse to $s(B)$ which agree with the given one over $Z$.
\end{lemma}

\begin{proof} See Section 9 of \cite{CMS}. Note that they don't prove that transverse sections are $C^{\infty}$-generic, but do show that there is an abundance of them (roughly, one can build one from any choice of local sections which span the tangent space to $E$ locally). In particular, \cite[Proposition 9.20]{CMS} allows one to extend transverse sections which are already prescribed on closed subsets (e.g. the boundary).
\end{proof}

In particular, when $E$ is a vector bundle, there are multivalued sections which are transverse to the zero-section. Note that in general there is no tranverse $G$-equivariant section.

\begin{lemma}\label{lem:mvp_is_euler}
If $\pi: E\to B$ is a vector bundle carrying an action of $G$ lifting that on $B$, and $s$ is a multivalued section transverse to zero, then the zero-set of $s$ defines a rational singular chain which represents the rational $G$-equivariant Thom class of $E$, equivalently of the Thom class of the orbibundle $E/G \to B/G$. 
\end{lemma}

\begin{proof}
    Direct from properties of the Thom class.
\end{proof}

\begin{remark}\label{rmk:non-global ok}
    We have discussed the theory in terms of equivariant multi-valued sections of $G$-bundles $E \to B$ over a smooth $G$-manifold. However, results of \cite{McDuff06} imply the same results for a general (not necessarily global quotient) orbibundle $\cE \to \cG$ over a smooth orbifold $\cG$, viewed as an \'etale proper groupoid. Indeed,  letting $\cQ^{\geq 0}$ denote the category with one object for each non-negative rational and only identity morphisms, \cite{McDuff06} interprets the multivalued sections of \cite{CMS} as smooth functors $\cE \to \cQ^{\geq 0}$  satisfying a local condition analogous to \eqref{eqn:multivalued_section}, and there is a corresponding Euler chain given by the vanishing locus (which itself inherits the structure of a groupoid).  This extends the theory of multivalued perturbations to general finite-dimensional smooth orbibundles, which is relevant in light of Remark \ref{rmk:not_global_quotient}.
\end{remark}

We have a stabilisation operation: if $(G,\scrT, E,s)$ is a global chart with $\lambda$ a multivalued section of $E \to \scrT$ close to $s$ and transverse to $0$, and $p: F\to \scrT$ a $G$-bundle, we obtain a multivalued section of the stabilised chart $(G,Tot(F), p^*(E \oplus F), s \oplus \Delta)$ by taking $\tilde{\lambda}: p^*(E \oplus F) \to \bQ_{\geq 0}$ to be $p^*\lambda \cdot \chi_{\Delta_F}$, where  $\chi_{\Delta_F}$ is a characteristic function along the diagonal in $p^*(E \oplus F)$.
The multivalued section $\tilde{\lambda}$ is again transverse to $0$. By Remark \ref{rmk:non-global ok}, the same stabilisation and extension of perturbation operation makes sense if $p: F \to \scrT/G$ is an orbibundle which is not a global quotient.

Now take an ordered marked flow category $C$ with a global chart lift. A \emph{perturbation system} for $C$ consists of a collection of 
$I_{pq,\alpha}$-families of
multivalued sections $(\tilde{s}_{pq,\alpha}^{\bfg,m,I}: E_{pq,\alpha}^{\bfg,m} \to \bQ_{\geq 0})$ of the global charts $E_{pq,\alpha}^{\bfg,m} \to \scrT_{pq,\alpha}^{\bfg,m}$ of $\cM_{pq}^{\bfg,m}$ which have image in $C^0$-small neighbourhoods of the given obstruction sections $s_{pq,\alpha}^{\bfg,m,I}: \scrT_{pq,\alpha}^{\bfg,m} \to E_{pq,\alpha}^{\bfg,m}$.   We furthermore insist that the $\tilde{s}_{pq,\alpha}^{\bfg,m,I}$ satisfy
the compatibility with respect to Definition \ref{d:coherenthomotopy}(2)(3) and \ref{d:global lift}(3)(4).
In other words, we insist that
\begin{enumerate}
    \item for any $\alpha \le \gamma$, the $I_{pq,\gamma}$-family of sections $(\tilde{s}_{pq,\alpha}^{\bfg,m,I})|_{\partial_{\gamma}^{pq} I_{pq,\alpha}}$ is a stabilisation of   $\tilde{s}_{pq,\gamma}^{\bfg,m,I}|_{\partial^{\alpha} E_{pq,\gamma}^{\bfg,m}}$
    \item 
for any $\alpha = (p r_1\ldots r_s q,\bfh) \in \bA_{pq}$, $\beta_i \in \bA_{r_i,r_{i+1}}$ for $i=0,\dots,s$ and setting $\beta$ to be the concatenation of $(\beta_i)_{i=0}^s$ refining $\alpha$, the $(I_{\alpha,\beta}=I_{pr_1,\beta_0} \times \dots \times I_{r_sq,\beta_s})$-family of sections $\tilde{s}_{pq,\alpha}^{\bfg,I}|_{\partial_{\beta}^{\alpha}I_{pq,\beta}}$
 is a stabilisation of 
$ \tilde{s}_{pr_1,\beta_0}^{\bfg_0(\alpha),I}\times_{[pt/\Gamma_{r_1}]} \cdots \times_{[pt/\Gamma_{r_s}]} \tilde{s}_{r_sq,\beta_s}^{\bfg_s(\alpha),I}$
\item all the compositions of the restriction of stabilisation maps are isomorphic when the domain and target are the same.
\end{enumerate}

\begin{lemma} \label{l:pertubationexist}
Let $C$ be an ordered marked flow category. Assume that it admits a global chart lift satisfying the conditions of Lemma \ref{lem:smoothing}. Then the lift admits a perturbation system.
\end{lemma}

\begin{proof}
First we recall that a global chart lift of the flow category remembers the data entering into the explicit stabilisations at each boundary, cf. Definition \ref{d:itStab} and Remark \ref{rmk:itStab}, and is equipped with collections of coherent homotopies for all strata, which interpolate the various stabilised obstruction sections, cf. Definition \ref{d:coherenthomotopy} and \ref{d:global lift}. 

The construction of a perturbation system involves two inductions, each of which is in turn an induction over dimension with a fixed number of insertions $k$, and then over the number of insertions.
\begin{enumerate}
    \item We first smooth the global chart lift, which involves additional stabilisations constructed inductively over action and $k$ (since we obtain stable smoothings from Proposition \ref{prop:stable_smoothing}). In this step we do not change the obstruction sections or any data over the homotopy cubes.
    \item We then inductively replace the sections  by transverse multivalued sections. When $\cA_{\mathbb{Z}}(q) - \cA_{\mathbb{Z}}(p)$ is minimal and positive, the poset $\bA_{pq}^0$ is a singleton and $\scrT_{pq}$ has no boundary or corners, so the existence of a multivalued perturbation is exactly Lemma \ref{lem:mvp_exists}. The same Lemma then allows one to extend over adjacent strata. After a perturbation system is chosen for $k=0$, we consider $k=1$ and again run the induction over $\cA_{\mathbb{Z}}(q) - \cA_{\mathbb{Z}}(p)$.
\end{enumerate}

\end{proof}

We can now define the chain complex. Fix a characteristic zero coefficient field $\bk$.

\begin{definition}\label{d:chaincomplex}
Let $\cC$ be an ordered marked  flow category over a poset (and hence with object set) $\cP$ with a smooth global chart lift. The (action filtered) \emph{flow category chain complex}  $(C_*^{\cC})$ of $\cC$ is defined as follows:
\begin{enumerate}
\item Let $C$ be the $\bK$-vector space generated by $\cP$. The free action of $\Pi$ on $\cP$ makes $C$ a $\bK[\Pi]$-module.
There is a ring homomorphism $\bK[\Pi] \to \Lambda_{\Pi}$ given by $a \mapsto T^{\cA_{\Pi}(a)}$.
The chain group $C_*^{\cC}$ is
\[
C_*^{\cC}=C \otimes_{\bK[\Pi]} \Lambda_{\Pi}.
\]
More explicitly, if we pick a representative $p$ for each $[p] \in \cP/\Pi$,
a general element of $C_*^{\cC}$ is of the form
\[
\sum_{[p] \in \cP/\Pi} m_p\cdot p \text{ with } m_p \in \Lambda_\Pi.
\]

\item The action of $T^r p$ is defined to be $\cA(p)-r$. The action of a linear combination of terms of the form $T^rp$ is the maximum of the action of the terms.

\item The differential sends $q \mapsto \sum_p \sum_k \sum_{\bfg \in (\|\Gamma\|^{\circ})^{k}} \frac{n_{pq}^\bfg}{k!} \cdot |\Gamma_p| \cdot p$ where $n_{pq}^\bfg \in \bk$ is the sum of $\tilde{s}_{pq,\alpha}^{\bfg,I}(x_i) \in \bQ$ over the set of zeroes $\{x_i\}$ of the transverse multivalued perturbation $\tilde{s}_{pq,\alpha}^{\bfg,I}$ over all $\alpha$.  
\end{enumerate}
\end{definition}

The transversality imposed in the final point above is where we use  smoothness of the global chart lifts.

\begin{remark}\label{r:action_de}
    Since $\cA(p)<\cA(q)$ if $\cM_{pq} \neq \emptyset$, the differential is action decreasing.
\end{remark}

\begin{remark}
It is important that we take into account the zeroes of $\tilde{s}_{pq,\alpha}^{\bfg,I}$ over the homotopy cube, i.e. we count zeroes over the whole `collaring' $I$, cf. Figure \ref{fig:cube} (itself a cartoon derived from Figure \ref{fig:externalcoherent}).
\end{remark}

\begin{figure}[ht]
\begin{center}
\begin{tikzpicture}

\draw[semithick] (3,-1) -- (-1,-1);
\draw[semithick] (-1,-1) -- (-1,3);
\draw[semithick] (3,0) -- (-1,0);
\draw[semithick] (0,-1) -- (0,3);
\draw[semithick,fill] (-0.5,-0.5) circle (0.1);
\draw[semithick,fill] (-0.5,1.5) circle (0.1);
\draw[semithick,fill] (0.5,1) circle (0.1);
\draw[semithick,fill] (2,2) circle (0.1);

\end{tikzpicture}
\end{center}
\caption{The differential counts zeroes of the interpolating section in the homotopy cube (cf. Figure \ref{fig:externalcoherent}). Black dots represent zeroes.}\label{fig:cube}
\end{figure}

\begin{prop}\label{p:dsquare=0}
The differential squares to zero, and the quasi-isomorphism type of the resulting $(C_*^{\cC})$ is independent of the choice of multivalued perturbations $\tilde{s}_{pq}$.
\end{prop}

\begin{proof} 
The analogous proofs to those in \cite{BX} show both statements, inductively constructing  extensions of multivalued perturbations from boundary strata of charts to the whole charts.  (\cite{BX} use polynomial perturbations to define counts over $\bZ$, which we neither require nor desire.)  Compare also to Theorem 10.12 of \cite{CMS}.

The main difference between the references and our setting is that we use ordered marked flow categories rather than flow categories, so the combinatorial structure of the boundary is different. Indeed, Remark \ref{r:nchooser} corresponds to the fact that for each $k$, 
\[
\sum_{r} \sum_{k_0+k_1=k} \sum_{\bfg_0 \in (\|\Gamma\|^{\circ})^{k_0}, \bfg_1 \in (\|\Gamma\|^{\circ})^{k_1}} {k \choose {k_0}} \, n_{pr}^{\bfg_0} n_{rq}^{\bfg_1}|\Gamma_r|=0
\]
where the factor $|\Gamma_r|$ comes from taking fibre product over $[pt/\Gamma_r]$.
Multiplying the equation by $|\Gamma_p|/k!$, it is precisely the coefficient of $p$ in the differential-squared of $q$. Therefore, we conclude that the differential squares to zero. 
\end{proof}

\begin{remark}\label{r:deform}
We will need a variant of Definition \ref{d:chaincomplex} for our intended applications, cf. Definition \ref{d:chaincomplex2}. 
Let $\mathbf{a}:\|\Gamma\|^{\circ} \to \Lambda_{\ge 0}$ be a function.
For $\bfg=(g_1,\dots,g_k) \in (\|\Gamma\|^{\circ})^{k}$, we define $\bfa^{\bfg}:=\prod_{i=1}^k \bfa(g_i) \in \Lambda_{\ge 0}$.
The $\bfa$-deformed chain complex is defined as in Definition \ref{d:chaincomplex} with $\tilde{s}_{pq,\alpha}^{\bfg,I}$ being replaced by $\bfa^{\bfg}\tilde{s}_{pq,\alpha}^{\bfg,I} \in \Lambda_{\ge 0}$.
The analogue of Proposition \ref{p:dsquare=0} holds for the $\bfa$-deformed chain complex.
\end{remark}

\subsection{Ordered marked flow bimodules}\label{s:module}

\emph{In this section we introduce ordered marked flow bimodules, which arise from spaces of solutions to continuation maps.}

Let $\cP$ and $\cP'$ be countable posets as above with actions $\cA^{\cP}$ and $\cA^{\cP'}$ and integralized actions $\cA_{\mathbb{Z}}^{\cP}$ and $\cA_{\mathbb{Z}}^{\cP'}$ respectively.  We assume that the $\cA_{\Pi}, \cA_{\Pi,\bZ}:\Pi \to \mathbb{Z}$ are common for both $\cP$ and $\cP'$. 
Denote the poset $\bA_{pq}^\bfg$ for $\cP$ by $\bA_{pq}^{\bfg,\cP}$ and the poset $\bA_{p'q'}^\bfg$ for $\cP'$ by $\bA_{p'q'}^{\bfg,\cP'}$.

For $p \in \cP$, $p' \in \cP'$ and $\bfg \in (\|\Gamma\|^{\circ})^{k}$, we define a new poset
\[
\bA_{pp'}^\bfg :=\{\alpha=(pq_1\dots q_lq'_{1}\dots q'_{l'}p',h_1,\dots,h_k) | p<\dots<q_l \in \cP, q'_{1}< \dots<p' \in \cP', h_j \in \{0,\dots,l+l'\}\} 
\]
with partial order defined to be
\[
(ps_1\dots s_ms'_{1}\dots s'_{m'}p',h_{1,s} \dots, h_{k,s})  \le (pr_1\dots r_lr'_{1}\dots r'_{l'}p',h_{1,r}, \dots, h_{k,r}) 
\]
if and only if there are injective monotonic increasing functions $f:\{1,\dots,l\} \to \{1,\dots,m\}$ 
and $f':\{1,\dots,l'\} \to \{1,\dots,m'\}$ 
such that
\begin{enumerate}
\item $r_i=s_{f(i)}$ for all $i \in \{1,\dots,l\}$, $r'_i=s'_{f'(i)}$ for all $i \in \{1,\dots,l'\}$, and
\item $F(h_{j,r}) \le h_{j,s} \le F(h_{j,r}+1)$ where $F:\{0,1,\dots,l+l'+1\} \to \{1,\dots,m+m'\}$ is defined by 
\[
    F(x)= 
\begin{cases} 
0 & \text {if } x=0\\
    f(x)& \text{if } x \in \{1,\dots,l\}\\
    f'(x-l)+m   & \text{if } x \in \{l+1,\dots,l+l'\}\\
    m+m' &\text{if } x=l+l'+1
\end{cases}
\]
\end{enumerate}
As in Remark \ref{r:compatibleH}, the conditions on $h_j$'s can be understood geometrically as insisting that the $j^{th}$ ordered orbifold marked point lies in the correct broken cylinder, and $X^{g_j}/C(g_j)$ is the twisted sector to which the marked point maps.

There is a maximal element $(pp',\mathbf{0})$ and $\bA^\bfg_{pp'}$ is homogeneous, so elements have a well-defined `depth' (where $(pp',\mathbf{0})$ is the unique element of depth $0$). There are isomorphisms of homogeneous posets: for $\alpha=(pqp',\bfh)$, $q \in \cP$ and $q' \in \cP'$, we have
\[
\bA^{\bfg_0(\alpha),\cP}_{pq} \times \bA_{qp'}^{\bfg_1(\alpha)} \cong \partial^{\alpha} \bA_{pp'}^\bfg 
\qquad \qquad \bA_{pq'}^{\bfg_0(\alpha)} \times \bA^{\bfg_1(\alpha),\cP'}_{q'p'} \cong \partial^{\alpha} \bA_{pp'}^\bfg.
\]
Let $\bA_{pp'}^k= \sqcup_{\bfg \in (\|\Gamma\|^{\circ})^{k}} \bA^\bfg_{pp'}$ and $\bA_{pp'}= \sqcup_k \bA_{pp'}^k$.

Suppose we have ordered marked flow categories $\cC$ and $\cC'$ over $\cP$ and $\cP'$ respectively. 
Denote the morphism space between $p,q \in \cP$ by $\cM_{pq}$, and morphism space between $p',q' \in \cP$ by $\cM_{p'q'}$.
A flow bimodule $\scrB$ from $\cC$ to $\cC'$ consists of the following data (cf. \cite[Definition 3.10]{BX}): 
for any $p \in \cP$ and $p' \in \cP'$, a compact $\bA_{pp'}$-orbispace $\scrB_{pp'}= \sqcup \scrB_{pp'}^k$ together with evaluation maps
\[
\scrB_{pp'} \to [pt/\Gamma_{p}], \text{ and }\scrB_{pp'} \to [pt/\Gamma_{p'}]
\]
and stratified orbispace isomorphisms (for $p,q \in \cP$ and $p',q' \in \cP'$)
\[
\centerline{
\xymatrix{
\cM_{pq}^{\bfg_0(pqp',\bfh)} \times_{[pt/\Gamma_q]} \scrB_{qp'}^{\bfg_1(pqp',\bfh)} \ar[r] \ar[d] & \partial^{(pqp',\bfh)} \scrB_{pp'}^\bfg \ar[d] \\
\bA^{\bfg_0(pqp',\bfh),\cP}_{pq} \times \bA_{qp'}^{\bfg_1(pqp',\bfh)} \ar[r] & \partial^{(pqp',\bfh)} \bA_{pp'}^\bfg
}
 \quad 
\xymatrix{
 \scrB_{pq'}^{\bfg_0(pq'p',\bfh)} \times_{[pt/\Gamma_{q'}]} \cM_{q'p'}^{\bfg_1(pq'p',\bfh)} \ar[r] \ar[d] & \partial^{(pq'p',\bfh)} \scrB_{pp'}^\bfg \ar[d] \\
 \bA_{pq'}^{\bfg_0(pq'p',\bfh)}\times \bA^{\bfg_1(pq'p',\bfh),\cP'}_{q'p'} \ar[r] & \partial^{(pq'p',\bfh)} \bA_{pp'}^\bfg
}
}
\]
such that
\begin{align}\label{eq:monoaction}
\scrB_{pp'}\neq \emptyset \Rightarrow \cA^{\cP}(p) < \cA^{\cP'}(p') \text{ and }
\cA_{\mathbb{Z}}^{\cP}(p) < \cA_{\mathbb{Z}}^{\cP'}(p')
\end{align}
and that there are strictly associative boundary face maps (the $\bfg$'s in the superscripts are omitted for simplicity)
\begin{align*}
\xymatrix{
\scrB_{pq_1'}\times_{[pt/\Gamma_{q_1'}]} \cM_{q_1'q_2'} \times_{[pt/\Gamma_{q_2'}]} \cM_{q_2'p'} \ar[r] \ar[d] & \scrB_{pq_1'}\times_{[pt/\Gamma_{q_1'}]} \cM_{q_1'p'} \ar[d] \\
\scrB_{pq_2'}\times_{[pt/\Gamma_{q_2'}]} \cM_{q_2'p'}  \ar[r] & \scrB_{pp'}
}
\end{align*}
\begin{align*}
\xymatrix{
\cM_{pq_1}\times_{[pt/\Gamma_{q_1}]} \cM_{q_1q_2} \times_{[pt/\Gamma_{q_2}]} \scrB_{q_2p'} \ar[r] \ar[d] & \cM_{pq_1}\times_{[pt/\Gamma_{q_1}]} \scrB_{q_1p'} \ar[d] \\
\cM_{pq_2}\times_{[pt/\Gamma_{q_2}]} \scrB_{q_2p'}  \ar[r] & \scrB_{pp'}
}
\end{align*}
\begin{align*}
\xymatrix{
\cM_{pq}\times_{[pt/\Gamma_q]} \scrB_{qq'} \times_{[pt/\Gamma_{q'}]} \cM_{q'p'} \ar[r] \ar[d] & \cM_{pq}\times_{[pt/\Gamma_q]} \scrB_{qp'} \ar[d] \\
\scrB_{pq'}\times_{[pt/\Gamma_{q'}]} \cM_{q'p'}  \ar[r] & \scrB_{pp'}
}
\end{align*}
We also need a continuous map \[\vdim: \scrB_{pp'} \to \mathbb{Z}\] 
so that the analogue of \eqref{eq:vdimsum} holds.
We further insist that $\Pi$ acts strictly, via stratified homeomorphisms 
\[
\xymatrix{
\scrB_{pp'} \ar[r] \ar[d] & \scrB_{ap,ap'} \ar[d] \\
\bA_{pp'}  \ar[r] & \bA_{ap,ap'}
}
\]
for $a\in \Pi$ and $p,q\in \cP$.

\begin{remark}\label{r:shift Ham}
 When it comes to Hamiltonian Floer theory,  \eqref{eq:monoaction} (and its later cousin \eqref{eq:monoaction2}) is achieved by considering a monotone family of Hamiltonian functions. At the cost of shifting one of the two Hamiltonians by a constant, which has the effect of changing the action (and integralized action) by a constant, we are able to construct a monotone family from one to the other.
\end{remark}

We also need to consider a $1$-parameter family of flow bimodules. The presentation below is a slight modification of \cite[Section 4.6.1]{BX}.
Let $\scrB$ and $\scrB'$ be two flow bimodules from $\cM$ to $\cM'$.
Denote the poset which $\scrB_{pp'}$ is modelled on by $\bA_{pp'}^{\scrB}=\bA_{pp'}$, and similarly for $\scrB_{pp'}'$. 
Let $\bA^{hmtp}_{pp'}= \bA_{pp'}^{\scrB} \sqcup \bA_{pp'}^{hmtp, \circ} \sqcup \bA_{pp'}^{\scrB'}$, where $\bA_{pp'}^{hmtp, \circ}=\bA_{pp'}$.
The partial ordering on $\bA^{hmtp}_{pp'}$ is generated the partial orderings on $\bA_{pp'}^{\scrB}$, $\bA_{pp'}^{hmtp, \circ}$ and $\bA_{pp'}^{\scrB}$ individually together with $\alpha^{hmtp, \circ} > \alpha^{\scrB}$ and 
$\alpha^{hmtp, \circ} > \alpha^{\scrB'}$ if $\alpha^{\scrB}, \alpha^{hmtp, \circ},  \alpha^{\scrB'}$ are $\alpha$'s in 
$\bA_{pp'}^{\scrB}$, $\bA_{pp'}^{hmtp, \circ}$ and $\bA_{pp'}^{\scrB'}$ respectively.
With respect to this partial ordering, there is a unique maximal element $\alpha=(pp', \mathbf{0}) \in \alpha^{hmtp, \circ}$.
Let $\partial^+ \bA^{hmtp}_{pp'}=\bA_{pp'}^{\scrB}$ and $\partial^- \bA^{hmtp}_{pp'}=\bA_{pp'}^{\scrB'}$.
For any $\alpha \in \bA_{pp'}^{hmtp, \circ}$, $\partial^{\alpha} \bA^{hmtp}_{pp'}$ consists of elements that are less than or equal to $\alpha$ so we have $\partial^{\alpha} \bA^{hmtp}_{pp'}=\partial^{\alpha} \bA_{pp'}^{\scrB} \sqcup \partial^{\alpha} \bA_{pp'}^{hmtp, \circ} \sqcup \partial^{\alpha} \bA_{pp'}^{\scrB'}$.
For example, when $\alpha=(pqp',\bfh)$, we have 
\begin{align*}
\partial^{\alpha} \bA^{hmtp}_{pp'}&=\partial^{\alpha} \bA_{pp'}^{\scrB} \sqcup \partial^{\alpha} \bA_{pp'}^{hmtp, \circ} \sqcup \partial^{\alpha} \bA_{pp'}^{\scrB'}\\
&=\bA_{pq}^{\bfg_0(\alpha),\cP} \times  \bA_{qp'}^{\bfg_1(\alpha),\scrB} \sqcup \bA_{pq}^{\bfg_0(\alpha),\cP} \times  \bA_{qp'}^{\bfg_1(\alpha),hmtp, \circ} \sqcup \bA_{pq}^{\bfg_0(\alpha),\cP} \times  \bA_{qp'}^{\bfg_1(\alpha),\scrB'}\\
&=\bA_{pq}^{\bfg_0(\alpha),\cP} \times  (\bA_{qp'}^{\bfg_1(\alpha),\scrB} \sqcup  \bA_{qp'}^{\bfg_1(\alpha),hmtp, \circ} \sqcup  \bA_{qp'}^{\bfg_1(\alpha),\scrB'})\\
&=\bA_{pq}^{\bfg_0(\alpha),\cP} \times \bA^{\bfg_1(\alpha),hmtp}_{qp'}
\end{align*}
and similarly when $\alpha=pq'p'$.

A homotopy of flow bimodules $\scrB^{hmtp}$ from $\scrB$ and $\scrB'$ consists of the following data:
for any $p,q \in \cP$ and $p',q' \in \cP'$, a compact $\bA^{hmtp}_{pp'}$-manifold $\scrB^{hmtp}_{pp'}$ together with stratified homeomorphisms 
\[
\xymatrix{
\scrB_{pp'} \ar[r] \ar[d] & \partial^{+} \scrB^{hmtp}_{pp'} \ar[d] \\
\bA_{pp'}^{\scrB} \ar[r] & \partial^{+} \bA^{hmtp}_{pp'}
}
 \quad 
\xymatrix{
\scrB'_{pp'} \ar[r] \ar[d] & \partial^{-} \scrB^{hmtp}_{pp'} \ar[d] \\
\bA_{pp'}^{\scrB'} \ar[r] & \partial^{-} \bA^{hmtp}_{pp'}
}
\]
and
\[
\xymatrix{
\cM_{pq} \times_{[pt/\Gamma_q]} \scrB^{hmtp}_{qp'} \ar[r] \ar[d] & \partial^{pqp'} \scrB^{hmtp}_{pp'} \ar[d] \\
\bA^{\cP}_{pq} \times \bA^{hmtp}_{qp'} \ar[r] & \partial^{pqp'} \bA^{hmtp}_{pp'}
}
 \quad 
\xymatrix{
 \scrB^{hmtp}_{pq'}\times_{[pt/\Gamma_{q'}]} \cM_{q'p'} \ar[r] \ar[d] & \partial^{pq'p'} \scrB^{hmtp}_{pp'} \ar[d] \\
 \bA^{hmtp}_{pq'}\times \bA^{\cP'}_{q'p'} \ar[r] & \partial^{pq'p'} \bA^{hmtp}_{pp'}
}
\]
 such that 
\[
\scrB^{hmtp}_{pp'} \neq \emptyset \Rightarrow \cA^{\cP}(p) < \cA^{\cP'}(p')\text{ and }
\cA_{\mathbb{Z}}^{\cP}(p) < \cA_{\mathbb{Z}}^{\cP'}(p')
\]
and that there are strictly associative boundary face maps which we omit.
There is a continuous function $\vdim: \scrB^{hmtp}_{pp'} \to \mathbb{Z}$ which satisfies the analogue of \eqref{eq:vdimsum}.

Given a global chart lift of $\cC$ and $\cC'$, there is a natural notion of a (homotopy coherent) global chart lift for a flow bimodule from $\cC$ to $\cC'$ (generalizing Definition \ref{d:coherenthomotopy} and \ref{d:global lift}): this comprises global chart lifts for all the moduli spaces $\scrB_{pp'}$, of  virtual dimension $\vdim$,  which on boundary strata are homotopic to stabilisations of fibre product global charts for the factors $\cM_{pq}^{\bfg_0} \times_{[pt/\Gamma_q]} \scrB_{qp'}^{\bfg_1}$ or $\scrB_{pq'}^{\bfg_0}\times_{[pt/\Gamma_{q'}]} \cM_{q'p'}^{\bfg_1}$. A global chart lift of a homotopy $\scrB^{hmtp}$ of flow bimodules from $\scrB$ to $\scrB'$ is defined similarly.

We will need the following statement, which follows from a direct analogue of the arguments from Section \ref{s:multivalue} and Proposition \ref{p:dsquare=0}, and  whose proof will be omitted.
\begin{prop}
    A global chart lift for a flow bimodule $\scrB$ from $\cC$ to $\cC'$ induces an action decreasing chain map $\tau_{\scrB}: C^{\cC'}_* \to C^{\cC}_*$.
\end{prop}

\subsection{Two-simplices of bimodules}

\emph{Ordered marked flow bimodules cannot be composed so we introduce $2$-simplices of ordered marked flow bimodules to capture composition of Floer continuation maps. An $n$-simplex of flow bimodules was introduced in \cite[Section 4]{Abouzaid-Blumberg}.} We only require $2$-simplices since we require a construction at the cohomological rather than $A_{\infty}$-level.


We fix three posets $\cP,\cP', \cP''$ with their corresponding actions, integralized actions such that $\cA_{\Pi}$
and $\cA_{\Pi,\mathbb{Z}}$ are common for all three of them.
Let $\bfg \in (\|\Gamma\|^{\circ})^k$.
For $p \in \cP, p' \in \cP'$ and $p'' \in \cP''$, we denote the posets $\bA_{pp'}^{\bfg}, \bA_{pp''}^{\bfg}, \bA_{p'p''}^{\bfg}$ by $\bA_{pp'}^{\bfg,(\cP, \cP')}, \bA_{pp''}^{\bfg,(\cP,\cP'')}, \bA_{p'p''}^{\bfg,(\cP',\cP'')}$ respectively.

Let $\Vec{\cP}=(\cP,\cP', \cP'')$. For $p \in \cP$ and $p'' \in \cP''$, we define a poset
\[
(\bA_{p;p''}^{\bfg,\vec{\cP}})^{main}= \left\{ \alpha  = (p r_1\ldots r_l,  r_1' r_2'\ldots r_{l'}', r_1'' \ldots r''_{l''} p'',h_1,\ldots,h_k) \right \}  
\]
where we insist that $h_j \in \{0, \dots, l+l'+l''\}$ and
\[
 p<r_1 < \ldots < r_l \in \cP,  \quad r_1' < \ldots < r'_{l'} \in \cP', \quad r''_1 < \ldots < r''_{l''} < p'' \in \cP''.
\]
for some $l,l',l'' \ge 0$.

The partial order  $\leq$ is again given by inclusions of the triple of sequences $\{r_i\}, \{r_i'\}, \{r_i''\}$ and the compatibility with the $h_j$'s. Thus $(pp'',\mathbf{0})$ is the unique maximal element. Note that the sub-poset $(\bA_{p;p''}^{\bfg,\vec{\cP}})^{main}_{l' = 0}$ consisting of elements with $l'=0$ is canonically isomorphic to $\bA_{pp''}^{\bfg,(\cP,\cP'')}$.

Then we define another poset
\[
\bA_{p;p''}^{\bfg,\vec{\cP}}:=(\bA_{p;p''}^{\bfg,\vec{\cP}})^{main} \cup \bA_{p;p''}^{\bfg,(\cP,\cP'')}
\]
The partial order on $\bA_{p;p''}^{\bfg,\vec{\cP}}$ restricted to $(\bA_{p;p''}^{\bfg,\vec{\cP}})^{main}$ is the partial ordering on $(\bA_{p;p''}^{\bfg,\vec{\cP}})^{main}$.
Elements of $(\bA_{p;p''}^{\bfg,\vec{\cP}})^{main}$ with $l'>0$ are not comparable to elements of $\bA_{p;p''}^{\bfg,(\cP,\cP'')}$ with respect to the partial ordering.
An element $\alpha \in (\bA_{p;p''}^{\bfg,\vec{\cP}})^{main}_{l'=0}$ is bigger than $\beta \in \bA_{p,p''}^{\bfg,(\cP,\cP'')}$ if and only if, under the isomorphism  $(\bA_{p;p''}^{\bfg,\vec{\cP}})^{main}_{l'=0} \simeq \bA_{p,p''}^{\bfg,(\cP,\cP'')}$, $\alpha$ is bigger than or equal to $\beta$.

\begin{remark}
  To compare it with \cite{Abouzaid-Blumberg}[Definition 4.2], the elements $p, r, r', \dots$ correspond to edges of their directed arcs. Since we have the action functions, we don't need the energy labelling $\lambda$ of their directed arcs. The subset labels at the vertices of their directed arcs remember which family of continuation data one should put in the end.
  In our case, we only consider two-simplices (i.e. $n=2$) so the only possible labelling is $\{1\}$ or $\emptyset$ when $l'=0$. The one corresponding to the label $\{1\}$ is $(\bA_{p;p''}^{\bfg,\vec{\cP}})^{main}_{l'=0}$ and the other one corresponds to $\bA_{p;p''}^{\bfg,(\cP,\cP'')}$.
  Geometrically, $\alpha \in (\bA_{p;p''}^{\bfg,\vec{\cP}})^{main}_{l'=0}$ corresponds to a stratum over which we consider a one-parameter family of continuation data interpolating the broken one with the one over the corresponding $\alpha \in \bA_{p;p''}^{\bfg,(\cP,\cP'')}$.
\end{remark}

To distinguish elements in $(\bA_{p;p''}^{\bfg,\vec{\cP}})^{main}_{l'=0}$ and in $\bA_{p;p''}^{\bfg,(\cP,\cP'')}$, we use $(\alpha,\{1\})$ and $(\alpha,\emptyset)$ to denote the corresponding $\alpha$ in $(\bA_{p;p''}^{\bfg,\vec{\cP}})^{main}_{l'=0}$ and in $\bA_{p;p''}^{\bfg,(\cP,\cP'')}$, respectively.

Similarly, the poset $\bA_{p;p''}^{\bfg,\vec{\cP}}$ has a well-defined depth function. The unique maximal element is $(pp'',\mathbf{0}, \{1\})$. Let $\bA_{pp''}^{k,\vec{\cP}}= \sqcup_{\bfg \in (\|\Gamma\|^{\circ})^{k}} \bA^{\bfg,\vec{\cP}}_{pp''}$ and $\bA_{pp''}^{\vec{\cP}}= \sqcup_k \bA_{pp''}^{k,\vec{\cP}}$.
\begin{example}\label{e:depth_1}
Elements of depth one are of the form $(prp'',\bfh, \{1\})$, $(pr''p'',\bfh, \{1\})$, $(pr'p'',\bfh)$
and $(pp'',\mathbf{0}, \emptyset)$, where $r \in \cP$, $r' \in \cP'$ and $r'' \in \cP''$, respectively.    
\end{example}

Suppose we have ordered marked flow categories $\cC$, $\cC'$ and $\cC''$ over $\cP$, $\cP'$ and $\cP''$ respectively. 
A two-simplex of flow bimodule $\scrB^{\vec{\cC}}$ for $\vec{\cC}:=(\cC,\cC',\cC'')$ consists of the following data: 
flow bimodules $\scrB^{(\cC,\cC')}$,  $\scrB^{(\cC,\cC'')}$ and $\scrB^{(\cC',\cC'')}$ between respective flow categories indicated by the superscripts, and for any $p \in \cP$, $p'' \in \cP''$, a compact $\bA_{pp''}^{\vec{\cP}}$-orbispace $\scrB_{pp''}^{\vec{\cC}}$
together with evaluation maps
\[
\scrB_{pp''}^{\vec{\cC}}\to [pt/\Gamma_{p}], \text{ and }\scrB_{pp''}^{\vec{\cC}} \to [pt/\Gamma_{p''}]
\]
and stratified orbispace isomorphisms (for $p,q \in \cP$, $p' \in \cP'$, and $p'',q'' \in \cP''$)
\[
\centerline{
\xymatrix{
\cM_{pq}^{\bfg_0(pqp'',\bfh)} \times_{[pt/\Gamma_q]} \scrB_{qp''}^{\bfg_1(pqp',\bfh), \vec{\cC}} \ar[r] \ar[d] & \partial^{(pqp'',\bfh)} \scrB_{pp''}^{\bfg,\vec{\cC}} \ar[d] \\
\bA^{\bfg_0(pqp'',\bfh),\cP}_{pq} \times \bA_{qp'''}^{\bfg_1(pqp'',\bfh), \vec{\cP}} \ar[r] & \partial^{(pqp'',\bfh)} \bA_{pp''}^{\bfg, \vec{\cP}}
}
 \quad 
\xymatrix{
 \scrB_{pq''}^{\bfg_0(pq''p'',\bfh), \vec{\cC}} \times_{[pt/\Gamma_{q''}]} \cM_{q''p''}^{\bfg_1(pq''p'',\bfh)} \ar[r] \ar[d] & \partial^{(pq''p'',\bfh)} \scrB_{pp''}^{\bfg,\vec{\cC}} \ar[d] \\
 \bA_{pq''}^{\bfg_0(pq''p'',\bfh), \vec{\cP}}\times \bA^{\bfg_1(pq''p'',\bfh),\cP''}_{q''p''} \ar[r] & \partial^{(pq''p'',\bfh)} \bA_{pp''}^{\bfg,\vec{P}}
}
}
\]
\[
\centerline{
\xymatrix{
\scrB_{pp'}^{\bfg_0(pp'p'',\bfh), (\cC,\cC')}  \times_{[pt/\Gamma_{p'}]} \scrB_{p'p''}^{\bfg_1(pp'p'',\bfh), (\cC',\cC'')} \ar[r] \ar[d] & \partial^{(pp'p'',\bfh)} \scrB_{pp''}^{\bfg,\vec{\cC}} \ar[d] \\
\bA^{\bfg_0(pp'p'',\bfh),(\cP,\cP')}_{pp'} \times \bA_{p'p''}^{\bfg_1(pp'p'',\bfh), (\cP',\cP'')} \ar[r] & \partial^{(pp'p'',\bfh)} \bA_{pp''}^{\bfg, \vec{\cP}}
}}
\]
and equalities
\[
 \centerline{
\xymatrix{
 \scrB_{pp''}^{\bfg, (\cC,\cC'')}  \ar[r]^{\equiv} \ar[d] & \partial^{(pp'',\mathbf{0},\emptyset)} \scrB_{pp''}^{\bfg, \vec{\cC}} \ar[d] \\
 \bA_{pp''}^{\bfg, (\cP,\cP'')} \ar[r]^{\equiv} & \partial^{(pp'',\mathbf{0},\emptyset)} \bA_{pp''}^{\bfg, \vec{\cP}}
}
}
\]
such that
\begin{align}\label{eq:monoaction2}
\scrB_{pp''}^{\vec{\cC}} \neq \emptyset \Rightarrow \cA^{\cP}(p) < \cA^{\cP''}(p'') \text{ and }
\cA_{\mathbb{Z}}^{\cP}(p) < \cA_{\mathbb{Z}}^{\cP''}(p'')
\end{align}
and that there are strictly associative boundary face maps similar as before.
We also need a continuous function $\vdim: \scrB_{pp''}^{\vec{\cC}} \to \mathbb{Z}$ and that $\Pi$ acts strictly to $\{\scrB_{pp''}^{\vec{\cC}}: p \in \cP, p'' \in \cP''\}$ as before.

Given global chart lifts of $\cC$, $\cC'$, $\cC''$, $\scrB^{(\cC,\cC')}$, $\scrB^{(\cC,\cC'')}$ and $\scrB^{(\cC',\cC'')}$,  a global chart lift for $\scrB^{\vec{\cC}}$ is a collection of global charts for all the spaces $\scrB^{\vec{\cC}}_{pp''}$ such that every boundary stratum admits an iterative stabilisation which is homotopic to stabilisations of fibre product charts coherently (cf. Definition \ref{d:coherenthomotopy} and \ref{d:global lift}).


\begin{prop}\label{p:product}
    A global chart lift for $\scrB^{\vec{\cP}}$, and choice of perturbation system, induces a chain homotopy $\tau_{\scrB^{\vec{\cP}}}:C^{\cC''} \to C^{\cC}$ from $\tau_{\scrB^{(\cP,\cP')}} \circ \tau_{\scrB^{(\cP',\cP'')}}$ to $\tau_{\scrB^{(\cP,\cP'')}}$.
\end{prop}

\begin{proof}
    This follows the same steps as used in Section \ref{s:multivalue} and Proposition \ref{p:dsquare=0}, cf. \cite[Proposition B.1(3)]{Abouzaid-Blumberg} for the appearance of 2-simplices of bimodules in establishing the corresponding invariance statement in the smooth Hamiltonian Floer setting. 
    Indeed, the third and fourth depth-one strata in Example \ref{e:depth_1} correspond to $\tau_{\scrB^{(\cP,\cP')}} \circ \tau_{\scrB^{(\cP',\cP'')}}$ and $\tau_{\scrB^{(\cP,\cP'')}}$ respectively.
\end{proof}

\subsection{Bilinear maps}

\emph{The usual Floer product is defined by a bilinear map of flow categories; here we introduce the corresponding notion in the ordered marked situation.}

We fix three posets $\cP,\cP', \cP''$ and $\bfg \in (\|\Gamma\|^{\circ})^k$ as before.
We define a poset
\[
\bA_{p;p'p''}^\bfg = \left\{ \alpha  = (p r_1\ldots r_l,  r_1' r_2'\ldots r_{l'}'p', r_1'' \ldots r''_{l''} p'',h_1,\ldots,h_k) \right \} 
\]
where we insist that $h_j \in \{0, \dots, l+l'+l''\}$ and
\[
 p<r_1 < \ldots < r_l \in \cP,  \quad r_1' < \ldots < r'_{l'}<p' \in \cP', \quad r''_1 < \ldots < r''_{l''} < p'' \in \cP''.
\]
The partial order  $\leq$ is again given by inclusions of the triple of sequences $\{r_i\}, \{r_i'\}, \{r_i''\}$ and the compatibility with the $h_j$'s. Thus $(pp'p'',\mathbf{0})$ is the unique maximal element (depth zero)  and the poset is homogeneous. 
We have isomorphisms
\begin{align*}
\bA_{p\tilde{p}}^{\bfg_0(p\tilde{p}p'p'',\bfh),\cP} \times \bA_{\tilde{p};p'p''}^{\bfg_1(p\tilde{p}p'p'',\bfh)} \cong \partial^{(p\tilde{p}p'p'',\bfh)} \bA_{p;p'p''}^\bfg \\
  \bA_{p; \tilde{p}'p''}^{\bfg_0(p\tilde{p}'p'p'',\bfh)} \times \bA_{\tilde{p'}p'}^{\bfg_1(p\tilde{p}'p'p'',\bfh), \cP'}\cong \partial^{(p\tilde{p}'p'p'',\bfh)} \bA_{p;p'p''}^\bfg \\ 
\bA_{p;p'\tilde{p}''}^{\bfg_0(pp'\tilde{p}''p'',\bfh)} \times \bA_{\tilde{p}''{p}''}^{\bfg_1(pp'\tilde{p}''p'',\bfh),\cP''} \cong \partial^{(pp'\tilde{p}''p'',\bfh)} \bA_{p;p'p''}^\bfg
\end{align*}
We assume the posets have a common translation group $\Pi$ and action $\cA_{\Pi}: \Pi \to \bZ$.

Let $\cC, \cC', \cC''$ be flow categories  which live over the posets $\cP, \cP', \cP''$.  A \emph{bilinear map} $\cR: \cC' \times \cC'' \to \cC$ consists of a compact $\bA_{p;p'p''}$-spaces $\cR_{p;p'p''}$ for $p\in \cC, p' \in \cC', p'' \in \cC''$ with stratified homeomorphisms
\[
\xymatrix{
\cC_{p\tilde{p}} \times_{[pt/\Gamma_{\tilde{p}}]} \cR_{\tilde{p};p'p''} \ar[r] \ar[d] & \partial^{p\tilde{p}p'p''} \cR_{p;p'p''} \ar[d] \\
\bA_{p\tilde{p}}^{\cP} \times \bA_{\tilde{p};p'p''} \ar[r] & \partial^{p\tilde{p}p'p''} \bA_{p;p'p''} 
}
\]
and similarly for the other poset boundaries, satisfying the obvious associativity conditions analogous to those of a morphism of flow categories.  We insist that
\begin{align}\label{eq:action_triangle}
\cR_{p;p'p''} \neq \emptyset \ \Rightarrow \ 
\cA^{\cP}(p) < \cA^{\cP'}(p') + \cA^{\cP''}(p'') \text{ and }
\cA_{\mathbb{Z}}^{\cP}(p) < \cA_{\mathbb{Z}}^{\cP'}(p') + \cA_{\mathbb{Z}}^{\cP''}(p'').
\end{align}
and that the spaces carry a strict $\Pi$-action by stratified homeomorphisms and admit a continuous function $\vdim$ to $\bZ$ such that the analogue of \eqref{eq:vdimsum} is satisfied.

Given global chart lifts of $\cC, \cC'$ and $\cC''$,  a global chart lift for a bilinear map is a collection of global charts for all the spaces $\cR_{p;p'p''}$ such that every boundary stratum admits an iterative stabilisation which is homotopic to stabilisations of fibre product charts coherently (cf. Definition \ref{d:coherenthomotopy} and \ref{d:global lift}).

We will need the following statement whose proof will be omitted (cf. Section \ref{s:multivalue} and Proposition \ref{p:dsquare=0}).
\begin{prop}\label{p:product}
    A global chart lift for a bilinear map $\cC'\times \cC'' \to \cC$ induces a product map $C^{\cC'} \otimes C^{\cC''} \to C^{\cC}$ (i.e. a bilinear map which commutes with the differentials).
\end{prop}

Given global chart lifts of ordered marked flow categories $\cC_i$, for $i=1,2,3,4$, over $\cP_i$, global chart lifts of 
bimodules $\scrB$ from $\cC_1$ to  $\cC_2$, and global chart lifts of 
bilinear maps $\cC_3 \times \cC_4 \to \cC_2$ and $\cC_3 \times \cC_4 \to \cC_1$, we need a similar notion of poset and its global chart lift to be able to say that the composition
\begin{align}\label{eq:composed_p1}
C^{\cC_3} \otimes C^{\cC_4} \to C^{\cC_2} \to C^{\cC_1}
\end{align}
is chain homotopic to 
\begin{align}\label{eq:composed_p2}
C^{\cC_3} \otimes C^{\cC_4} \to C^{\cC_1}.
\end{align}
The poset is
\begin{align*}
\bA_{p_1;p_3p_4}^{\bfg,(\cP_1,\cP_2;\cP_3,\cP_4)} :=& (\bA_{p_1;p_3p_4}^{\bfg,(\cP_1,\cP_2;\cP_3,\cP_4)})^{main} \sqcup \bA_{p_1;p_3p_4}^{\bfg,(\cP_1;\cP_3,\cP_4)} \\
(\bA_{p_1;p_3p_4}^{\bfg,(\cP_1,\cP_2;\cP_3,\cP_4)})^{main} :=& \left\{ \alpha  = (p_1 r_{1,1}\ldots r_{l_1,1}, r_{1,2} \ldots r_{l_2,2},  r_{1,3} r_{2,3}\ldots r_{l_3,3}p_3, r_{1,4} \ldots r_{l_4,4} p_4,h_1,\ldots,h_k) \right \}
\end{align*}
where $\bA_{p_1;p_3p_4}^{\bfg,(\cP_1;\cP_3,\cP_4)}$ is the poset governing the bilinear map $\cC_3 \times \cC_4 \to \cC_1$, for the definition of $(\bA_{p_1;p_3p_4}^{\bfg,(\cP_1,\cP_2;\cP_3,\cP_4)})^{main}$, we require that $l_i \ge 0$ for $i=1,2,3,4$, $h_i \in \{0,\dots,l_1+l_2+l_3+l_4\}$ for all  $i$, and 
\[
 p_1< \ldots < r_{l_1,1} \in \cP_1,  \quad r_{1,2} < \ldots < r_{l_2,2} \in \cP_2, \quad r_{1,j} < \ldots < r_{l_j,j} < p_j \in \cP_j.
\]
for $j=3,4$. 

Given spaces $\cR_{p_1;p_3p_4}^{\bfg,(\cP_1,\cP_2;\cP_3,\cP_4)}$ lying over the poset $\bA_{p_1;p_3p_4}^{\bfg,(\cP_1,\cP_2;\cP_3,\cP_4)}$ and satisfying the corresponding conditions as above, we can define their global chart lifts as before. The consequence of having a global chart lift of $\cR^{\bfg,(\cP_1,\cP_2;\cP_3,\cP_4)}$ is the following:

\begin{prop}\label{p:product_compose}
    A global chart lift for a bilinear map of $\cR^{\bfg,(\cP_1,\cP_2;\cP_3,\cP_4)}$ induces a chain homotopy between the compositions \eqref{eq:composed_p1} and \eqref{eq:composed_p2}.
\end{prop}


Chain homotopies between other compositions can be obtained similarly.

\section{Orbifold Hamiltonian Floer cohomology}\label{s:Hamiltonian_orbifold}

In this section we define bulk-deformed orbifold Hamiltonian Floer cohomology. We follow the general template from Bai-Xu \cite{BX} (and borrow a construction of $2$-forms to define integralised actions from Rezchikov \cite{Rez}), but many details differ throughout.  We associate to a smooth time-dependent function $H/\Gamma$ on $Y = [X/\Gamma]$ a marked ordered flow category; the orbifold Floer complex is then constructed as in Section  \ref{s:multivalue}, by equipping that flow category  with a homotopy-coherent global chart lift. Invariance of the chain complex is discussed in Section \ref{s:system} and \ref{s:Haminv}, in particular Proposition \ref{p:chaincomplex2indep}, \ref{p:orbFloer} and the proof of Theorem \ref{t:main}(1). \emph{A posteriori} we learn that the flow category itself is an invariant of $Y$, and only the global chart lift depends on the presentation of $Y$ as a global quotient (see Remark \ref{r:Yobjects}, \ref{r:Ymorphism} and Proposition \ref{p:chaincomplex2indep}).

\subsection{The Hamiltonian Floer ordered marked flow category for global quotients}\label{s:orbifoldFloer}

\emph{We set up the bulk-deformed Hamiltonian Floer ordered marked flow category.  Since it suffices for our applications in \cite{MSS2}, we impose the simplifying hypothesis that the bulk $\mathfrak{b}$ is a linear combination of multiples of the cohomological units of the non-trivial sectors of $IY$.}

Let $Y=[X/\Gamma]$ and $H_X:=H:X \times S^1 \to \R$ be a $\Gamma$-invariant Hamiltonian function.  We will write $H_Y:=H/\Gamma: Y \times S^1 \to \bR$ for the associated (smooth) Hamiltonian function $H/\Gamma$ on the orbifold quotient. We assume that $\Gamma$ acts effectively, so $Y$ is an effective orbifold.
We also assume that $H/\Gamma$ is a non-degenerate Hamiltonian function (i.e. $1$ is not an eigenvalue of the linearised return map of any $1$-periodic orbit of $H/\Gamma$).
Since $\Gamma$ is a finite group, a generic $\Gamma$-invariant Hamiltonian $H$ gives a non-degenerate $H/\Gamma$ \cite{Was}\footnote{A $1$-periodic orbit of $H/\Gamma$ lifts to a time-$1$ chord of $H$ so the non-degeneracy condition of $H/\Gamma$ is different from the usual non-degeneracy condition of $H$.}. Our goal is to define a (marked ordered) flow category $\cC(H_Y)$ or, if we wish to emphasise all choices, $(\cC(H_Y,\frak{b});J_Y)$ since there will be dependence on a bulk class $\frak{b}$ and choice of almost complex structure $J_Y$.

We make the following assumption:

\begin{Hypothesis}\label{h:sector}
The bulk $\mathfrak{b}=\sum_{(g) \in \|\Gamma\|^{\circ}} \frak{v}_{(g)}[X^g/C(g)]$ for some $\frak{v}_{(g)} \in \Lambda_{>0}$.
In particular, all orbifold marked points of a representable orbifold map to $Y$ that are labelled by $(g)$ are constrained by the fundamental class $[X^g/C(g)]$.
All marked points labelled $(g)$ are assigned the weight $\val(\frak{v}_{(g)})$.
\end{Hypothesis}

We define a flow category $\cC(H_Y)$ as follows.

\noindent \textbf{Objects}  are oriented representable $1$-periodic Hamiltonian orbits of $H/\Gamma$ in $Y$ together with a formal capping.
More precisely, an object consists of a
commutative diagram  of the form 
\[
\xymatrix{
\hat{S} \ar[r]^{\hat{x}} \ar[d]^{\pi_x} & X \ar[d] \\ S^1 \ar[r]^{x} & Y
}
\]
such that $\pi_x$ is the quotient map of a free $\Gamma$-space $\hat{S}$, $x$ is a Hamiltonian orbit of $H/\Gamma$ and $\hat{x}$ is $\Gamma$-equivariant.
We require that $x$ be {\it orientable} in the sense that when the isotropy group of $x$ is non-trivial, the action of the isotropy group on the orientation line of $x$ preserves the orientation. 
This condition is parallel to the orientability of a Morse critical point in Equation \ref{eqn:morse} (see \cite{MMRM}, \cite{GZ21} for further discussion).
We specify two further pieces of data.
The first piece is a capping, that is a representable orbifold disc  $\bar{c}_x: D^2 \to Y$ extending $x$.
We follow the cohomological conventions in \cite[Section 6.1]{CGHMSS} (i.e. $\bar{c}_x$ should be interpreted as a  representable map from an orbifold cylinder $\bar{c}_x: [0,1] \times S^1 \to Y$ such that $x(t)=\bar{c}_x(0,t)$ and $\bar{c}_x(1,t)$ is a constant) so we orient $\partial D^2$ with the opposite of the standard orientation.
Let $k_{(g)}(\bar{c}_x)$ be the number of orbifold marked points labelled by the conjugacy class $(g)$.
The action of a capping is defined to be\footnote{We follow the sign and action conventions in \cite{CGHMSS}.} 
\begin{align}\label{eqn:action of capping}
\cA(\bar{c}_x):=\int_x H_Y(t) dt-\int_{D^2} \bar{c}_x^*\omega_Y - \sum_{(g) \in \|\Gamma\|^\circ} k_{(g)}(\bar{c}_x)\cdot \val(\frak{v}_{(g)})    
\end{align}

Note that this can also be evaluated upstairs in $X$.
Indeed, the capping disc $D^2 \to Y$ has an associated admissible cover $\Sigma \to X$ and we have 
\begin{align*}
\int_x H_Y(t) dt &= 1/|\Gamma|\int_{\hat{S}} H(t) \pi_x^*dt 
\\
\int_{D^2} \bar{c}_x^*\omega_Y+\sum_{(g)} k_{(g)}(\bar{c}_x)\cdot \val(\frak{v}_{(g)})&= 1/|\Gamma| \cdot \left(\int_{\Sigma} \bar{c}_{\hat{x}}^*\omega_X \right) + \sum_{(g)} k_{(g)}(\bar{c}_x)\cdot \val(\frak{v}_{(g)}).
\end{align*}

The integralized action $\cA_{\mathbb{Z}}$ for $\cC(H_Y)$ will be introduced in Section \ref{s:interalaction}.

\begin{remark}\label{r:Adiff}
If $\bar{c}_x$ and $\bar{c}'_x$ are two different cappings of $x$, then we can reverse the orientation of $\bar{c}_x$ and glue to $\bar{c}'_x$ to get on orbifold sphere $A$.
Then $\cA(\bar{c}_x)-\cA(\bar{c}'_x) \in \omega_Y(A)+\sum_{(g)} \val(\frak{v}_{(g)}) \mathbb{Z}$.    
We have chosen $\Pi$ to be a subgroup of $\mathbb{R}$ containing $\omega_Y(\pi_2(Y))$ and $\val(\frak{v}_{(g)}) \in \Pi$ by definition.
Therefore, $\cA(\bar{c}_x)-\cA(\bar{c}'_x) \in \Pi$.
\end{remark}

The second piece of further data is a real number $U \in \Pi$.
 We use the notation $c_x$ to denote $(x,\bar{c}_x,U)$ or $(\bar{c}_x,U)$.
We call $c_x$ a formal capping of $x$ and define the action of $c_x=(\bar{c}_x,U)$ to be
\begin{align}\label{eq:formalaction}
\cA(c_x):=\cA(\bar{c}_x)-U
\end{align}

\begin{definition}\label{d:equivalent cap}
Two formal capped oriented orbits $c_x=(\bar{c}_x, U_x)$ and $c_y=(\bar{c}_y,  U_y)$ are  equivalent if $x=y$ and
\begin{align}
\cA(c_x)= \cA(c_y)
\end{align} 
\end{definition}

\begin{definition}\label{d:action}
The \emph{action spectrum} $\Spec(Y,H_Y)$ is the set of values of the action over all formal capped oriented orbits. 
\end{definition} 

The spectral invariants for bulk-deformed Hamiltonian Floer cohomology for the Hamiltonian $H/\Gamma$ will belong to the action spectrum.

\begin{remark}\label{r:GenSpec}
Since $\frak{v}_{(g)} \in \Lambda_{>0}$, we have $k_{(g)}(\bar{c}_x)\cdot \val(\frak{v}_{(g)}) \in \Pi$
and therefore $\Spec(Y,H_Y)$ is independent of the choice of $\frak{v}_{(g)}$.
\end{remark}

Let $\cP$ be the set of equivalence classes of formal capped oriented orbits $c_x$. 
The group $\Pi$ acts on $\cP$ by adding to the value $U$. 
Let $\cA_{\Pi}:\Pi \to \mathbb{R}$ be the inclusion so $\cA(U \cdot c_x)=\cA(c_x)-\cA_{\Pi}(U)$ (cf. Section \ref{s:flow}).

\begin{definition}\label{d:object}
An object of $\cC(H_Y)$ is an element in $\cP$, that is an equivalence class of a formal capped oriented orbit $c_x$.    
\end{definition}

\begin{definition}
An automorphism of $\hat{x}$ is a $\Gamma$-equivariant isomorphism $f:\hat{S} \to \hat{S}$ such that $\hat{x} \circ f=\hat{x}$.
The isotropy group $\Gamma_{x}$ of an object $c_x$ is defined to be the automorphism group of  $\hat{x}$.    
\end{definition}

\begin{remark}
Hamiltonian orbits that admit no cappings do not define objects in $\cC(H_Y)$.
\end{remark}

\begin{remark}
Formal capping is introduced because connect summing orbifold spheres only makes sense, in general, if you connect sum at smooth points (so the monodromy representations are compatible), and that this operation only gives a groupoid, and not a group. This is why we need an element $U \in \Pi$ as well as the geometric cap  $\bar{c}_x$ (cf. Definition \ref{d:equivalent cap}).
\end{remark}

\begin{remark}[(Intrinsic to $Y$)]\label{r:Yobjects}
The objects of $\cC(H_Y)$ can be defined using $Y$ and $H/\Gamma$, which are independent of the presentation of $Y$ as a global quotient. Indeed, $x$ is an oriented representable Hamiltonian loop and $\hat{x}, \pi_x$ are determined by $x$ using the pull-back. Similarly, $\bar{c}_x$ is a representable orbifold disc. In Equation \eqref{eqn:action of capping}, the index set $\|\Gamma\|^\circ$ of the sum is the index set of the non-trivial orbifold sectors of $Y$ so it is intrinsic to $Y$.
The subgroup $\frac{1}{|\Gamma|}\omega_X(H_2(X,\mathbb{Z}))+\sum_{(g)} \val(\frak{v}_{(g)}) \mathbb{Z}$ that $\Pi$ must contain also only depends on $Y$.
\end{remark}

Fix a smooth identification of the interior of $D^2$ with an orbifold complex plane $(\mathbb{C},\bfm)$. The complex orbibundle $\bar{c}_x^*(TY,J)$ admits a canonical desingularization $|\bar{c}_x^*(TY,J)|$ over the smooth $\mathbb{C}$ (see \cite[Section 4.2]{Chen-Ruan}, Remark \ref{r:indexformula}).
By trivializing $|\bar{c}_x^*(TY,J)|$, the linearization of the Hamiltonian flow along $x|_{\partial D^2}$ defines a path of symplectic matrices and hence we can define its Conley-Zehnder index $\mu_{CZ}(c_x):=\mu_{CZ}(\bar{c}_x) \in \mathbb{Z}$.

\begin{definition}[Parity decomposition]\label{d:parity}
    We say that $c_x$ is even if $\mu_{CZ}(c_x)$ is even, and is odd otherwise.
\end{definition}

\begin{remark}\label{r:underlyingparity}
Note that whether $c_x$ is even or odd only depends on the underlying $x$ because the Conley-Zehnder index modulo $2$ only depends on the underlying orbit but not the formal capping. Indeed if $\bar{c}_x$ and $\bar{c}_x'$ are two different caps of $x$, $A:=\bar{c}_x\#\overline{\bar{c}_x'}$ is an orbifold sphere and $\mu_{CZ}(x,\bar{c}_x)-\mu_{CZ}(x,\bar{c}_x')=2c_1(A)$.
\end{remark}

\begin{remark}
    One should not be confused by the fact that $\mu_{CZ}(x,\bar{c}_x) \in \mathbb{Z}$ but $H_{orb}(Y)$ is in general $\mathbb{Q}$-graded. For example, if we take a $C^2$-small Hamiltonian $H$ on $Y$ and a constant cap $\bar{c}_x$ to a constant orbit $x$ in a non-trivial sector $X^g/C(g)$, then $\bar{c}_x$ is a cap with one orbifold point in the interior and 
    $\mu_{CZ}(x,\bar{c}_x)=\mu_{CZ}(|\bar{c}_x^*(TY,J)|)=\mu_{CZ}(\bar{c}_x^*(TY,J))-2a(g,x)$, where the age $a(g,x)$ and $\mu_{CZ}(\bar{c}_x^*(TY,J))$ are in $\mathbb{Q}$ in general (note this final Conley-Zehnder index is of the orbibundle and not its desingularisation).
\end{remark}

\begin{remark}[(An analogue of Remark \ref{r:breaking})]\label{r:recapmu}
Let $u:S \to Y$ be a representable orbifold sphere and $\bar{c}_x$ be a capping of $x$.
If $\bar{c}_x'$ is another capping of $x$ obtained by gluing $u$ and $\bar{c}_x$ along an orbifold point with age $a$ and $a'$ respectively, then 
\[\mu_{CZ}(x,\bar{c}_x')=\mu_{CZ}(x,\bar{c}_x)+2(\langle c_1(|u^*TY|), |S|\rangle+(a+a')).\]
\end{remark}

We state the analogue of Remark \ref{r:recapmu} for action here, whose proof is immediate.

\begin{lemma}\label{l:recapA}
Let $u:S \to Y$ be a representable orbifold sphere with $k$-orbifold points and $c_x$ be a formal capping of $x$.
If $c_x'$ is another capping of $x$ obtained by gluing $u$ and $c_x$ along a smooth point, resp. an orbifold point with label $h$), then $\cA(c_x')=\cA(c_x)+\omega_Y(u)+\sum_{(g)} k_{(g)}(u)  \val(\frak{v}_{(g)})$, resp. $\cA(c_x')=\cA(c_x)+\omega_Y(u)-2\val(\frak{v}_h)+ \sum_{(g)} k_{(g)}(u) \val(\frak{v_{(g)}})$.
\end{lemma}

Before discussing morphisms, we recall that the domains relevant to Hamiltonian Floer theory are genus zero pre-stable curves with two distinguished marked points, which will be punctures when writing down the Floer equation, one incoming and one outgoing. Any such curve has a distinguished chain of components which connect the punctures, which we call \emph{cylindrical}; the other components are \emph{spherical}. 

\noindent \textbf{Morphisms} between formal capped oriented orbits $c_x$ and $c_y$  are given by the moduli space of possibly broken Floer cylinders between them. The precise definition is a little delicate.

\begin{definition}\label{d:orbicylinder}
    A marked prestable orbicylinder consists of the following data:
    \begin{enumerate}
        \item A prestable genus $0$ orbifold curve $\cC$ with $2+h$ ordered marked points such that two of them, denoted by $z_-,z_+$ are distinguished.
        \item A marked point can be a smooth point or an orbifold point, all orbifold points are marked points, nodes between irreducible components are not marked points,
        \item Denote the chain of irreducible components from $z_-$ to $z_+$ by $C_0,\dots,C_r$ and the node connecting $C_{i-1}$ and $C_{i}$ by $z_i$ (by convention $z_-=z_0$ and $z_+=z_{r+1}$). We have an asymptotic marker $l_{z_i,C_i}$ at $T_{z_i}C_i$ for $i=0, \dots, r$ (i.e. $l_{z_i,C_i}$ is an element in $(T_{z_i}C_i \setminus \{0\})/\mathbb{R}_{>0}$ where $\mathbb{R}_{>0}$ acts by scaling).
    \end{enumerate}   
\end{definition}
Two marked prestable orbicylinders are isomorphic if there is a biholomorphism between the underlying prestable genus $0$ curves which interwines all the data.
We denote a marked prestable orbicylinder by $\scrC^{\dagger}$.

Each $\scrC^{\dagger}$ comes with an asymptotic marker $l_{z_{i+1},C_i}$ at $T_{z_{i+1}}C_{i}$ canonically.
Indeed, there is a unique geodesic on $C_i$ from $z_i$ to $z_{i+1}$ whose tangent vector aligns with the asymptotic marker.
The tangent vector of the geodesic at $z_{i+1}$ induces the canonical asymptotic marker $l_{z_{i+1},C_i}$ at $T_{z_{i+1}}C_{i}$.

\begin{definition}
    A $\Gamma$-admissible cover $\Sigma^{\dagger}$ of a marked prestable orbicylinder $\scrC^{\dagger}$ consists of two pieces of data:
    \begin{enumerate}
        \item The first piece is a $\Gamma$-admissible cover $\pi_{\Sigma}:\Sigma \to \cC$.
We can think of $\Sigma$ as an $\Gamma$-admissible cover of each irreducible component of $\cC$ together with an $\Gamma$-equivariant isomorphism between the preimage of the nodal points of $\cC$ (in particular, an isomorphism between the $\Gamma$-orbit above $z_i \in C_{i-1}$ and $z_i \in C_i$).
\item We have an induced map $(\pi_{\Sigma})_*$ on the tangent space of the irreducible components so we can take the preimage of $l_{z_i,C_i}$ and $l_{z_{i+1},C_i}$ under $(\pi_{\Sigma})_*$, which we denote by $\tilde{l}_{z_i,C_i}$ and $\tilde{l}_{z_{i+1},C_i}$
respectively.
They are free and transitive $\Gamma$-orbits.
Every element in $\tilde{l}_{z_i,C_i}$ and $\tilde{l}_{z_{i+1},C_i}$ is an asymptotic marker of a point in $\Sigma$, called the base-point of the asymptotic marker.
The second piece of data of $\Sigma^{\dagger}$ is a base-point preserving $\Gamma$-equivariant isomorphism $\tilde{l}_{z_i,C_i} \simeq \tilde{l}_{z_i,C_{i-1}}$ (the base-points are identified as points in $\Sigma$).
    \end{enumerate}
\end{definition}

An isomorphism between $\Gamma$-admissible covers $\Sigma^{\dagger}$ and $\Sigma^{\dagger'}$ of $\scrC^{\dagger}$ is an isomorphism of the $\Gamma$-admissible covers of $\cC$ which interwines the second piece of data.

\begin{example}
    The preimage $\pi_{\Sigma}^{-1}(z_i)$ is a free $\Gamma$-orbit if and only if $z_i$ carries no isotropy group.
    In this case, the second piece of data is redundant because there is a unique  base-point preserving isomorphism between the asymptotic markers.
\end{example}
\begin{example}\label{e:innerAut}

    Conversely, when the isotropy group of $z_i$ is $\mathbb{Z}/r\mathbb{Z}$, then the isotropy group of every point of $\pi_{\Sigma}^{-1}(z_i)$ with respect to the $\Gamma$ action is $\mathbb{Z}/r\mathbb{Z}$.
    Pick a point $\tilde{z} \in \pi_{\Sigma}^{-1}(z_i)$. The isotropy group is generated by an element  $g \in \Gamma $ with $ord(g)=r$.
    By choosing a marker $l \in \tilde{l}_{z_i,C_{i-1}}$ based at $\tilde{z}$, we can equivariantly identify $\tilde{l}_{z_i,C_{i-1}}$ with $\Gamma$ such that $l$ corresponds to the identity element in $\Gamma$.
    For any $h \in \Gamma$, there are $r$ many elements of  $\tilde{l}_{z_i,C_{i-1}}$ based at $\tilde{z} \cdot h$, namely, $\{g^n h:n=1,\dots,r\}$.
    The map $q \mapsto g \cdot q$ defines a $\Gamma$-equivariant base-point  preserving automorphism of $\tilde{l}_{z_i,C_{i-1}}$.
    Indeed, all $\Gamma$-equivariant base-point  preserving automorphisms are powers of this one.
    Therefore, the second piece of data of $\Sigma^{\dagger}$ is a choice among $|\Gamma_{z_i}|$ many choices for each $z_i$.   
\end{example}

An element in the moduli space $\cM(c_x,c_y)$ consists of a commutative diagram  of the form
\begin{align}\label{eq:orbicylinder}
\xymatrix{
\Sigma^\dagger \ar[r]^{u} \ar[d]^{\pi_u} & X \ar[d] \\ \scrC^\dagger \ar[r]^{v} & Y
}
\end{align}
such that $\scrC^\dagger$ is a marked prestable orbicylinder, $\pi_u$ is the quotient map of a $\Gamma$-admissible cover $\Sigma^\dagger \to \scrC^\dagger$,  the map $u$ is $\Gamma$-equivariant, $v$ satisfies the Floer equation, and the ends of the cylindrical components of $\scrC^\dagger$ are mapped by $v$ to the corresponding asymptotic orbits. 
More precisely, we choose a positive cylindrical end at $z_+$ compatible with the asymptotic marker (i.e. holomorphic embedding $\epsilon_+:[0,\infty) \times S^1 \to \scrC^\dagger$ such that $\lim_{s \to \infty} \epsilon(s,t)=z_+$ and $\lim_{s \to \infty} \partial_s \epsilon(s,0)$ aligns with the asymptotic marker) and a negative cylindrical end $\epsilon_-$ at $z_+$ compatible with the asymptotic marker.
The asymptotic conditions for $v$ are $\lim_{s \to \infty} v(\epsilon_+(s,t))=y(t)$ and $\lim_{s \to -\infty} v(\epsilon_-(s,t))=x(t)$).

 We insist that 
the Floer cylinder and caps are  related by
\begin{align}\label{eq:gluesurface}
c_y\#v:=(\bar{c}_y \# v, U_y) \sim (\bar{c}_x,   U_x)
\end{align}
i.e. these are topologically equivalent caps for $y$ with respect to the previously introduced equivalence.

The Floer equation for $v$ asserts that at regular points of the underlying prestable domain $C$, i.e. points which are neither nodes nor orbifold points, the map satisfies
\begin{equation} \label{eqn:Floer equation}
\left. \begin{aligned}
(dv - X_{H/\Gamma} \otimes dt)^{0,1}_{J/\Gamma} = 0 \\ (dv)^{0,1}_{J/\Gamma} = 0 \end{aligned}\right \} \ \mathrm{for} \ \begin{cases} v & \mathrm{cylindrical} \\ v & \mathrm{spherical} \end{cases}
\end{equation}
We insist that on the cylindrical components, the geodesic determined by the asymptotic markers is where $\{t=0\}$ in the domain curve for the Floer equation.
We also impose the finite energy condition (the integration over the spherical components should be interpreted as their $\omega_Y$-areas)
\[
\int_C \| \partial_s v \|^2 dsdt < \infty
\]
which implies that $v$ converges exponentially fast at the two punctures of the domain to $1$-periodic orbits of $H/\Gamma$ (which we assumed to be non-degenerate), equivalently that $u$ converges exponentially fast at all punctures to periodic orbits of $H$.
Considering \eqref{eqn:Floer equation} upstairs for $u$, note that the pullback $X_H \pi_u^*(dt)$ vanishes at branch points of the map $\pi_u$, and away from those branch points there are well-defined local  co-ordinates $(s,t)$ on $\Sigma$.  The topology on the moduli space is induced by the Gromov topology. 

\begin{remark}
    Notice that a puncture of $\Sigma$ above a point $z_i \in \scrC$ with isotropy group $\mathbb{Z}/r\mathbb{Z}$ converges to a $r$-periodic orbit of $H$.
\end{remark}

Two different $u_0:\Sigma_0^\dagger \to X$ and $u_1:\Sigma_1^\dagger \to X$ are equivalent if there is an isomorphism $f:\Sigma_0^\dagger \simeq \Sigma_1^\dagger$ such that $u_0=u_1 \circ f$ up to an overall translation in the $s$-direction.
The moduli space of these equivalence classes is the morphism space $\cM(c_x,c_y)$.
The subspace $\cM^k(c_x,c_y) \subset \cM(c_x,c_y)$ consists of those elements whose underlying orbifold cylinder $\scrC$ has $k$ ordered orbifold marked points on top of the two distinguished ones.

The virtual dimension of $\cM^k(c_x,c_y)$  is defined to be the index of an element $[v:\scrC^\dagger \to Y] \in \cM^k(c_x,c_y)$ such that  $\scrC$ is irreducible.
\begin{lemma}
The virtual dimension of $\cM^k(c_x,c_y)$ is given by 
$\mu_{CZ}(c_x)-\mu_{CZ}(c_y)-1+2k$.
\end{lemma}

\begin{proof}
This is about how index behaves under gluing. When $x$ is a loop mapping to the trivial sector of $IY$, 
then how the index changes under gluing $v$ and $c_x$ is classical.
When $x$ is mapped to a non-trivial sector, the computation is local and the case of a closed curve was explained in Remark \ref{r:breaking}.
\end{proof}

We define the \emph{geometric energy} of a Floer solution $v: \scrC \to Y$ to be
\begin{equation} \label{geometric energy} 
E_{geo}(v) := \int_{\scrC} \|\partial_s v\|^2dsdt + \sum_{(g)} k_{(g)}(v)\cdot \val(\frak{v}_{(g)})
\end{equation}
where the stable domain $\scrC$ of $v$ carries $k_{(g)}(v)$ orbifold points labelled by $g$. 

\begin{lemma}
If $v \in \cM(c_x,c_y)$ then $\cA(c_y) - \cA(c_x) = E_{geo}(v)$. 
\end{lemma}

\begin{proof}
Standard but incorporating the bulk insertion (see Definition \ref{d:equivalent cap}).
\end{proof}

\noindent \textbf{Composition.}  In orbifold Morse theory, a key insight due to Cho and Hong \cite[Example 5.2]{Cho-Hong} is that a broken Morse flow line may not have a unique smoothing. 
It corresponds in Equation \eqref{eq:codim1} to the fact that we have a fibre product rather than a product (see Remark \ref{rmk: weights work}).


The moduli space $\cM(c_x,c_z)$ is naturally stratified by the poset $\bA_{c_xc_z}$ defined in \eqref{eq:poset}.  In particular, it has a collection of locally closed `boundary facets' which are the strata labelled by depth one chains $c_x < c_y < c_z$, and an interior $\mathring{\cM}(c_x,c_z)$ which is the top stratum, comprising one cylindrical component and finitely many (possibly no) spherical components.

\begin{lemma}[cf. \cite{GZ21}, Proposition 2.9]\label{l:boundarycover} There is an embedding 
\[
\cM(c_x,c_y) \times_{[pt/\Gamma_y]} \cM(c_y,c_z) \to \cM(c_x,c_z)
\]
fibred over $\cM(c_x,c_y) \times \cM(c_y,c_z)$, and the images of such maps, as $y$ varies, enumerate the boundary facets of the image.
\end{lemma}

\begin{proof}

Suppose that $[u_0:\Sigma_0^{\dagger} \to X] \in \cM(c_x,c_y)$
and $[u_1:\Sigma_1^{\dagger} \to X] \in \cM(c_y,c_z)$.
For $i=0,1$, let $z_i \in \scrC_i^{\dagger}$ be the puncture that goes to $y$.
We can form the wedge $\scrC_0^{\dagger} \vee \scrC_1^{\dagger}$ by identifying $z_0$ and $z_1$.
To construct an element in $\cM(c_x,c_z)$ from $u_0, u_1$, we need to first construct a $\Gamma$-admissible cover of $\scrC_0^{\dagger} \vee \scrC_1^{\dagger}$.
Since the automorphism group of $y$ is $\Gamma_y$, there are exactly $|\Gamma_y|$ many ways to construct a $\Gamma$-admissible cover of $\scrC_0^{\dagger} \vee \scrC_1^{\dagger}$ by identifying the $\Gamma$-orbits over $z_0$ and $z_1$ and specifying a $\Gamma$-isomorphism of the asymptotic markers.
Indeed, consider $\hat{y}:\hat{S} \to X$ covering $y:S^1 \to Y$.
The domain $\hat{S}$ is a union of circles, one for each point in $\Sigma_0$ above $z_0$.
Automorphisms of $\hat{y}$ which permute the circles of $\hat{S}$ correspond to permuting how we identify the $\Gamma$-orbit above $z_0$
and the $\Gamma$-orbit above $z_1$ to form the wedge $\Sigma_0 \vee \Sigma_1$.
The group of automorphisms of $\hat{y}$ which don't permute the circles is the isotropy group of $z_0$ (equivalently of $z_1$) and corresponds to changing the second piece of data in the definition of a $\Gamma$-admissible cover (cf. Example \ref{e:innerAut}).
Denote the admissible cover of $\scrC_0^{\dagger} \vee \scrC_1^{\dagger}$ corresponding to $g \in \Gamma_y$ by $\Sigma_g^\dagger:=\Sigma_0^{\dagger} \vee_g \Sigma_1^{\dagger}$.
All together, we exactly have a map of orbispaces
\begin{align}\label{eq:boundary}
\cM(c_x,c_y) \times_{[pt/\Gamma_y]} \cM(c_y,c_z) &\to \partial^{c_xc_yc_z}\cM(c_x,c_z) \subset \cM(c_x,c_z) \\
([u_0:\Sigma_0^{\dagger} \to X], [u_1:\Sigma_1^{\dagger} \to X], g) &\mapsto u_0 \vee u_1:\Sigma_g^{\dagger} \to X
\end{align}

We take a closer look at the isotropy group of the elements in $\cM(c_x,c_y) \times_{[pt/\Gamma_y]} \cM(c_y,c_z)$
and in $\partial^{c_xc_yc_z}\cM(c_x,c_z)$.
Denote the automorphism group of $[u_i:\Sigma_i^{\dagger} \to X]$ by $\Gamma_{u_i}$.
By restricting the action to the puncture, we have a group homomorphism $\Gamma_{u_i} \to \Gamma_y$.
Then $\Gamma_y$ has a decomposition into $\Gamma_{u_0} \times \Gamma_{u_1}$-orbits. 
If $g_0,g_1 \in \Gamma_y$ are in the same $\Gamma_{u_0} \times \Gamma_{u_1}$-orbit, then 
$([u_0:\Sigma_0^{\dagger} \to X], [u_1:\Sigma_1^{\dagger} \to X], g_0)$ is isomorphic to 
$([u_0:\Sigma_0^{\dagger} \to X], [u_1:\Sigma_1^{\dagger} \to X], g_1)$, where the isomorphism is given by the element in $\Gamma_{u_0} \times \Gamma_{u_1}$ which fixes $u_0$, $u_1$ and sends $g_0$ to $g_1$.
Similarly, $u_0 \vee u_1:\Sigma_{g_0}^{\dagger} \to X$ and $u_0 \vee u_1:\Sigma_{g_1}^{\dagger} \to X$ are isomorphic for a similar reason.
Indeed, there is a bijective correspondence between isomorphisms classes of elements in $\cM(c_x,c_y) \times_{[pt/\Gamma_y]} \cM(c_y,c_z)$ that lie above $(u_0,u_1)$ with $\Gamma_{u_0} \times \Gamma_{u_1}$-orbits of $\Gamma_y$, and the correspondence respects the automorphism groups.
The same is true for elements in $\partial^{c_xc_yc_z}\cM(c_x,c_z)$ that lie above $(u_0,u_1)$.
Therefore, \eqref{eq:boundary} is an isomorphism of orbispaces.
\end{proof}

\begin{remark}[(Intrinsic to $Y$)]\label{r:Ymorphism}
    For every $\cM(c_x,c_y)$, the top stratum only has one cylindrical component so $\Sigma^\dagger$ and the maps $u, \pi_u$ only depend on the representable orbifold morphism $v:\scrC^\dagger \to Y$.
    Therefore, the top stratum of every $\cM(c_x,c_y)$ can be defined in terms of $Y$ (not using the presentation of $Y$ as a global quotient).
    By Lemma \ref{l:boundarycover}, every other stratum can be expressed as a fibre product of other moduli spaces.
    Therefore, by induction, we can see that every $\cM(c_x,c_y)$ can be defined in terms of $Y$ (without appealing to the presentation of $Y$ as a global quotient).
\end{remark}

\begin{remark}
    If $A,B,C$ are orbispaces and $f:A\to C$, $g:B \to C$ are representable morphisms between orbispaces, then we can form the fibre product $A \times_C B$. However, the universal property of $A \times_C B$ is \emph{not} that given $i:D \to A$ and $j: D \to B$ such that $f \circ i=g \circ j$, then there is a unique morphism $k:D \to A \times_C B$ such that $i$ and $j$ factor though $k$.
    The correct universal property is given in \cite[Remark 4.23]{Lerman}. Therefore, even though there are natural maps $\partial^{c_xc_yc_z}\cM(c_x,c_z) \to \cM(c_x,c_y)$, $\partial^{c_xc_yc_z}\cM(c_x,c_z) \to \cM(c_y,c_z)$ and an induced map
    $\partial^{c_xc_yc_z}\cM(c_x,c_z) \to \cM(c_x,c_y) \times_{[pt/\Gamma_y]} \cM(c_y,c_z)$, the induced map is not an isomorphism in general. Indeed, its image is $\{(u_0,u_1,id)| u_0 \in \cM(c_x,c_y), u_1 \in \cM(c_y,c_z)\}$, so it is not surjective as long as $\Gamma_y$ is non-trivial.
\end{remark}

\subsection{Integralized actions}\label{s:interalaction}

\emph{Global chart lifts for the category are constructed inductively over an integralized action, introduced here.}

This section follows \cite[Section 3]{Rez}.   Recall that we have assumed the symplectic form $[\omega_X] \in H^2(X;\bQ)$.

Following \cite[Equation (41)]{Rez}, we define an integralized action $\cA_{\mathbb{Z}}:\mathrm{Ob}(\cC(H_Y)) \to \bZ$ as follows:

The action of a formal capped orbit $c_x=(\bar{c}_x,U_x)$ equals (see Equation \eqref{eqn:action of capping} and \eqref{eq:formalaction})
\[
\cA(c_x)=\int_x H_Y(t) dt - \int \bar{c}_x^*\omega_Y   - \sum_{(g) \in \|\Gamma\|^\circ} k_{(g)}(\bar{c}_x)\cdot \val(\frak{v}_{(g)}) -U_x.
\]

For each orbit $x$ in $Y$, we have the set of actions of all formal capping discs
\[
\cA_x := \left\{ \cA(c_x) :c_x=(\bar{c}_x,U_x) \text{ is a formal capping disc of } x\right\} \subset \mathbb{R}
\]
Since $\Pi$ is a discrete subgroup containing $\omega_Y(H_2(Y,\mathbb{Z}))$, $\omega_X$ is rational and $\val(\frak{v}_{(g)}) \in \Pi$,  the set $\cA_x$ belongs to $\int_x H(t) dt  + \frac{1}{m_x} \bZ$ for some $m_x \in \mathbb{Z}$.  We pick $\varepsilon_x \in \bR$ 
so that 
\[
\cA_x + \varepsilon_x \subset \bQ.
\]
Since there are finitely many orbits, we can pick $N \in \bN$ so that
\[
 N\cdot (\cA_x + \varepsilon_x) \subset \bZ
\]
for all orbits $x$ in $Y$.
For these choices of $\epsilon_x$ and $N$, we define the associated integralized action of a formal capped oriented orbit $c_x=(\bar{c}_x,U_x)$ to be
\[
\cA_{\mathbb{Z}} (c_x) := N(\cA(c_x)+\varepsilon_x)
\]
This is integer-valued by construction. 
Similarly, for representable orbifold spheres $v$ in $Y$, we define the integralised action $\cA_{\Pi,\mathbb{Z}}$ by
\[
\cA_{\Pi,\mathbb{Z}}(v):=N\cA_{\Pi}(v)=N \left(\int v^*\omega_Y+ \sum_{(g) \in \|\Gamma\|^\circ} k_{(g)}(v)\cdot \val(\frak{v}_{(g)})\right)
\]
In practice, we will need to choose $\epsilon_x$ sufficiently small so that Lemma \ref{lem:positive} below holds.

Using a suitable cut-off function (see \cite[Equation (44)]{Rez} for the details) we can pick a Hamiltonian function 
\begin{align}\label{eq:tildeH}
\tilde{H}:X \times S^1 \to \bR
\end{align}
which is $\Gamma$-invariant, which agrees with $N\cdot H$ away from a neighbourhood of the periodic orbits of $H$, but agrees with $N\cdot (H + \varepsilon_x)$ near any orbit component of $\hat{x}(\hat{S})$. 
Then for a morphism $u: \Sigma^\dagger \to X$ from $c_x$ to $c_y$ one has
\[
\cA_{\mathbb{Z}}(c_y) - \cA_{\mathbb{Z}}(c_x) = \frac{1}{|\Gamma|}\int_u \omega_u + N\sum_{(g)} k_{(g)}(v) \val(\frak{v}_{(g)}) \, \qquad \omega_u := u^*(N\omega_X) - d(\tilde{H}_t(u) \pi_{u}^*dt)
\]
where $\pi_u:\Sigma^\dagger \to \scrC^\dagger$ is the quotient map and $k_{(g)}(v)$ is the number of orbifold points on $\scrC^\dagger$ labelled by the sector $(g)$.

The following is a variation on \cite[Lemmas 19,20]{Rez}:

\begin{lemma}\label{lem:positive}
When $\epsilon_x>0$ are chosen to be sufficiently small, the action difference quantity $E(u):= \int_u \omega_u + N\sum_{(g)}  k_{(g)}(v)\val(\frak{v}_{(g)})  $ is strictly positive on every irreducible component of a possibly-broken stable Floer cylinder $u$.  Moreover, $E:\cM(c_x,c_y) \to \mathbb{Z}$ is proper.
\end{lemma}

\begin{proof} The first term of $E(u)$ is strictly positive on any non-constant component when $\epsilon_x$ are sufficiently small (see \cite[Lemma 20 and the paragraph before Equation (40)]{Rez}).
Stability implies that any constant (necessarily spherical) component carries a non-empty set of orbifold marked points, in which case the second term is positive. Therefore, $E(u)$ is strictly positive. On the other hand, the properness of $E$ follows from the argument in \cite[Lemma 19]{Rez} that the integralized action difference is Lipschitz with respect to the usual action difference.
\end{proof}

We fix once and for all a choice of $\epsilon_x$ (for all $x$) and $N$ so that Lemma \ref{lem:positive} holds.

We will need to consider Hermitian line bundles on $\Sigma^\dagger$ when we construct global charts.
Concerning the analytic aspect, it is easier to replace $\omega_X$ by another cohomologous closed $2$-form.

The following is a variant of \cite[Lemma 16]{Rez}:

\begin{lemma}\label{l:Omega}
There is a closed $\Gamma$-invariant  2-form $\Omega$ on $X$ with the properties that
\begin{enumerate}
\item $\Omega$ is cohomologous to $\omega_X$, hence represents a rational cohomology class;
\item  for any $1$-periodic orbit $x:S^1 \to Y$ and the associated $\hat{x}:\hat{S} \to X$,
$\Omega$ vanishes in a small open neighbourhood $U$ of any component of $\hat{x}(\hat{S}) \subset X$ and tames $J_X$ away from such open sets;
\item $\int_{\Sigma} \bar{c}_{\hat{x}}^*\omega_X = \int_{\Sigma} \bar{c}_{\hat{x}}^*\Omega$ for every $\bar{c}_{\hat{x}}:\Sigma \to X$ covering a capping $\bar{c}_x$ in $Y$.
\end{enumerate}
\end{lemma}

\begin{proof} 
Following \cite[Lemma 16]{Rez}, we get a closed $2$-form $\Omega$ satisfying the itemized properties.
Group-invariance can be achieved by averaging.
\end{proof}

We will choose the open set $U$ sufficiently small so that no non-constant finite energy Floer cylinder of $H/\Gamma$ has image in $U/\Gamma$.  The existence of such a $U$ follows from \cite[Lemma 17]{Rez}.

As a consequence of Lemma\ref{l:Omega}(3), for any morphism $u$ from $c_x$ to $c_y$, we have
\[
\int u^*\omega_X=\int u^*\Omega.
\]
Therefore, we have
\[
\cA_{\mathbb{Z}}(c_y) - \cA_{\mathbb{Z}}(c_x) = \frac{1}{|\Gamma|}\int_u \Omega_u + N\sum_{(g)} k_{(g)}(v) \val(\frak{v}_{(g)}) \, \qquad \Omega_u := u^*(N\Omega) - d(\tilde{H}_t(u) \pi_{u}^*dt)
\]

\begin{lemma}\label{l:additivity}
If $u$ is a morphism obtained by gluing morphisms $u=u_1\# u_2$ then the integralized action difference for $u$ is the sum of those for the $u_i$, in particular is strictly greater than that for either $u_i$. 
\end{lemma}

\begin{proof}
Additivity holds by construction.
\end{proof}

 The construction of a global chart lift for the flow category $\cC(H_Y)$ will be inductive, where the induction is via integralized action. That this gives a valid induction scheme follows from Lemma \ref{lem:positive} and \ref{l:additivity}, which imply that every moduli space occuring in a boundary stratum will already have been encountered.

\subsection{Global charts for orbifold Floer cylinders}\label{s:globalchart_cylinder}

\noindent \emph{Both  \cite{BX,Rez} construct Hamiltonian Floer theory via the global chart construction from \cite{AMS1}, which treated only genus zero curves. We use the more general construction from \cite{AMS2}, which allows one to work with curves of arbitrary genus (as inevitably appear in the orbifold setting).}

Fix a closed $(X,\omega_X)$, a taming $J_X$, and a time-dependent Hamiltonian $H: X \times S^1 \to \bR$, with $J_X$ and $H$ both $\Gamma$-equivariant and $H$ non-degenerate.

Recall that the domains relevant to Hamiltonian Floer theory have both cylindrical and spherical components. On each cylindrical component we fix a geodesic half-circle between the two punctures in order to write down the Floer equation (see Definition \ref{d:orbicylinder} and the definition of $\cM(c_x,c_y)$).

In order to construct a global chart for $\cM(c_x,c_y)$, we need to first describe a moduli space 
of equivariant embeddings from admissible covers of genus zero orbifold curves $\scrC^\dagger$ to $\bP(V)$.

Recall $\scrF_{0,\bfm}(V^*) \subset \cK_{\bfm}([\bP(V)/\Gamma],V^*)$ from Equation \ref{eq:condition}.
Let $\scrF_{0,2+h}(V^*)$ be the union of  $\scrF_{0,\bfm}(V^*)$ over those $\bfm \in \cH_{2+h}$ such that all the marked points other than the first two are labelled by non-trivial conjugacy classes of $\Gamma$. 
In particular, we have $2+h$ marked points, of which the first two $z_-,z_+$ are distinguished. Under Hypothesis \ref{h:sector}, we assume that the non-distinguished marked points are orbifold points labelled by a non-trivial sector and the distinguished marked points can be smooth points or orbifold points.

Let $\scrF_{0,2+h}^\dagger(V^*)$ be the moduli of equivariant holomorphic maps $u:\Sigma^\dagger \to \bP(V)$ such that the underlying map $\bar{v}:\scrC \to [\bP(V)/\Gamma]$ lies in $\scrF_{0,2+h}(V^*)$.
In other words, $\bar{v}$ is obtained by descending $u$ to the quotient $v:\scrC^\dagger \to [\bP(V)/\Gamma]$ and then forgetting the asymptotic markers on the cylindrical components.
Alternatively, we can first forget the asymptotic markers to obtain $\bar{u}:\Sigma \to \bP(V)$ and then descend to $\bar{v}$.
The requirement that $\bar{v} \in \scrF_{0,2+h}(V^*)$ implies that $\bar{u}$ is an embedding and hence $u$ is an embedding.
Therefore, any element in $\scrF_{0,2+h}^\dagger(V^*)$ is automorphism-free.

We want to describe the forgetful map $\scrF_{0,2+h}^\dagger(V^*) \to \scrF_{0,2+h}(V^*)$. We start with an example.

\begin{example}\label{e:realblowup}
    Consider the map
    \[
    f:E=\{x,y,z \in \mathbb{C}: xy=z\} \to \mathbb{C}_z
    \]
which is a local model of a nodal degeneration of a curve.
Consider the $\mathbb{Z}/ r\mathbb{Z}$-action rotating $x,y$ by
$x \mapsto e^{2\pi\sqrt{-1}/r} x$ and $y \mapsto e^{-2\pi\sqrt{-1}/r} y$.
Denote the quotient stack by $[E']$ and its coarse moduli space by $E$. More explicitly, we have
\[
E'=\{x',y',z' \in \mathbb{C}: x'y'=z'^r\}
\]
where the quotient map is given by $x'=x^r$, $y'=y^r$ and $z'=z=xy$.
The map 
\[
f':[E'] \to \mathbb{C}_{z'}
\] 
is a local model of the degeneration of a curve to a nodal one with a $\mathbb{Z}/ r\mathbb{Z}$-automorphism at the node.

Let $\theta\in [0,2\pi)$ and consider the ray $z'=Re^{\sqrt{-1}\theta}$ for $R> 0$.
We can parametrize $\{x'y'=(Re^{\sqrt{-1}\theta})^r\}$ by $x'=R^re^{s+\sqrt{-1}(t+r\theta)}$ and $y'=e^{-s-\sqrt{-1}t}$.
The $\{t=0\}$ locus, when $R$ goes to $0$, converges to $\{arg(x')=r \theta\}$ on $\{y'=0, z'=0\}$ and $\{arg(y')=0\}$ on $\{x'=0, z'=0\}$.
In this way, we can associate geodesic labels on the components of the curve $\scrC_0:=f'^{-1}(0)$ by starting with an angle $\theta$.
Two geodesic labels are the same if and only if the angles $\theta, \theta'$ differ by an integer multiple of $2\pi/r$.
For any $z'$, $\Sigma_{z'}:=f^{-1}(z')$ is a uniformization chart of $\scrC_{z'}:=f'^{-1}(z')$.
If we consider the ray $z=Re^{\sqrt{-1}\theta}$ upstairs and parametrize by $x$ and $y$ instead, the geodesic labels associated to $\theta$ on $\Sigma_0$ will be $\{arg(x)=\theta\}$ on $\{y=0, z=0\}$ and $\{arg(y)=0\}$ on $\{x=0, z=0\}$.
It means that there is a bijective correspondence between $\theta \in [0,2\pi)$ and geodesic labels on the components of $\Sigma_0$ up to simultaneous $S^1$ rotation of the components.

This implies that the real blow-up at $\{z=0\}=\{z'=0\}$ parametrizes a curve $\Sigma_z$ together with geodesic labels on the components of $\Sigma_z$ up to simultaneous $S^1$ rotation of the components.
\end{example}

Now we consider the general case.
Each irreducible component of $\partial \scrF_{0,2+h}(V^*)$, the boundary of $\scrF_{0,2+h}(V^*)$, corresponds to a degeneration of the domain $\scrC$.
Let $\partial^{cyl}\scrF_{0,2+h}(V^*)$ be the union of the irreducible components of $\partial \scrF_{0,2+h}(V^*)$ which correspond to a degeneration of the cylindrical component to two cylindrical components (so along the intersection of $r$ many different such irreducible boundary components, the domain curve of a generic element has $r+1$ cylindrical components).
Let $Bl(\scrF)_{0,2+h}(V^*)$ be the real blow-up of $\scrF_{0,2+h}(V^*)$ along $\partial^{cyl}\scrF_{0,2+h}(V^*)$.
By Lemma \ref{lem:normal crossing}, the boundary of $\scrF_{0,2+h}(V^*)$ is a normal crossing divisor so $\partial^{cyl}\scrF_{0,2+h}(V^*)$ is a normal crossing divisor as well.
It implies that $Bl(\scrF)_{0,2+h}(V^*)$ is naturally a smooth manifold  with corners.

\begin{lemma}[cf. \cite{BX} Lemma 5.15]\label{l:realblowup}
    The forgetful map $\scrF_{0,2+h}^\dagger(V^*) \to \scrF_{0,2+h}(V^*)$ factors through $\scrF_{0,2+h}^\dagger(V^*) \to Bl(\scrF)_{0,2+h}(V^*)$ as smooth manifolds with boundary and corners, and the latter map is an $S^1$-bundle. 
\end{lemma}

\begin{proof}
    Suppose that $\phi$ is a generic point on an irreducible component of $\partial^{cyl}\scrF_{0,2+h}(V^*)$ so $\scrC_{\phi}$ has two irreducible cylindrical components.
    We denote the two distinguished marked points of $\scrC_{\phi}$ by $z_-$, $z_+$, and the nodal point by $z_0$.
    We write $\phi$ as $u=u_0 \vee u_1: \Sigma_\phi=\Sigma_{\phi_0} \vee \Sigma_{\phi_1} \to \bP(V)$.
    Let $p \in \Sigma_{\phi_0} \cap \Sigma_{\phi_1}$ be a lift of $z_0$.
    The isotropy group $\Gamma_p$  of $p$ is the same as the isotropy group of the node of $\scrC_{\phi}$ and therefore is cyclic (see Remark \ref{r:isotropy}), which we denote by $\langle g \rangle$.

    There is a $T_\scrC:=S^1 \times S^1$ family of marked orbicylinder structures we can put on $\scrC_{\phi}$ (see Definition \ref{d:orbicylinder}, item 3)  given the fact that $\scrC_{\phi}$ has two irreducible cylindrical components and the two $S^1$'s are $(T_{z_-}\scrC_{\phi_0} \setminus \{0\})/\mathbb{R}_{>0} $ and $(T_{z_0}\scrC_{\phi_1} \setminus \{0\})/\mathbb{R}_{>0} $ respectively. Here we use that $v$ is an embedding to avoid encountering any accidental isomorphism between different choices of orbicylinder structure in this $T_{\scrC}$ family.
    For each marked prestable orbicylinder structure $\scrC^\dagger$, there are $ord(g)$ many admissible covers $\Sigma^\dagger$ covering $\scrC^\dagger$.
    Indeed, given $l_0' \in (T_{z_-}\scrC_{\phi_0} \setminus \{0\})/\mathbb{R}_{>0} $ and $l_1 \in (T_{z_0}\scrC_{\phi_1} \setminus \{0\})/\mathbb{R}_{>0} $, $l_0'$ determines an asymptotic marker $l_0 \in (T_{z_0}\scrC_{\phi_0} \setminus \{0\})/\mathbb{R}_{>0}$.
    There are $ord(g)$ many possible lifts of $l_0$ at $p$ and another  $ord(g)$ many possible lifts of $l_1$ at $p$.
    Any pair of lifts $(\tilde{l}_0,\tilde{l}_1)$ at $p$ of $(l_0,l_1)$ determines a base-point preserving equivariant isomorphism between the asymptotic markers, namely, the one sending $\tilde{l}_0$ to $\tilde{l}_1$.
    Two isomorphisms are the same if and only if two pairs of lifts differ by a simultaneous action of $\langle g \rangle$.
    Therefore, in total, the family of admissible cover structures on $\scrC$ is $T_{\Sigma}=S^1 \times S^1/\{(t_0,t_1)=(t_0+\frac{1}{ord(g)}, t_1+\frac{1}{ord(g)})\}$ and the natural map $T_{\Sigma} \to T_\scrC$ is a $ord(g)$-fold covering.

    After making an identification between each of $(T_{z_0}\scrC_{\phi_0} \setminus \{0\})/\mathbb{R}_{>0}$, $(T_{z_0}\scrC_{\phi_1} \setminus \{0\})/\mathbb{R}_{>0} $ and $S^1$, we can make sense of the angle difference of the two asymptotic markers at $z_0$, which is measured by $\theta_0-\theta_1 \in S_{\scrC}=S^1$ for $(\theta_0,\theta_1) \in T_\scrC$.
    The composition $T_{\Sigma} \to T_{\scrC} \to S_{\scrC}$ factors through $T_{\Sigma} \to S_{\Sigma} \to S_{\scrC}$, where $S_{\Sigma}=S^1$ measures the angle difference of the asymptotic markers $\tilde{l}_0,\tilde{l}_1$ at $p$, the first map is $(t_0,t_1) \mapsto t_0-t_1$ and the second map is multiplication by $ord(g)$.

    Example \ref{e:realblowup} is a universal local model of a degeneration which develops an orbifold node with isotropy group $\mathbb{Z}/r\mathbb{Z}$. 
    It explains how the circle fibre of a real blow-up is identified with $S_{\Sigma}$.
    Since $\partial^{cyl}\scrF_{0,2+h}(V^*)$ is a normal crossing divisor, the local model at a stratum of codimension $s$ will be a product of $s$ many such local models.
    In this case, the family of marked orbicylinder structures will be $T_{\scrC}=(S^1)^{s+1}$.
    The angle difference map will be an $S^1$-bundle $T_{\scrC} \to (S^1)^s=:S_{\scrC}$ given by $(\theta_0,\dots, \theta_s) \mapsto (\theta_0-\theta_1, \dots, \theta_{s-1}-\theta_s)$.
    The space of angle differences of the lifts will be $S_{\Sigma}=(S^1)^s$ and the natural map $S_{\Sigma} \to S_{\scrC}$ is a $ord(g_1) \dots ord(g_s)$-fold covering, where $\langle g_i \rangle$ is the isotropy group of the $i^{th}$-node.
    The family $T_{\Sigma}$ of admissible covers will be a $ord(g_1) \dots ord(g_s)$-fold covering of $T_{\scrC}$, and an $S^1$-principle bundle of $S_{\Sigma}$.
    Indeed, when $s>1$, an easier way to think about $T_{\Sigma}$ is that it is the fibre product of the maps $S_{\Sigma} \to S_{\scrC}$ and $T_{\scrC} \to S_{\scrC}$.
    Alternatively, we can think of fibres of $T_{\Sigma} \to S_{\Sigma}$ (which are also fibres of $T_{\scrC} \to S_{\scrC}$) as the circle of real directions at the tangent space at $z_-$.

    By running this local model in the universal family, we obtain an $S^1$-bundle $\scrF_{0,2+h}^\dagger(V^*) \to Bl(\scrF)_{0,2+h}(V^*)$, which fibrewise is given by $T_{\Sigma} \to S_{\Sigma}$.
\end{proof}

Let $c_x=(\bar{c}_x,U_x)$ and $c_y=(\bar{c}_y, U_y)$  be objects of the flow category, i.e. equivalence classes of formal capped marked periodic orbits (see Section \ref{s:orbifoldFloer}).
Similar to \eqref{eq:closedmodulidecom2}, we have a decomposition
\begin{align}
\cM(c_x,c_y)=\cup_{\bfm} \cM_{\bfm}(c_x,c_y)=\cup_{\bfm} \cup_{\epsilon} \cM_{\bfm,\epsilon}(c_x,c_y)
\end{align}
where $\bfm$ is the Hurwitz orbit of the domain orbifold curve and $\epsilon$ is an equivariant homotopy class relative to $\bar{c}_x$ and $\bar{c}_y$.
In view of Hypothesis \ref{h:sector}, we will only consider Hurwitz orbits $\bfm$ such that the orbifold marked points are labelled by the corresponding inertia sectors.
We are going to construct a global chart for the space of orbifold Floer cylinders $\cM_{\bfm,\epsilon}(c_x,c_y)$ following the template used to construct a global chart for closed prestable orbifold genus zero curves (Section \ref{s:globalchartclosed}). 

This time, we use $\scrF^\dagger_{0,2+h}(V^*)$ instead of $\scrF_{0,2+h}(V^*)$.
Denote the universal domain curve of $\scrF^\dagger_{0,2+h}(V^*)$ by $\scrC^{orb}$ and denote the stable marked cylinder $\scrC^{orb}|_{\phi}$ by $\Sigma^\dagger_{\phi}$.

We have a fixed $\Gamma$-equivariant Hermitian line bundle $L$ on $X$ such that the curvature $2$-form is $\Omega$ as constructed in Lemma \ref{l:Omega}. 
Recall the Hamiltonian function $\tilde{H}$ from \eqref{eq:tildeH}.
For any $\Gamma$-equivariant smooth  map $u:\Sigma^\dagger_{\phi} \to X$, let $u^*L_{\tilde{H}}$ be the Hermitian line bundle on $\Sigma^\dagger_{\phi}$ whose curvature $2$-form is given by 
\begin{align}\label{eq:cuv2form}
u^*\Omega+ d(\tilde{H}(u)\pi_{u}^*dt)
\end{align}
where $\pi_u:\Sigma^\dagger_{\phi} \to \scrC^\dagger$ is the quotient map.
Let $k$ be a positive integer. 
If $k$ is large, then for any $\Gamma$-equivariant smooth  map $u:\Sigma^\dagger_{\phi} \to X$ representing an element in $\cM_{\bfm, \epsilon}(c_x,c_y)$, the Hermitian line bundle $L_u:=(\omega_{log} \otimes u^*L_{\tilde{H}}^{\otimes 3})^{\otimes k}$ will satisfy $H^1(L_u)=0$.
Therefore, there is a uniform large $k$ such that $H^0(L_u)$ is independent of $u \in \cM(c_x,c_y)$ as a $\Gamma$-representation.
Let $V$ be the $\Gamma$-representation such that $V^* \simeq H^0(L_u)$ and denote the corresponding $\scrF^\dagger_{0,2+h}$ space by $\scrF^\dagger$.
We take a finite dimensional approximation scheme $\{W_{\nu}\}$ for  the infinite-dimensional $\Gamma$-representation  $\omega_{\scrC^{orb}/\scrF^\dagger}\otimes TX$ such that $\{V_{\nu}:=W_{\mu}^{\Gamma}\}$ is a finite dimensional approximation for   $(\omega_{\scrC^{orb}/\scrF^\dagger}\otimes TX)^{\Gamma}$.
We can form a space $(\scrF^\dagger_{0,2+h})'$ similar to $\scrF'$ (see Equation \ref{eq:double}).

 The group $G$ is $U_{\Gamma}(V)$.
The pre-thickening $\scrT$ comprises tuples $( \phi,u,e,F)$ where 
\begin{enumerate}
\item $\phi \in \scrF^\dagger_{0,2+h}$ such that $\Sigma^{\dagger}$ has pure Hurwitz orbit $\bfm$; 
\item $u: \Sigma^\dagger_{\phi} \to X$ is a smooth $\Gamma$-equivariant map representing the class $\epsilon$;
\item  $e\in V_{\mu}$;
\item $F:H^0(L_u) \to V^*$ is an $\Gamma$-equivariant isomorphism such that the matrix $H_F$ of inner products of $\{F^{-1}(f_j)\}$ has positive eigenvalues
\end{enumerate}
such that (i)  the analog of \eqref{eqn:diagonal_condition} is satisfied with $(\scrF^\dagger_{0,2+h})'$ replacing $\scrF'$, (ii) 
$u$ converges exponentially fast at punctures on cylindrical components to Hamiltonian orbits (where the punctures over $z_-$ and $z_+$ are mapped to $x$ and $y$ respectively); (iii)  such that the asymptotic markers align with the starting point of the orbits, and (iv) which solve a  Floer equation which we write in shorthand as (recall that $v$ is the corresponding map to $Y$ induced by $u$)
\begin{equation} \label{eqn:Floer equation1}
\left. \begin{aligned}
(dv - X_{H/\Gamma} \otimes dt)^{0,1}_{J/\Gamma} + \lambda_{\mu}(e) = 0 \\ (dv)^{0,1}_{J/\Gamma} + \lambda_{\mu}(e) = 0 \end{aligned}\right \} \ \mathrm{for} \ \begin{cases} v & \mathrm{cylindrical} \\ v & \mathrm{spherical} \end{cases}
\end{equation}
(cf. the previously introduced  Cauchy-Riemann equation \eqref{eqn:CR}).

The \emph{obstruction bundle} $E \to \scrT$ has fibre $H \oplus V_{\mu} \oplus \scrH$.  The section $s$ is defined by
\[
(\phi,u,e,F) \mapsto (\Phi^{-1}(\phi,\phi_F), e, exp^{-1}H_F).
\]

\begin{lemma}\label{l:mfd}
 $\scrT$ is a topological manifold in a neighbourhood of $s^{-1}(0)$ for sufficiently large $k$ and $\mu$.
\end{lemma}

\begin{proof}
Once $\mu$ is sufficiently large the obstruction space surjects to the cokernel of the $\overline{\partial}$-operator at all Floer solutions, and hence in a neighbourhood of $s^{-1}(0)$. The result then follows from classical gluing analysis, compare to \cite{AMS1,AMS2}.
\end{proof}

\begin{prop}\label{p:singlechart}
For sufficiently large $k$ and $\mu$ the data $\bT(\rho) = (G,\scrT,E,s)$  above defines a global chart for the moduli space of stable Floer admissible covers $\cM_{\bfm,\epsilon}(c_x,c_y)$ in class $\epsilon$ and with domain Hurwitz orbit $\bfm$. \end{prop}

\begin{proof}
When $s=0$, the vanishing of $e$ means that the map $u$ solves the usual Floer equation.  Since $\phi = \phi_F$, the Floer curve in $X$ determines the underlying domain curve. The matrix of inner products is unitary, so dividing out by the gauge group shows that $s^{-1}(0)/G = \cM_{\bfm,\epsilon}(c_x,c_y)$.
\end{proof}

By taking the union over all $\epsilon,\bfm$, we get a global chart $\bT(c_x,c_z)$ of $\cM(c_x,c_y)$.

\subsection{A zig-zag of global charts}\label{s:zigzag}

\emph{We explain how the global chart of a boundary stratum is related to a fibre product of global charts by a zig-zag.}

By Proposition \ref{p:singlechart}, we can construct a global chart $\bT(c_x,c_z)$ of $\cM(c_x,c_z)$.
It induces a global chart $\partial^{c_xc_yc_z}\bT(c_x,c_z)$ of $\partial^{c_xc_yc_z}\cM(c_x,c_z)$.
The latter space is a disjoint union
\[
\partial^{c_xc_yc_z}\cM(c_x,c_z)=\sqcup \partial^{c_xc_yc_z}_{\bfm_0,\bfm_1,\epsilon_0,\epsilon_1}\cM(c_x,c_z)
\]
where $\partial^{c_xc_yc_z}_{\bfm_0,\bfm_1,\epsilon_0,\epsilon_1}\cM(c_x,c_z)$ consists of two (possibly broken) Floer trajectories, one from $c_x$ to $c_y$ in the homotopy class $\epsilon_0$ and with Hurwitz orbit $\bfm_0$, the other from $c_y$ to $c_z$ in the homotopy class $\epsilon_1$ and with Hurwitz orbit $\bfm_1$.
We have a further decomposition 
\begin{align}\label{eq:furtherDecom}
    \partial^{c_xc_yc_z}_{\bfm_0,\bfm_1,\epsilon_0,\epsilon_1}\cM(c_x,c_z)= \sqcup_{\bfm} \partial^{c_xc_yc_z,\bfm}_{\bfm_0,\bfm_1,\epsilon_0,\epsilon_1}\cM(c_x,c_z)
\end{align}
where 
\[
\partial^{c_xc_yc_z,\bfm}_{\bfm_0,\bfm_1,\epsilon_0,\epsilon_1}\cM(c_x,c_z):=\partial^{c_xc_yc_z}_{\bfm_0,\bfm_1,\epsilon_0,\epsilon_1}\cM(c_x,c_z) \cap \cM_{\bfm,\epsilon_0+\epsilon_1}(c_x,c_z).
\]
Indeed, the Hurwitz orbits $\bfm_0, \bfm_1$ are not enough to determine the Hurwitz orbit of the glued curve if we don't specify the gluing: for instance, it is straightforward to write down two disconnected Riemann surfaces and two different gluings of the components with different global topologies. But the admissible cover $\Sigma^\dagger$ associated to any element of $\partial^{c_xc_yc_z}_{\bfm_0,\bfm_1,\epsilon_0,\epsilon_1}\cM(c_x,c_z)$ already fixes an identification of the nodes and hence determines $\bfm$.

%





We have a similar decomposition 
\[
\partial^{c_xc_yc_z}\bT(c_x,c_z)=\sqcup \partial^{c_xc_yc_z}_{\bfm_0,\bfm_1,\epsilon_0,\epsilon_1}\bT(c_x,c_z)
=\sqcup \partial^{c_xc_yc_z,\bfm}_{\bfm_0,\bfm_1,\epsilon_0,\epsilon_1}\bT(c_x,c_z).
\]

On the other hand, since
\begin{align}\label{eq:boundaryisom}
\partial^{c_xc_yc_z}_{\bfm_0,\bfm_1,\epsilon_0,\epsilon_1}\cM(c_x,c_z)= \cM_{\bfm_0}(c_x,c_y;\epsilon_0) \times_{[pt/\Gamma_y]} \cM_{\bfm_1}(c_y,c_z;\epsilon_1)    
\end{align}
the global charts $\bT_{\bfm_0}(c_x,c_y;\epsilon_0)$ and $\bT_{\bfm_1}(c_y,c_z;\epsilon_1)$ of $\cM_{\bfm_0}(c_x,c_y;\epsilon_0)$ and $\cM_{\bfm_1}(c_y,c_z;\epsilon_1)$ define a fibre product global chart 
$\bT_{\bfm_0}(c_x,c_y;\epsilon_0) \times_{[pt/\Gamma_y]} \bT_{\bfm_1}(c_y,c_z;\epsilon_1)$
of $\partial^{c_xc_yc_z}_{\bfm_0,\bfm_1,\epsilon_0,\epsilon_1}\cM(c_x,c_z)$ (see Equation \eqref{eq:fibreproduct}).
By \eqref{eq:furtherDecom}, we can further decompose $\bT_{\bfm_0}(c_x,c_y;\epsilon_0) \times_{[pt/\Gamma_y]} \bT_{\bfm_1}(c_y,c_z;\epsilon_1)$ accordingly, which corresponds to decomposing $\Gamma_y$ into equivalence classes.
We denote the decomposition by 
\[
\bT_{\bfm_0}(c_x,c_y;\epsilon_0) \times_{[pt/\Gamma_y]} \bT_{\bfm_1}(c_y,c_z;\epsilon_1) = \sqcup_{\bfm} (\bT_{\bfm_0}(c_x,c_y;\epsilon_0) \times_{[pt/\Gamma_y]}^\bfm \bT_{\bfm_1}(c_y,c_z;\epsilon_1))
\]

The main result (or construction) of this subsection is the following; see \cite[Proposition 4.68]{AMS2}  for an analogous statement for Gromov-Witten invariants.

\begin{prop}\label{p:zigzag}
    There is a $[0,1]$-parametrized family of global charts $(\bT_t=(G,\scrT,E,s_t))_{t \in [0,1]}$ such that $\bT_0$ is an iterative stablization of $\partial^{c_xc_yc_z,\bfm}_{ \bfm_0,\bfm_1,\epsilon_0,\epsilon_1}\bT(c_x,c_z)$ and $\bT_1$  is an iterative stablization of $\bT_{\bfm_0}(c_x,c_y;\epsilon_0) \times_{[pt/\Gamma_y]}^{\bfm} \bT_{\bfm_0}(c_y,c_z;\epsilon_1)$. 
\end{prop}

First, we recall that the global chart $\partial^{c_xc_yc_z}\bT(c_x,c_z)$ on $\partial^{c_xc_yc_z,\bfm}_{\bfm_0,\bfm_1,\epsilon_0,\epsilon_1}\cM(c_x,c_z)$ consists of the following data: 
\begin{enumerate}
    \item The compact Lie group $G=U_{\Gamma}(V)$, where $V^*$ is isomorphic to $H^0(L_u)$ for any element $u\in \cM_{\bfm,\epsilon_0+\epsilon_1}(c_x,c_z)$
    \item The thickening space consisting of $(\phi,u,e,F)$ such that $\phi \in \partial_{\bfm_0,\bfm_1} \scrF^\dagger$ so $\Sigma^\dagger_{\phi}$ is of the form\footnote{Recall that $\Sigma^\dagger_{\phi}$ consists of the data of $\Sigma^\dagger_{\phi_0} \vee \Sigma^\dagger_{\phi_1}$ and a preferred base-point preserving equivariant isomorphism of the asymptotic markers above the node.} $\Sigma^\dagger_{\phi_0} \vee \Sigma^\dagger_{\phi_1}$ where the associated Hurwitz orbit of $\Sigma_{\phi_i}$ is $\bfm_i$, a $\Gamma$-equivariant map $u=u_0 \vee u_1:\Sigma^\dagger_{\phi_0} \vee \Sigma^\dagger_{\phi_1} \to X$ so that $u_i$ is in the class $\epsilon_i$, $e \in V_{\mu}$, a $\Gamma$-equivariant isomorphism $F:H^0(L_u) \to V^*$ such that \eqref{eqn:diagonal_condition} and \eqref{eqn:Floer equation1} are satisfied, and the matrix associated to $F$ has positive eigenvalues.
    \item The obstruction bundle, with fibre $H \oplus V_{\mu} \oplus \scrH$.
    \item The obstruction section, which sends $(\phi,u,e,F)$ to $(\Phi^{-1}(\phi, \phi_F),e,\exp^{-1}(H_F))$.
\end{enumerate}
By our assumption, $\phi_F$ defines an embedding from $\Sigma^\dagger_{\phi}$ to $\bP(V)$ so it in particular defines an embedding from $\Sigma^\dagger_{\phi_i}$ to $\bP(V)$. This is dual to the natural restriction map $H^0(L_u) \to H^0(L_u|_{\Sigma_i})$, which is surjective.

 On the other hand,
$\bT_{\bfm_0}(c_x,c_y;\epsilon_0) \times_{[pt/\Gamma_y]}^{\bfm} \bT_{\bfm_0}(c_y,c_z;\epsilon_1)$
 consists of the following data: 
\begin{enumerate}
    \item The compact Lie group $G=U_{\Gamma}(V_0) \times U_{\Gamma}(V_1)$, where $V_i^*$ is isomorphic to $H^0(L_{u_i})$ for $u_0\in \cM_{\bfm_0,\epsilon_0}(c_x,c_y)$ and $u_1\in \cM_{\bfm_1,\epsilon_1}(c_y,c_z)$
    \item The thickening space consisting of $(\phi_0,u_0,e_0,F_0)$, $(\phi_1,u_1,e_1,F_1)$ and an element $g \in \Gamma_y$ such that $(\phi_0,u_0,e_0,F_0)$ and $(\phi_1,u_1,e_1,F_1)$ satisfy the conditions explained before, and $g$ is such that under the isomorphism \eqref{eq:boundaryisom}, the resulting glued curve has Hurwitz orbit $\bfm$.
    \item The obstruction bundle with  fibre $\oplus_{i=0,1} H_i \oplus (V_{i})_{\mu} \oplus \scrH_i$
    \item The obstruction section  
    \[ \oplus_{i=0,1}(\phi_i,u_i,e_i,F_i) \ \mapsto \  \oplus_{i=0,1}(\Phi_i^{-1}(\phi_i, \phi_{F_i}),e_i,\exp^{-1}(H_{F_i}))\]
\end{enumerate}

\begin{example}(Comparing the fibre product over $[pt/\Gamma_y]$ with that  over $[pt/(\mathbb{Z}/r\mathbb{Z})]$).
    Suppose that $u_0 \in \cM_{\bfm_0}(c_x,c_y;\epsilon_0)$ and $u_1 \in \cM_{\bfm_1}(c_y,c_z;\epsilon_1)$ both have trivial isotropy group.
    Then there are $|\Gamma_y|$ many points in $\cM_{\bfm_0}(c_x,c_y;\epsilon_0) \times_{[pt/\Gamma_y]} \cM_{\bfm_1}(c_y,c_z;\epsilon_1)$ which lie over $(u_0,u_1) \in \cM_{\bfm_0}(c_x,c_y;\epsilon_0) \times \cM_{\bfm_1}(c_y,c_z;\epsilon_1)$.
    They correspond to $|\Gamma_y|$ many different ways to identify the output asymptote of $u_0$ to the input asymptote of $u_1$, and give $|\Gamma_y|$ different ways to glue  to a $\Gamma$-equivariant map $u_0\#u_1$.

    In terms of global charts, $u_i$ gives $G_i$ many elements $(\phi_{F_i},u_i,0,F_i)$ in the respective thickening. The assumption that $u_i$ has trivial isotropy group implies that the $G_i$ many different choices of $F_i$ are all pairwise inequivalent. 
    Therefore, we have $G_0 \times G_1 \times \Gamma$ many inequivalent elements $(\phi_{F_0},u_0,0,F_0,\phi_{F_1},u_1,0,F_1,g)$ in the fibre product global chart each of which forgets to $(u_0,u_1) \in \cM_{\bfm_0}(c_x,c_y;\epsilon_0) \times \cM_{\bfm_1}(c_y,c_z;\epsilon_1)$.

    On the other hand, consider $[u_0,u_1,g]$ as a point in  $\partial^{c_xc_yc_z}_{\bfm_0,\bfm_1,\epsilon_0,\epsilon_1}\cM(c_x,c_z)$.
    The element $g$ determines how the output punctures of $\Sigma^\dagger_{\phi_1}$ and input punctures of $\Sigma^\dagger_{\phi_2}$ are identified.
More precisely, after we choose an isomorphism \eqref{eq:boundaryisom} that is compatible with the forgetful map to $\cM_{\bfm_0}(c_x,c_y;\epsilon_0) \times \cM_{\bfm_1}(c_y,c_z;\epsilon_1)$, the preimage of $[u_0,u_1,id]$ gives us a way to identify the output punctures of $\Sigma^\dagger_{\phi_0}$ with the input punctures of $\Sigma^\dagger_{\phi_1}$ (including the asymptotic markers). Then the element $[u_0,u_1,g]$ corresponds to twisting the identification by the action of $g$ on the punctures of $\Sigma^\dagger_{\phi_1}$. We denote the admissible cover associated to $g$ by $\Sigma^\dagger_{\phi_g}=\Sigma^\dagger_{\phi_0} \vee_g \Sigma^\dagger_{\phi_1}$.  

    There are two possibilities. The automorphism $g$ either moves the input asymptotic marker of $\Sigma^\dagger_{\phi_1}$ to an asymptotic marker at another input puncture, or to an asymptotic marker at the same input puncture.
    We call the first one type $1$ and the second one type $2$. 
   If $g$ is of type $2$, then the underlying nodal curve $\Sigma_{\phi_g}$ is the same as $\Sigma_{\phi_{id}}$.
   The subgroup of type $2$ elements in $\Gamma_y$ is a cyclic group $C$, which is the isotropy group of the underlying node in $\scrC^{\dagger}_{\phi}$ (cf. Remark \ref{r:isotropy}).

    Together with $u_0,u_1$, we see $g$ determines a map $u_g:=(u_0,u_1):\Sigma^\dagger_{\phi_g} \to X$. 
    Consequently, we have  $G \times |\Gamma_y|$ many different elements $(\phi_{F_g},u_g,0,F_g)$, where $\Sigma^\dagger_{\phi_{F_g}}=\Sigma^\dagger_{\phi_g}$, in the global chart of $\partial^{c_xc_yc_z}_{\bfm_0,\bfm_1,\epsilon_0,\epsilon_1}\cM(c_x,c_z)$.
    We can consider the forgetful map (for $i=0,1$)
    \[
    (\phi_{F_g}:\Sigma^\dagger_{\phi_g} \to \bP(V)) \mapsto (\phi_{F_g,i}:\Sigma^\dagger_{\phi_i} \to \bP(V))
    \]
    The discussion in the previous paragraph implies that $\phi_{F_g}$ has the same forgetful image as $\phi_{F_{id}}$ if and only if  $g \in C$.
    Indeed, since $\phi_{F_g,0}$ and $\phi_{F_g,1}$ are embeddings, they give an identification of the output points of $\Sigma_{\phi_0}$ and input points of $\Sigma_{\phi_1}$ but not an identification of the asymptotic markers at these points.
    In this way, we realize $\{\phi_{F_g}: g \in C\}$ as the fibre product of $\{\phi_{F_{id},0}\} \times_{[pt/C]} \{\phi_{F_{id},1}\}$.
    Therefore, one should think of the $|\Gamma_y|$ many elements $[u_0,u_1,g]$ being partitioned into $|\Gamma_y|/|C|$ many sets, each of which has cardinality $|C|$ and corresponds to an underlying pair of framed curves $(\phi_{F_{g},0},\phi_{F_{g},1})$.
  \end{example}

To relate the two global charts, we consider the doubly framed global chart $\bT$ which consists of the following data (cf. Proposition \ref{p:doubly-closed}):
\begin{enumerate}
    \item The compact Lie group $G=U_{\Gamma}(V)  \times U_{\Gamma}(V_0) \times U_{\Gamma}(V_1)$.
    \item The thickening space consisting of $(\phi,\tilde{\phi}_0,\tilde{\phi}_1,u,e,e_0,e_1,F,F_0,F_1)$ such that the following hold:
\begin{itemize}
\item For $i=0,1$, let 
\[
\scrF_{0,2+h}^{\dagger}([\bP(V)/\Gamma] \times [\bP(V_i)/\Gamma], (*,V_i^*)) \subset \cK_{0,2+h}^{\dagger}([\bP(V)/\Gamma] \times [\bP(V_i)/\Gamma], (*,V_i^*))
\]
consist of elements such that the projection (by forgetting the second and first factor respectively) to $ \cK_{0,2+h}^{\dagger}([\bP(V)/\Gamma]) \times \cK_{0,2+h}^{\dagger}([\bP(V_i)/\Gamma],V_i^*)$ lies in $ \scrF_{0,2+h}^{\dagger}([\bP(V)/\Gamma],R_i) \times \scrF_{0,2+h}^{\dagger}([\bP(V_i)/\Gamma],V_i^*)$ for some representation $R_i$ that is a quotient of $V^*$.
We require that $\tilde{\phi}_i \in \scrF_{0,2+h}^{\dagger}([\bP(V)/\Gamma] \times [\bP(V_i)/\Gamma], (*,V_i^*))$ and denote the image of $\tilde{\phi}_i$ under the forgetful map by $(\tilde{\phi}_i|_{\bP(V)},\tilde{\phi}_i|_{\bP(V_i)})$.
    \item $\phi \in \partial_{\bfm_0,\bfm_1} \scrF_{0,2+h}^{\dagger}([\bP(V)/\Gamma],V^*)$ so it is of the form $\phi=\phi_0 \vee \phi_1$. We require that $\tilde{\phi}_i|_{\bP(V)}=\phi_i$.
    \item There are finite dimensional approximation schemes
    $W_{\mu}$ of $(\omega_{\cC^{orb}/\scrF_{0,2+h}^{\dagger}([\bP(V)/\Gamma],V^*)} \otimes TX)$ and  $(W_i)_{\mu}$ of $(\omega_{\cC^{orb}/\scrF_{0,2+h}( [\bP(V_i)/\Gamma], V_i^*)} \otimes TX)$ and the corresponding finite dimensional approximation schemes
    $\lambda_{\mu}:V_{\mu} \to (\omega_{\cC^{orb}/\scrF_{0,2+h}^{\dagger}([\bP(V)/\Gamma],V^*)} \otimes TX)^{\Gamma}$ and $(\lambda_{i})_{\mu}: (V_{i})_{\mu} \to (\omega_{\cC^{orb}/\scrF_{0,2+h}( [\bP(V_i)/\Gamma], V_i^*)} \otimes TX)^{\Gamma}$. We require that $e \in V_{\mu}$ and $e_i \in (V_i)_{\mu}$. 
    \item $u=u_0 \vee u_1:\Sigma_{\phi}^\dagger=\Sigma_{\phi_0}^\dagger\vee \Sigma_{\phi_1}^\dagger \to X$ is a $\Gamma$ equivariant map satisfying the Floer equation (i.e. $(dv-X_{H/\Gamma} \otimes dt)^{0,1}_{J/\Gamma}+\lambda_{\mu}(e)+(\lambda_1)_{\mu}(e_1)+(\lambda_2)_{\mu}(e_2)=0$ with finite energy and exponential convergence) on cylindrical components and holomorphic equation on the spherical components such  that $u_i$ is in the class $\epsilon_i$. More precisely, to make sense of $(\lambda_i)_{\mu}(e_i)$, we need to use $\tilde{\phi}_i$ to embed $\Sigma_{\phi_i}^\dagger$ to $[\bP(V)/\Gamma] \times [\bP(V_i)/\Gamma]$ and then project to $[\bP(V_i)/\Gamma]$.
    \item $F:H^0(L_u) \to V^*$ and $F_i:H^0(L_{u_i}) \to V_i^*$ are $\Gamma$-equivariant isomorphisms such that the associated matrices $H_F$, $H_{F_i}$ have positive eigenvalues. Here $H_F$ is defined with respect to the fibrewise metric on $\Sigma^\dagger_\phi$ (i.e. fibre of the universal curve over $\phi$), and  $H_{F_i}$ is defined with respect to the fibrewise metric on $\Sigma^\dagger_{\tilde{\phi}_i|_{\bP(V_i)}}$ (i.e. as a curve over a point in $\scrF_{0,2+h}^{\dagger}([\bP(V_i)/\Gamma],V_i^*)$).
    They induce the associated maps $\phi_F: \Sigma^\dagger_{\phi} \to \bP(V)$, $\phi_{F_i}: \Sigma^\dagger_{\phi_i} \to \bP(V_i)$ and the restrictions $\phi_F|_{\Sigma_{\phi_i}^\dagger}: \Sigma^\dagger_{\phi_i} \to \bP(V)$.
    \item $\phi_F$ and $(\phi_{F_i})$ lie in neighborhoods of the respective diagonals so that $\Phi^{-1}(\phi_F,\phi)$ and $\Phi_i^{-1}(\phi_{F_i}, \tilde{\phi}_i)$ are well-defined
\end{itemize}
    
    \item The obstruction bundle has fibre $H \oplus V_{\mu} \oplus \scrH \oplus  \oplus_{i=0,1} H_i \oplus (V_i)_{\mu} \oplus \scrH_i$. Here $H$ and $H_i$ are the fibres of the normal bundle of the diagonal inside the square of  $\scrF_{0,2+h}^{\dagger}([\bP(V)/\Gamma],V^*)$, and  $\scrF_{0,2+h}^{\dagger}( [\bP(V_i)/\Gamma], V_i^*)$, respectively. The vector spaces $\scrH$ and $\scrH_i$ are subspaces of Hermitian matrices associated to $U_{\Gamma}(V)$ and $U_{\Gamma}(V_i)$.
    \item The obstruction section is the obvious map.
\end{enumerate}

\begin{remark}\label{r:smoothfibreproduct}
The space of $(\phi,\tilde{\phi}_0,\tilde{\phi}_1)$ such that  $\phi \in \partial_{\bfm_0,\bfm_1} \scrF_{0,2+h}^{\dagger}([\bP(V)/\Gamma],V^*)$,  $ \tilde{\phi}_i \in \scrF_{0,2+h}^{\dagger}([\bP(V)/\Gamma] \times [\bP(V_i)/\Gamma], (*,V_i^*))$
and $\tilde{\phi}_i|_{\bP(V)}=\phi_i$ is smooth because it is an iterative fibre product
over $\scrF_{0,2+h}^{\dagger}([\bP(V)/\Gamma], R_i)$
and the map $\scrF_{0,2+h}^{\dagger}([\bP(V)/\Gamma] \times [\bP(V_i)/\Gamma], (*,V_i^*)) \to \scrF_{0,2+h}^{\dagger}([\bP(V)/\Gamma], R_i)$ is submersive.
\end{remark}

\begin{remark}\label{r:gluingdata}
    The related \cite[Proposition 4.68]{AMS2} uses a slightly different family of pairs $(\tilde{\phi}_0,\tilde{\phi}_1)$ for which the output marked point of $\tilde{\phi}_0|_{\bP(V)}$ coincides with the input marked point of $\tilde{\phi}_1|_{\bP(V)}$.
    We cannot do the same here because $\phi$ remembers a choice of identification of the asymptotic markers, so we cannot recover $\phi$ from  $\tilde{\phi}_0$ and $\tilde{\phi}_1$.
\end{remark}


Following the same reasoning as in Lemma \ref{l:mfd} and Proposition \ref{p:singlechart}, we conclude the following:

\begin{lemma}
    For sufficiently large $k$ and $\mu$, $\bT$ is a global chart of $\partial^{c_xc_yc_z,\bfm}_{ \bfm_0,\bfm_1,\epsilon_0,\epsilon_1}\cM(c_x,c_z)$.
\end{lemma}

We need two more variants of $\bT$:
\begin{enumerate}
    \item The thickening space of the first one $\bT_{\partial}$ consists of $(\phi,\tilde{\phi}_0,\tilde{\phi}_1,u,e,F,F_0,F_1)$. That is, we don't use $e_i \in (V_i)_{\mu}$ to perturb the equation and $(V_i)_{\mu}$ does not show up in the obstruction bundle.
    \item The thickening space of the second one $\bT_{\times}$ consists of $(\phi,\tilde{\phi}_0,\tilde{\phi}_1,u,e_0,e_1,F,F_0,F_1)$. That is, we don't use $e \in V_{\mu}$ to perturb the equation and $V_{\mu}$ does not show up in the obstruction bundle.
\end{enumerate}

When $k$ and $\mu$ are sufficiently large, $\bT_{\partial}$ and $\bT_{\times}$ are global charts of $\partial^{c_xc_yc_z,\bfm}_{ \bfm_0,\bfm_1,\epsilon_0,\epsilon_1}\cM(c_x,c_z)$ as well.

We have a  forgetful map $\pi_{\partial}$ from the thickening space of $\bT_{\partial}$ to that of $ \partial^{c_xc_yc_z,\bfm}_{\bfm_0,\bfm_1}\bT(c_x,c_z)$:
\[
\pi_{\partial}:(\phi,\tilde{\phi}_0,\tilde{\phi}_1,u,e,F,F_0,F_1) \mapsto (\phi,u,e,F).
\]
Similarly, we have another forgetful map $\pi_{\times}$ from the thickening space of $\bT_{\times}$ to the thickening space of $\bT_{\bfm_0}(c_x,c_y;\epsilon_0) \times_{[pt/\Gamma_y]}^\bfm \bT_{\bfm_1}(c_y,c_z;\epsilon_1)$:
\[
\pi_{\times}: (\phi,\tilde{\phi}_0,\tilde{\phi}_1,u,e_0,e_1,F,F_0,F_1) \mapsto (\tilde{\phi}_0|_{\bP(V_0)},u_0,e_0,F_0, \tilde{\phi}_1|_{\bP(V_1)},u_1,e_1,F_1,g)
\]
where $g \in \Gamma_y$  is determined by requiring that $\Sigma^\dagger_{\phi}=\Sigma^\dagger_{\phi_1} \vee_g \Sigma^\dagger_{\phi_2}$.

\begin{proof}[Proof of Proposition \ref{p:zigzag}]

Assume that $k,\mu$ are large enough such that $\bT$, $\partial^{c_xc_yc_z,\bfm}_{\bfm_0,\bfm_1}\bT(c_x,c_z)$ and
$\bT_{\bfm_0}(c_x,c_y;\epsilon_0) \times_{[pt/\Gamma_y]}^\bfm \bT_{\bfm_1}(c_y,c_z;\epsilon_1)$ are Kuranishi charts.

The map $\pi_{\partial}$ is simply forgetting the framing data $F_0$ and $F_1$ so $\bT_{\partial}$ is an iterative stabilisation of $\partial^{c_xc_yc_z,\bfm}_{\bfm_0,\bfm_1}\bT(c_x,c_z)$ (see \cite[Proposition 4.68]{AMS2} for a similar argument) modulo homotoping the sections.
Similarly, $\bT_{\times}$ is an iterative stabilisation of $\bT_{\bfm_0}(c_x,c_y;\epsilon_0) \times_{[pt/\Gamma_y]}^\bfm \bT_{\bfm_1}(c_y,c_z;\epsilon_1)$ modulo homotoping the sections.

On the other hand, the only difference between $\bT$, $\bT_{\partial}$ and $\bT_{\times}$ is that the obstruction space of $\bT$ contains the ones of $\bT_{\partial}$ and $\bT_{\times}$.
A similar argument as in \cite{AMS2} shows that $\bT$ is an iterative stabilisation of $\bT_{\partial}$ and of $\bT_{\times}$.

    Since the space of equivariant sections is a vector space and hence connected, we can find a homotopy of the sections relating the given iterative stabilisation of $\partial^{c_xc_yc_z,\bfm}_{\bfm_0,\bfm_1}\bT(c_x,c_z)$ and $\bT_{\bfm_0}(c_x,c_y;\epsilon_0) \times_{[pt/\Gamma_y]}^\bfm \bT_{\bfm_1}(c_y,c_z;\epsilon_1)$.
\end{proof}

We also need to describe similar global charts of $\partial^{c_xc_yc_zc_w}_{\bfm_0,\bfm_1,\bfm_2,\epsilon_0,\epsilon_1,\epsilon_2}\cM(c_x,c_w)$.
The notation means that the stable Floer cylinder from $c_x$ to $c_y$, from $c_y$ to $c_z$ and from $c_z$ to $c_w$ have Hurwitz orbit and classes being $(\bfm_0,\epsilon_0)$, $(\bfm_1,\epsilon_1)$ and $(\bfm_2,\epsilon_2)$, respectively.
We denote the Hurwitz orbit of gluing the first two curves, last two curves and all curves by $\bfm_{01}$, $\bfm_{12}$ and $\bfm$, respectively.
We have a corresponding decomposition
\[
\partial^{c_xc_yc_zc_w}_{\bfm_0,\bfm_1,\bfm_2,\epsilon_0,\epsilon_1,\epsilon_2}\cM(c_x,c_w)= \sqcup_{\bfm,\bfm_{01}, \bfm_{12}}  \partial^{c_xc_yc_zc_w, \bfm,\bfm_{01}, \bfm_{12}}_{\bfm_0,\bfm_1,\bfm_2,\epsilon_0,\epsilon_1,\epsilon_2}\cM(c_x,c_w)
\]
We have isomorphisms
\begin{align*}
\partial^{c_xc_yc_zc_w, \bfm,\bfm_{01}, \bfm_{12}}_{\bfm_0,\bfm_1,\bfm_2,\epsilon_0,\epsilon_1,\epsilon_2}\cM(c_x,c_w)
&=\partial^{c_xc_yc_z, \bfm_{01}}_{\bfm_0,\bfm_1,\epsilon_0,\epsilon_1}\cM(c_x,c_z) \times^{\bfm}_{[pt/\Gamma_z]} \cM_{\bfm_2,\epsilon_2}(c_z,c_w)\\
&=\cM_{\bfm_0,\epsilon_0}(c_x,c_y) \times^{\bfm}_{[pt/\Gamma_y]} \partial^{c_xc_yc_z, \bfm_{12}}_{\bfm_1,\bfm_2,\epsilon_1,\epsilon_2}\cM(c_y,c_w)\\
&=(\cM_{\bfm_0,\epsilon_0}(c_x,c_y) \times^{\bfm_{01}}_{[pt/\Gamma_y]} \cM_{\bfm_1,\epsilon_1}(c_y,c_z) \times^{\bfm_{12}}_{[pt/\Gamma_z]} \cM_{\bfm_2,\epsilon_2}(c_z,c_w))^{\bfm}
\end{align*}
where $(\cdot)^{\bfm}$ in the last term means that we consider the subspace of those elements such that the Hurwitz orbit after gluing all $3$ curves is $\bfm$.
We want to describe an analogue of the $\bT$ in Proposition \ref{p:zigzag} for $\partial^{c_xc_yc_zc_w, \bfm,\bfm_{01}, \bfm_{12}}_{\bfm_0,\bfm_1,\bfm_2,\epsilon_0,\epsilon_1,\epsilon_2}\cM(c_x,c_w)$ such that up to homotoping the section, it is an iterative stabilisation of various fibre product type Kuranishi charts. It consists of the following data.

\begin{enumerate}
    \item The compact Lie group $G=U_{\Gamma}(V) \times U_{\Gamma}(V_{01}) \times U_{\Gamma}(V_{12})  \times U_{\Gamma}(V_0) \times U_{\Gamma}(V_1) \times U_{\Gamma}(V_2)$.
    \item The thickening space consisting of \[
    (\phi,\phi_{01},\phi_{12},\tilde{\phi}_0,\tilde{\phi}_1,\tilde{\phi}_2,u,e,e_{01},e_{12},e_0,e_1,e_2,F,F_{01},F_{12},F_0,F_1,F_2)
    \]
    such that the following hold:
\begin{itemize}
\item The $\phi$'s define embeddings to $\bP(V)$,  $\bP(V) \times \bP(V_{01})$, $\bP(V) \times \bP(V_{12})$, 
$\bP(V) \times \bP(V_{01}) \times \bP(V_0)$, $\bP(V) \times \bP(V_{01}) \times \bP(V_{12}) \times  \bP(V_1)$
and $\bP(V) \times \bP(V_{12}) \times \bP(V_2)$ respectively.
    \item There are obvious equalities among different $\phi$'s when restricted to different projective spaces (i.e. $\phi_0|_{\Sigma_0^\dagger \vee \Sigma_1^\dagger}=\tilde{\phi}_{01}|_{\bP(V)}$, $\phi_0|_{\Sigma_1^\dagger \vee \Sigma_2^\dagger}=\tilde{\phi}_{12}|_{\bP(V)}$, $\phi_0|_{\Sigma_j^\dagger}=\tilde{\phi}_j|_{\bP(V)}$ for $j=0,1,2$, $\phi_{01}|_{\Sigma_j^\dagger}=\tilde{\phi}_j|_{\bP(V_{01})}$ for $j=0,1$ and $\phi_{12}|_{\Sigma_j^\dagger}=\tilde{\phi}_j|_{\bP(V_{12})}$ for $j=1,2$). 
    \item There are finite dimensional approximation schemes
     for $(\omega_{\cC^{orb}/\scrF_{0,2+h}( [\bP/\Gamma], d_i)} \otimes TX)^{\Gamma}$ for $\bP=\bP(V), \bP(V_{01}), \bP(V_{12}), \bP(V_0), \bP(V_1), \bP(V_2)$, and $e$'s are elements in the respective $V_\mu$'s
    \item $u:\Sigma_{\phi}^\dagger \to X$ is a $\Gamma$ equivariant map satisfying the Floer equation (i.e. $(dv-X_{H/\Gamma} \otimes dt)^{0,1}_{J/\Gamma}+\lambda_{\mu}(e)+(\lambda_{01})_{\mu}(e_{01})+(\lambda_{12})_{\mu}(e_{12})+\dots=0$) on cylindrical components and holomorphic equation on the spherical components such  that $u_i$ is in the class $\epsilon_i$.
    \item The $F$'s are $\Gamma$-equivariant isomorphisms such that the associated matrices have positive eigenvalues. 
    They induce the associaed maps $\phi_F$'s as well as their restrictions to components.
    \item The $\{\phi_F\}$'s and their restrictions lie in respective neighbourhoods of the diagonals
\end{itemize}
    \item The obstruction bundle and the obstruction section are defined analogously.
\end{enumerate}

\begin{remark}
    Similar to Remark \ref{r:smoothfibreproduct}, the space of $(\phi,\tilde{\phi}_{01},\tilde{\phi}_{12},\tilde{\phi}_0,\tilde{\phi}_1,\tilde{\phi}_2)$ with equalities among different $\phi$'s is smooth because it is an iterative fibre product over the spaces
    \begin{align*}
       & \partial_{\bfm_0,\bfm_1} \scrF_{0,2+h}^{\dagger}([\bP(V)/\Gamma],R_{01}), \quad  \partial_{\bfm_1,\bfm_2} \scrF_{0,2+h}^{\dagger}([\bP(V)/\Gamma],R_{12}), \\
     &\scrF_{0,2+h}^{\dagger}([\bP(V)/\Gamma],R_j) \text{ for } j=0,1,2, \quad
     \scrF_{0,2+h}^{\dagger}([\bP(V_{01})/\Gamma],R_{01,j}) \text{ for } j=0,1, \\
     &\text{and }  \scrF_{0,2+h}^{\dagger}([\bP(V_{12})/\Gamma],R_{12,j}) \text{ for } j=1,2
    \end{align*}
    for some appropriate $\Gamma$-representations $R_{01}, \dots, R_{12,j}$,
      and for each fibre product, there is a factor which maps submersively to the base.
      Here, we use that every boundary stratum of $\scrF_{0,2+h}^{\dagger}$ is smooth.
\end{remark}

\begin{remark}
Similar to Remark \ref{r:gluingdata}, $\tilde{\phi}_0|_{\bP(V_{01})}$ and $\tilde{\phi}_1|_{\bP(V_{01})}$ do not determine $\tilde{\phi}_{01}|_{\bP(V_{01})}$ because $\tilde{\phi}_{01}$ remembers how the asymptotic markers are identified.
On the other hand, both $\phi$ and $\tilde{\phi}_{01}$ remember how to glue $\Sigma_0^\dagger$ and $\Sigma_1^\dagger$ so the equality $\phi_0|_{\Sigma_0^\dagger \vee \Sigma_1^\dagger}=\tilde{\phi}_{01}|_{\bP(V)}$ is to ensure that the gluing data are the same.
\end{remark}

\begin{prop}\label{p:associativity}
   For sufficiently large $k, \mu$, the data above define a global Kuranishi chart of $\partial^{c_xc_yc_zc_w}_{\bfm_0,\bfm_1,\bfm_2,\epsilon_0,\epsilon_1,\epsilon_2}\cM(c_x,c_w)$ and up to homotopy of sections, it is an iterative stablization of the fibre product global charts.
\end{prop}

\begin{proof}
    The proof is similar to that given  before. For example, to see that it is an iterative stabilisation of the fibre product global chart of $(\cM_{\bfm_0,\epsilon_0}(c_x,c_y) \times^{\bfm_{01}}_{[pt/\Gamma_y]} \cM_{\bfm_1,\epsilon_1}(c_y,c_z) \times^{\bfm_{12}}_{[pt/\Gamma_z]} \cM_{\bfm_2,\epsilon_2}(c_z,c_w))^{\bfm}$, we first define a variant which doesn't take into account $e,e_{01},e_{12}$ so that we have a forgetful map to the thickening space of the fibre product global chart and then run the argument in \cite[Proposition 4.68]{AMS2}.
\end{proof}

\subsection{Inductive procedure for global charts }\label{s:system}

\emph{We explain the inductive construction of global chart lifts. Our approach is different from \cite{BX}.}

We fix $H$, $J_Y$ and the bulk $\mathfrak{b}$ as before. We have a Hamiltonian Floer order marked flow category $(\cC(H_Y,\mathfrak{b});J_Y)$ defined in Section \ref{s:orbifoldFloer}.

\begin{prop}\label{p:globalliftexist}
The Hamiltonian  Floer ordered marked flow category $(\cC(H_Y,\mathfrak{b});J_Y)$ admits a global chart lift (see Definition \ref{d:coherenthomotopy} and \ref{d:global lift}). 
\end{prop}

\begin{proof}
We choose an $\Omega$ as in Lemma \ref{l:Omega}.
We choose a $\Gamma$-equivariant Hermitian line bundle $L$ on $X$ with curvature $\Omega$.
By Lemma \ref{lem:positive}, the action difference $\cA_{\mathbb{Z}}(c_y)-\cA_{\mathbb{Z}}(c_x)$ of $\cM(c_x,c_y)$ is always a positive integer. 
We are going to define the global chart lift of$(\cC(H_Y,\mathfrak{b});J_Y)$ inductively on the action difference.

Since $X$ is compact and $H$ is non-degenerate, there are only finitely many $1$-periodic orbits of $H$.
Therefore, modulo the diagonal action of $\Pi$, there are only finitely many pairs $(c_x,c_y)$ such that  $\cA_{\mathbb{Z}}(c_y)-\cA_{\mathbb{Z}}(c_x)=1$.
For each such pair $(c_x,c_y)$, since the valuation of the bulk $\mathfrak{b}$ is strictly positive, there are only finitely many Hurwitz orbits $\bfm$ such that $\cM_{\bfm}(c_x,c_y)$ is non-empty.
For each $(c_x,c_y, \bfm)$, by Gromov compactness, there are only finitely many different classes $\epsilon$ for which 
$\cM_{\bfm, \epsilon}(c_x,c_y)$ is non-empty.
As a result, by Lemma \ref{l:regular}, there is $k_1>0$ such that $H^1(L_u)=0$ for all $u \in \cM_{\bfm,\epsilon}(c_x,c_y)$ with $\cA_{\mathbb{Z}}(c_y)-\cA_{\mathbb{Z}}(c_x)=1$.
By possibly choosing a larger $k_1$, we can moreover ensure that the sections of $H^0(L_u)$ define  
an embedding of $\Sigma^{\dagger}$ into $\bP(H^0(L_u)^*)$, because the degree of $L_u$ on every irreducible component strictly increase when $k$ increases (cf. proof of Proposition \ref{prop:chart for closed curves}).

Fix such $k_1$. There are finitely many $(c_x,c_y, \bfm, \epsilon)$ (modulo the $\Pi$ action) such that $\cA_{\mathbb{Z}}(c_y)-\cA_{\mathbb{Z}}(c_x)=1$ and $\cM_{\bfm,\epsilon}(c_x,c_y) \neq \emptyset$, giving us finitely many isomorphism types of faithful $\Gamma$-representations $H^0(L_u)$. 
For each isomorphism type, we pick the representative $V$ we have chosen in $\Theta$ (recall $\Theta$ was defined before Section \ref{ss:framecurves}).
Then we have the associated $\scrF^{\dagger}_{0,2+h}(V^*) \subset \cK([\bP(V)/\Gamma],V^*)$, where $h$ is the number of points in $\bfm$, $V \simeq H^0(L_u)$ for $u \in \cM_{\bfm,\epsilon}(c_x,c_y)$ and $d+1=\dim V$.
We choose a finite-dimensional  approximation scheme of $\omega_{\cC^{orb}/\scrF^{\dagger}} \otimes TX$ and $(\omega_{\cC^{orb}/\scrF^{\dagger}} \otimes TX)^{\Gamma}$ as in Section \ref{s:globalchart_cylinder}.
We also choose an equivariant fibrewise metric of $\cC^{orb}/\scrF^{\dagger}$, a neighborhood of the diagonal embedding of $\scrF^{\dagger}_{0,2+h}(V^*)$ and its identification $\Phi$ with the normal bundle as in \ref{s:globalchart_cylinder} to run the global chart construction.
Since $k_1$ is already chosen such that sections of $H^0(L_u)$ define 
an embedding of $\Sigma^{\dagger}$ into $\bP(H^0(L_u)^*)$ for all $u \in \cM_{\bfm,\epsilon}(c_x,c_y)$, by Proposition \ref{p:singlechart}, we get a global chart of $\cM_{\bfm,\epsilon}(c_x,c_y)$ when $\mu$ is large enough.
The boundary of $\cM_{\bfm,\epsilon}(c_x,c_y)$ is empty because $\cA_{\mathbb{Z}}$ takes integer values and it is additive (Lemma \ref{l:additivity}).
We require that the construction is strictly $\Pi$-equivariant with respect to the diagonal $\Pi$ action on pairs $(c_x,c_y)$.

Then we consider moduli spaces with action difference $\cA_{\mathbb{Z}}(c_y)-\cA_{\mathbb{Z}}(c_x)=2$.
By a similar reason, there are finitely many $(c_x,c_y, \bfm, \epsilon)$ (modulo the $\Pi$ action) such that $\cA_{\mathbb{Z}}(c_y)-\cA_{\mathbb{Z}}(c_x)=2$ and $\cM_{\bfm,\epsilon}(c_x,c_y) \neq \emptyset$.
We fix $k_2$ large enough that $H^0(L_u)$ defines 
an embedding of $\Sigma^{\dagger}$ into $\bP(H^0(L_u)^*)$ for all $u \in \cM_{\bfm,\epsilon}(c_x,c_y)$ with $\cA_{\mathbb{Z}}(c_y)-\cA_{\mathbb{Z}}(c_x)=2$.
This gives us finitely many isomorphism types of faithful representations, and for each of them, we choose the representative in $\Theta$.
We run the global chart construction as before.
The boundary of $\cM_{\bfm,\epsilon}(c_x,c_z)$ with $\cA_{\mathbb{Z}}(c_z)-\cA_{\mathbb{Z}}(c_x)=2$ is covered by fibred products of
$\cM_{\bfm_0,\epsilon_0}(c_x,c_y)$ and $\cM_{\bfm_1,\epsilon_1}(c_y,c_z)$ with $\cA_{\mathbb{Z}}(c_y)-\cA_{\mathbb{Z}}(c_x)=\cA_{\mathbb{Z}}(c_z)-\cA_{\mathbb{Z}}(c_y)=1$ (see Lemma \ref{l:boundarycover} and \ref{l:additivity}).
For each boundary stratum of $\cM_{\bfm,\epsilon}(c_x,c_z)$,
we apply Proposition \ref{p:zigzag} to construct a Kuranishi chart together with a $[0,1]$-parametrized family of sections such that the two ends are iterative stabilisations of the global chart of the boundary stratum of $\cM_{\bfm,\epsilon}(c_x,c_z)$ and the fibre product Kuranishi chart, respectively.
These data constitute a global chart lift of moduli spaces with action difference $\le 2$ (cf. Example \ref{eq:zig-zag}).

For moduli spaces with action difference $3$, we choose $k_3$ large and run the global chart construction as before.
The boundary is covered by fibred products of moduli spaces with action difference $(2,1)$, $(1,2)$ or $(1,1,1)$.
For the boundary strata covered by fibred products of moduli spaces with action difference $(2,1)$ and $(1,2)$, we construct a Kuranishi chart with a $[0,1]$-parametrized family as in Proposition \ref{p:zigzag} such that the two ends are iterative stabilisations of the respective Kuranishi charts as before.
Along the boundary strata covered by fibred products of moduli spaces with action difference $(1,1,1)$, we apply Proposition \ref{p:associativity} to construct a Kuranishi chart which, by possibly changing the section, is an iterative stabilisation of every relevant fibred product Kuranishi chart. We can equip it with a $[0,1] \times [0,1]$-parametrized family of sections so that the respective boundary components (induced by boundary strata of $[0,1] \times [0,1]$) are iterative stabilisations of the   corresponding ($[0,1]$-parametrised family of) fibre product Kuranishi charts (cf. Example \ref{eg:homotopyassociativity}).
These data constitute a global chart lift of moduli spaces with action difference $\le 3$. 

For moduli space with action difference $>3$, we run the analogous procedure. We require that the construction is strictly $\Pi$-equivariant.
This gives us a global chart lift of $(\cC(H_Y,\mathfrak{b});J_Y)$.
\end{proof}

Up to this point we have considered global charts in the topological category.  We need  smooth global charts to talk about transversality for weighted branched multisections when extracting a chain complex as in Section \ref{s:multivalue}.  

\begin{prop}\label{p:smoothlift}
The Hamiltonian Floer ordered marked flow category admits a smooth global chart lift.
\end{prop}

\begin{proof}
This is the ordered marked flow category version of a result due independently to Bai and Xu \cite{BX} and to Rezchikov \cite{Rez}, and follows the template explained in Proposition \ref{prop:stable_smoothing} and Lemma \ref{lem:smoothing}. The key input is that each thickening $\scrT = \scrT_{pq}^{\alpha}$ is a topological $G$-manifold (for a suitable Lie group $G$) with a $G$-equivariant $C^1_{loc}$-submersion $p: \scrT \to \scrF$ to a smooth quasi-projective variety, which is the corresponding space of domains (pre-stable orbifold cylinders in some $\bP(V)/\Gamma)$ with labelled perhaps stacky marked points).  This relies on the  fact that gluing theory gives the space of Floer solutions with a fixed (possibly reducible and unstable) domain  a smooth structure, which also applies to orbifold stable maps.  
\end{proof}

\begin{remark}
One can further mollify to make the obstruction sections smooth (cf. \cite{AMS2}), but since we will replace them with weighted branched sections anyway, we do not need this.
\end{remark}

We are now ready to give the following definition, which is part of Theorem \ref{t:main}.
Let $H,J$ be $\Gamma$-equivariant as above.
    For a bulk $\mathfrak{b}=\sum_{(g)} \frak{v}_{(g)} PD([X^g/C(g)])$ (cf. Hypothesis \ref{h:sector}), let 
    $\bfa:\|\Gamma\|^\circ \to \Lambda_{\ge 0}$ be
    $\bfa(g):=\frak{v}_{(g)}/\val(\frak{v}_{(g)})$.
    
\begin{definition}\label{d:chaincomplex2}
We define $CF(H/\Gamma,J/\Gamma;\mathfrak{b})$ to be the $\bfa$-deformed chain complex obtained by applying the procedure of Section \ref{s:multivalue} (more precisely, Remark \ref{r:deform} and Proposition \ref{p:dsquare=0})
to a smooth global chart lift as obtained from Proposition \ref{p:smoothlift}.
\end{definition}

While the flow category $\cC(H/\Gamma)$ only depends on $H/\Gamma$ and $J/\Gamma$, several auxiliary choices were made in the smooth global chart lift construction. (We defer the stronger invariance statement related to changing the Hamiltonian  to the following section.)

\begin{prop}\label{p:chaincomplex2indep}
    The chain complex $CF(H/\Gamma,J/\Gamma;\mathfrak{b})$ is independent of the auxiliary choices made to build the smooth global chart lift, up to action preserving chain homotopy equivalence. 
\end{prop}

\begin{proof}
    
It follows by a doubly framed trick as in Proposition \ref{p:doubly-closed}.
Given two collections of auxiliary choices $\mathcal{D}$ and $\mathcal{D}'$, we can run Proposition \ref{p:globalliftexist} for $\mathcal{D}$ and $\mathcal{D}'$ separately to obtain two global chart lifts of $(\cC(H_Y,\mathfrak{b});J_Y)$.
To relate the two global chart lifts, we need a master global chart lift that is a stabilisation of both.
To construct the master global chart lift, we just need to stick together the groups $G$, perturbation terms $e$, framings $F$, and obstruction bundles, etc as in Proposition \ref{p:doubly-closed}.
One difference here is that the zeroes of the obstruction section can intersect the boundary strata of the thickening.
We can perform the smoothing of the master global chart lift compatibly so that the multi-valued perturbations from the two presentations stabilise (and extend) to two different collections of multi-valued perturbations on the master global chart lift.
Since the two collections of multi-valued perturbations are on the same smooth global chart lift (the master global chart lift), the resulting chain complexes are chain homotopic.
\end{proof}

\subsection{Hamiltonian invariance of the chain complex}\label{s:Haminv}

\emph{We explain how to use continuation bimodules to obtain invariance of the orbifold chain complex under changing the Hamiltonian.}

We wish to understand the dependence on $H$.  More precisely, we will associate to a continuation path a filtered chain homotopy equivalence.  This will arise from an underlying bimodule of the corresponding Hamiltonian Floer ordered marked flow categories.

Let $\mathfrak{b}$ be a bulk as before.
Fix a smooth interpolation $\{H_s\}_{s \in \mathbb{R}}$ from $H_0=H_s$ for $s \le 0$, to $H_1=H_s$ for $s \ge 1$ (through $\Gamma$-equivariant Hamiltonians on $X$)
and $\{J_s\}_{s \in \mathbb{R}}$ from $J_0=J_s$ for $s \le 0$, to $J_1=J_s$ for $s \ge 1$.
The Hamiltonian Floer bimodule $\cR$ from $(\cC(H_0/\Gamma,\mathfrak{b}),J_0/\Gamma)$ to $(\cC(H_1/\Gamma,\mathfrak{b}),J_1/\Gamma)$ for objects $p \in \cC(H_0/\Gamma,\mathfrak{b})$ and $q \in \cC(H_1/\Gamma,\mathfrak{b})$ is the Gromov compactified moduli space $\cR_{pq}$ of $\Gamma$-equivariant maps $u:\Sigma^{\dagger} \to X$ such that
\begin{enumerate}
    \item $\Sigma^{\dagger}$ is an admissible cover of an orbifold cylinder $\scrC^{\dagger}$
    \item $\scrC^{\dagger}$ is equipped with a distinguished point $r\in \mathbb{R} \simeq \mathbb{R} \times \{0\} $
    \item $v:=[u/\Gamma]:\scrC^{\dagger} \to Y$ satisfies the Floer continuation equation determined by $r$. That is
    \begin{align}
        (dv-X_{H_{s-r}/\Gamma}(v)dt)^{0,1}_{J_{s-r}/\Gamma}=0 \label{eq:continuation}
    \end{align}
    and exponential convergence $\lim_{s \to -\infty} v(s,t)=p(t)$ and $\lim_{s \to \infty} v(s,t)=q(t)$.
\end{enumerate}
It is standard that $\cR_{pq}$ has a natural stratification making $\cR$ an ordered marked flow bimodule from $(\cC(H_0/\Gamma,\mathfrak{b}),J_0/\Gamma)$ to $(\cC(H_1/\Gamma,\mathfrak{b}),J_1/\Gamma)$.
For example, for $p, p' \in \cC(H_0/\Gamma,\mathfrak{b})$ and $q,q' \in \cC(H_1/\Gamma,\mathfrak{b})$ the codimension one boundary strata are the embeddings 
\begin{equation} \label{eqn:continuation boundary strata}
\cM^{H_0}_{pp'} \times_{[pt/\Gamma_{p'}]} \cR^{H_t}_{p'q} \hookrightarrow \cR^{H_t}_{pq} \hookleftarrow \cR^{H_t}_{pq'} \times_{[pt/\Gamma_{q'}]} \cM^{H_1}_{q'q}.
\end{equation}
We want to build a homotopy coherent global chart lift of $\cR_{pq}$ so that we get a chain map $C^{\cC(H_1/\Gamma,\mathfrak{b})} \to C^{\cC(H_0/\Gamma,\mathfrak{b})}$.

 Note that our previously defined integralized action depends on a choice of $\epsilon_x$ and $N$.
 On the other hand, the Hermitian line bundle $L$ depends on a choice of $\Omega$, which in turn depends on a neighbourhood $U$ of the periodic orbits of $H$. 
By possibly replacing $N$ with a larger one, we assume that the $N$ we choose for defining the flow categories of $H_0/\Gamma$ and $H_1/\Gamma$ agree. It implies that the resulting $\cA_{\Pi,\mathbb{Z}}$ agree. 
Similarly, we replace $\Omega$ by another closed 2-form which vanishes in a sufficiently small neighbourhood of the periodic orbits of both $H_0$ and $H_1$ such that Lemma \ref{l:Omega} still hold. 
The resulting chain complexes do not change under this replacement by Proposition \ref{p:chaincomplex2indep}.
Recall that we take $\Omega$ cohomologous to $[\omega_X]$ which is assumed rational.

Define $\partial^{cyl} \scrF_{0,2+h}^{\dagger}([\bP(V)/\Gamma], V^*)$ to be the codimension $1$ locus where the universal domain curve has at least two cylindrical components. 
 Following \cite[Definition 6.4]{Abouzaid-Blumberg}, we introduce $\mathcal{D}\scrF_{0,2+h}^{\dagger}([\bP(V)/\Gamma], V^*)$, the conic degeneration with discriminant  
 $\partial^{cyl} \scrF_{0,2+h}^{\dagger}([\bP(V)/\Gamma], V^*)$ over $\scrF_{0,2+h}^{\dagger}([\bP(V)/\Gamma], V^*)$.
  The space $\mathcal{D}\scrF_{0,2+h}^{\dagger}([\bP(V)/\Gamma], V^*)$ is a manifold with boundary and corners such that
\begin{itemize}
\item it admits a projection map $\mathcal{D}\scrF_{0,2+h}^{\dagger}([\bP(V)/\Gamma], V^*) \to \scrF_{0,2+h}^{\dagger}([\bP(V)/\Gamma], V^*)$ such that the fibre over a point in a codimension $k$ stratum is a union of $k+1$ many closed intervals \cite[Lemma 6.2]{Abouzaid-Blumberg}.
\end{itemize}

In our case, we can give an explicit model of $\mathcal{D}\scrF_{0,2+h}^{\dagger}([\bP(V)/\Gamma], V^*)$ (cf. \cite[Lemma B.17]{Abouzaid-Blumberg}). The universal curve $\mathcal{C}^{orb}$ over a point $\phi \in \scrF_{0,2+h}^{\dagger}([\bP(V)/\Gamma], V^*)$ is equipped with a $\Gamma$ action. The coarse moduli space $\mathcal{C}^{orb}/\Gamma$ is a smooth manifold with boundary and corners such that the fibre over $\phi$ is the prestable cylinder $\scrC_{\phi}^{\dagger}=\Sigma_{\phi}^\dagger/\Gamma$ (cf. Definition \ref{d:orbicylinder}).
The asymptotic marker on each cylindrical component determines a geodesic as explained after Definition \ref{d:orbicylinder}.
The union of these geodesics defines a submanifold of  $\mathcal{C}^{orb}/\Gamma$ which is stratified homeomorphic to $\mathcal{D}\scrF_{0,2+h}^{\dagger}([\bP(V)/\Gamma], V^*)$ and compatible with the projection map to $\scrF_{0,2+h}^{\dagger}([\bP(V)/\Gamma], V^*)$.
We denote by $\mathcal{D}^{\circ}\scrF_{0,2+h}^{\dagger}([\bP(V)/\Gamma], V^*) \subset \mathcal{D}\scrF_{0,2+h}^{\dagger}([\bP(V)/\Gamma], V^*)$ the subspace where over each fibre, we take the interior of the intervals (in other words, we forget the endpoints of the geodesics in the explicit model).

The space of continuation solutions admits a global chart in which the thickening is a tuple $(\phi,u,e,F)$ as before.
The main difference is that $\phi$ is taken in $\mathcal{D}^{\circ}\scrF_{0,2+h}^{\dagger}([\bP(V)/\Gamma], V^*)$.
It means that the universal curve $\Sigma^{\dagger}_\phi$ over $\phi$ comes with a distinguished point $r_{\phi}$ on a geodesic on one of the cylindrical components of $\scrC^{\dagger}_{\phi}$.
We will put the continuation Floer data near $r_{\phi}$ (see Equation \ref{eq:continuation}) and $u:\Sigma^{\dagger}_\phi \to X$ has to satisfied a $e$-perturbed Floer continuation equation (on the distinguished cylindrical component, and usual Floer equations on the other cylindrical components and holomorphic equation on the spherical components), while $F:H^0(L_u) \to V^*$ is an equivariant isomorphism as before.  
 This time, since $H$ is $s$-dependent, on the distinguished cylindrical component we need to take $L_u:=(\omega_{log} \otimes u^*L_H^{\otimes 3})^{\otimes k}$ for $u^*L_H$ being the Hermitian line bundle with curvature $u^*\Omega+d(\tilde{H}_{s-r}(u) \pi_u^*dt)$ (cf. Equation \eqref{eq:cuv2form}) where $\tilde{H}_s$ is an approximation of $NH_s$ as in Equation \eqref{eq:tildeH} such that $\tilde{H}_s$ agrees with $\tilde{H}_0$ and $\tilde{H}_1$ respectively at the two ends.

Similar to Proposition \ref{p:singlechart}, we get a global chart of $\mathcal{R}_{pq}$ when $k, \mu$ are large.

To build such global charts inductively, we need the (integralized) action difference for continuation solutions to be positive (cf. Equation \ref{eq:monoaction}). 
To do this, we pick a positive integer $C>0$ and replace $H_1$ and $\tilde{H}_1$ by $H_1+C$ and $\tilde{H}_1+NC$, respecitvely.
 The resulting Hamiltonian orbits, Floer solutions, integralized action difference, etc are the same so we can ensure that the ordered marked flow categories $\cC(H_1/\Gamma)$ and $\cC((H_1+C)/\Gamma)$ are isomorphic.
Therefore, we can ensure that the chain complexes $CF(H_1/\Gamma, J_1/\Gamma, \mathfrak{b})$ and $CF((H_1+C)/\Gamma, J_1/\Gamma, \mathfrak{b})$ differ only by an overall action shift of $C$.
When $C$ is large enough such that $H_0 \le H_1+C$, we can choose a monotone $H_s$.
Then the (integralized) action difference for continuation solutions will be positive.

Given the above ingredients, it is straightforward to modify the construction of the global chart lift of the ordered marked flow category $(\cC(H_i/\Gamma,\mathfrak{b}),J_i)$ to a (homotopy coherent) global chart lift of the ordered marked Floer bimodule $\cR$ between $(\cC(H_0/\Gamma,\mathfrak{b}),J_i)$ and $(\cC((H_1+C)/\Gamma,\mathfrak{b}),J_i)$. Therefore, we obtain an action decreasing chain map
\[
CF((H_1+C)/\Gamma,\mathfrak{b}) \to CF(H_0/\Gamma,\mathfrak{b}).
\]
Since $CF((H_1+C)/\Gamma,\mathfrak{b})$ agrees with $CF(H_1/\Gamma,\mathfrak{b})$ up to an overall action shift, we also have a chain map from the latter to $CF(H_0/\Gamma,\mathfrak{b})$ which is not necessarily action decreasing.

Next, we explain why the continuation maps are well-behaved with respect to composition.
Let $H^{01}$ be a smooth monotone interpolation from $H_0$ to $H_1$ and $H^{12}$ be a smooth monotone interpolation from $H_1$ to $H_2$.
Let $J^{01}$ and $J^{12}$ be the respective families of almost complex structures.
Let $C$ be a large number, $\psi \in [C,\infty]$ be a gluing parameter and define the glued data $H^{02}_\psi:=H^{01} \#_\psi H^{12}$, $J^{02}_\psi:=J^{01} \#_\psi J^{12}$ with the convention that $H^{02}_\infty$ and $J^{02}_\infty$ correspond to the broken data (i.e. no gluing), and $C$ is large enough such that the glued data are well-defined.

The Hamiltonian Floer 2-simplex of bimodules $\cR^{012}$ for objects $p \in \cC(H_0/\Gamma,\mathfrak{b})$ and $q \in \cC(H_2/\Gamma,\mathfrak{b})$ is the Gromov compactified moduli space $\cR^{012}_{pq}$ of $\Gamma$-equivariant maps $u:\Sigma^{\dagger} \to X$ such that
\begin{enumerate}
    \item $\Sigma^{\dagger}$ is an admissible cover of an orbifold cylinder $\scrC^{\dagger}$
    \item $\scrC^{\dagger}$ is equipped with a distinguished point $r\in \mathbb{R} \simeq \mathbb{R} \times \{0\} $
    \item $v:=[u/\Gamma]:\scrC^{\dagger} \to Y$ satisfies the Floer continuation equation determined by $r$ for some $\psi \in [C,\infty)$. That is
    \begin{align}
        (dv-X_{(H^{02}_\psi)_{s-r}/\Gamma}(v)dt)^{0,1}_{(J^{02}_\psi)_{s-r}/\Gamma}=0 \label{eq:continuation}
    \end{align}
    and exponential convergence $\lim_{s \to -\infty} v(s,t)=p(t)$ and $\lim_{s \to \infty} v(s,t)=q(t)$.
\end{enumerate}

\begin{remark}
    When we compactify it to get $\cR^{012}$, $\psi=\infty$ is allowed.
\end{remark}

\begin{remark}
    Instead of using $r$ and $\psi$, it is equivalent to using two distinguished points $r_0,r_1$ in $\mathbb{R}$ such that $r_0+C \le r_1$ and impose the condition
    \begin{align}
        (dv-X_{(H^{02}_{r_1-r_0})_{s-r_0}/\Gamma}(v)dt)^{0,1}_{(J^{02}_{r_1-r_0})_{s-r_0}/\Gamma}=0 \label{eq:continuation2}
    \end{align}
    These data are related by $r=r_0$ and $\psi=r_1-r_0$. It is understood that $\psi=\infty$ if and only if the two distinguished points are on different cylindrical components.
\end{remark}

For the moduli of framed curves, we need to use a generalization of $\mathcal{D}\scrF_{0,2+h}^{\dagger}([\bP(V)/\Gamma], V^*)$ whose universal curve has two distinguished points (instead of one) on the union of cylindrical components to put the continuation equation \eqref{eq:continuation2} on. For how to add one more distinguished marked point, see \cite[Lemma B.20]{Abouzaid-Blumberg} for the smooth manifold case and there is no new feature in our case.

With the moduli of framed curves, we can define thickening as before. A global chart lift of $\cR^{012}$ can be constructed inductively by integralized action difference. The outcome of the discussion is that.

\begin{lemma}\label{l:homotopic}
    Given smooth monotone homotopies $H^{01}$ from $H_0$ to $H_1$, $H^{12}$ from $H_1$ to $H_2$ and $H^{02}$ from $H_0$ to $H_2$, the chain map
    $\tau_{H^{01}} \circ \tau_{H^{12}}$ is chain homotopic to $\tau_{H^{02}}$.
\end{lemma}

\begin{proof}
By building a global chart lift  of the Hamiltonian Floer 2-simplex, we know that 
$\tau_{H^{01}} \circ \tau_{H^{12}}$ is chain homotopic to $\tau_{H^{01}\#_CH^{12}}$.
Then by a homotopy of bimodule and its global chart lift, we know that it is in turn chain homotopic to  $\tau_{H^{02}}$.
\end{proof}

\begin{corol}\label{c:constantcon}
The continuation map $\tau$ defined by constant (i.e. $s$-independent) Floer data is chain homotopic to the identity. 
\end{corol}

\begin{proof}
The $s$-independent Floer solutions are regular.
Therefore, by an energy filtration argument, we know that $\tau$ is a chain homotopy equivalence.

On the other hand, we can construct a Hamiltonian Floer 2-simplex whose three facets all correspond to constant continuation Floer data. By Lemma \ref{l:homotopic}, it implies that $\tau \circ \tau$ is chain homotopic to $\tau$.
By composing a homotopy inverse of $\tau$,  it shows that $\tau$ is chain homotopic to the identity.
\end{proof}

\begin{lemma}\label{l:+C}
The continuation map from $H$ to $H+C$ is a chain homotopy equivalence.
\end{lemma}

\begin{proof}
The moduli space of Floer solutions are the same for the constant continuation from $H$ to $H$ and for the `adding constant' continuation from $H$ to $H+C$.
The only difference is that the integralized actions involved are different so the framing and thickening are different.
But by a doubly framed trick, it shows that the chain maps $CF(H/\Gamma,\mathfrak{b}) \to CF(H/\Gamma,\mathfrak{b})$ and $CF(H/\Gamma,\mathfrak{b}) \to CF((H+C)/\Gamma,\mathfrak{b})$ are chain homotopic (under the obvious identification between $CF(H/\Gamma,\mathfrak{b})$ and $CF((H+C)/\Gamma,\mathfrak{b})$).
\end{proof}

\begin{corol}
The chain maps associated to continuation maps are chain homotopy equivalences.
\end{corol}

\begin{proof}
Suppose we have pairs $(H_0,J_0)$ and $(H_1,J_1)$ and a  monotone continuation path $(H_s,J_s)$ from $(H_0,J_0)$ to $(H_1,J_1)$ as above. We pick a large integer $C$ and another monotone continuation path $(H'_s,J'_s)$ from $(H_1,J_0)$ to $(H_0+C,J_1)$.
We consider a one-parameter family $H_s^r$ of continuation maps interpolating between the concatenation $H_s \# H'_s$ and the path $H_0+C(s)$ for some monotone real-valued function $C(s)$ from $0$ to $C$. 
By Lemma \ref{l:homotopic} and \ref{l:+C}, we know that the composition of $\tau_{H_s'}$ and $\tau_{H_s}$ is a chain homotopy equivalence.
Similarly, the composition of $\tau_{H_s+C}$ and $\tau_{H_s'}$ is also a chain homotopy equivalence so  $\tau_{H_s}$ is a chain homotopy equivalence.
\end{proof}

\begin{prop}\label{p:orbFloer}
After forgetting the action filtration, $CF(H/\Gamma,\mathfrak{b};J_Y)$ is independent of  $H$ and $J_Y$ up to chain homotopy equivalences.
\end{prop}

\begin{proof}
Follows directly from the preceding results.
\end{proof}

\begin{proof}[Proof of Theorem \ref{t:main}(A)]

Proposition \ref{p:orbFloer} addressed the independence of $H$ and $J_Y$.

To see that it is an invariant of the orbifold and doesn't depend on the presentation, we can argue as in Proposition \ref{p:doubly-closed} and \ref{p:chaincomplex2indep}.
Given two presentations $Y=[X/\Gamma]=[X'/\Gamma']$, we can run Proposition \ref{p:globalliftexist} for $[X/\Gamma]$ and $[X'/\Gamma']$ separately to obtain two global chart lifts of $(\cC(H_Y,\mathfrak{b});J_Y)$.
Then we build a master global chart lift that is a stabilisation of both.
Since the two collections of multi-valued perturbations are on the same smooth global chart lift (the master global chart lift), the resulting chain complexes are chain homotopic.
\end{proof}

\subsection{Orbifold Morse complex}\label{s:additive}

\emph{We explain how to extract the orbifold Morse complex $C_{orb}(Y)$, defined in Section \ref{sec:Morse}, from a flow category perspective and construct pearly continuation maps between different Morse complexes.}


Let $f: X \to \bR$ be a $\Gamma$-invariant stable Morse function (see Lemma \ref{l:stableMorse}). Then $f$ induces a Morse function on $Y$ and indeed on every sector of the inertia orbifold $IY$; more explicitly, we obtain Morse functions (still denoted) $f: X^{g} \to \bR$ which are $C(g)$-invariant, for every $g \in \Gamma$.  
We regard $f$ as a time-independent Hamiltonian function.
We also choose a time-independent $J_Y$ and we have the associated Hamiltonian Floer flow category $\cC(f)$.
On the other hand, we can also associate a Morse flow category $\cC_{Morse}(f)$ to $f$ which corresponds to the chain complex $C_{orb}(Y) \otimes \Lambda_{\Pi}$.

The objects of $\cC_{Morse}(f)$ are those objects $c_{x}$ of $\cC(f)$ such that $x$ is a constant map.
Spelling it out in Morse theoretical terms, an object of $\cC_{Morse}(f)$ is an equivalence class of triples $(x,g,U)$ where $x$ is a critical point of $f$ in $X^g$, orientable in the sense that the isotropy group acts as the identity on the normal bundle of $X^g$ at $x$, $U$ is an element in $\Pi$, and two triples $(x,g,U)$, $(x',g',U')$ are isomorphic if $(x,g)=(hx', hg'h^{-1})$ for some $h \in \Gamma$ and $U=U'$.
Indeed, given $(x,g,U)$, the pair $(x,g)$ determine a $\Gamma$-cover $\hat{S} \to S^1$ and $\Gamma$-equivariant map $\hat{x}:\hat{S} \to X$ (which is a constant map on each connected component) covering $x$ (see \cite[Section 6.1.1]{MMRM} for more details).
We can choose an orbifold cap $\bar{c}_x$ of $x$ and $(x,g,U)$ determines the object $(x,\bar{c}_x,U-\omega_Y(\bar{c}_x)) \in \cC(f)$. 
In the opposite direction, we can get a triple $(x,g,U)$ from an object $(x,\bar{c}_x,U)$ in $\cC(f)$ when $x$ is a constant map.

The morphism space from $(x,g,U)$ to $(x',g',U')$  is given by the moduli space of gradient trajectory (with respect to the metric determined by $J_Y$) from $x$ to $hx'$ if $g=hg'h^{-1}$ for some $h \in \Gamma$ and $U=U'$, and empty otherwise.
As a result, $\cC_{Morse}(f)$ can be partitioned into a disjoint union of flow categories $\cC_{Morse}(f;[g],U)$, one for each $([g],U)$ for a conjugacy class $[g]$ of $\Gamma$ and $U \in \Pi$ (cf. Equation \eqref{eq:orbiMorsecomplex}).

Now we explain how the weight of the Morse differential (see Equation \eqref{eq:Morse_weight}) is related to the count of the orbifold Floer cylinders.
Given a gradient trajectory $\eta$ from $(x,g)$ to $(y,g)$, we can think of $\eta$ as the image of a Floer cylinder $v:\scrC \to Y$ which comes from a $\Gamma$-equivariant map $u:\Sigma \to X$.
The Floer cylinder $C$ has no orbifold point and $\Sigma$ is determined by $g$.
The automorphism group of $\eta$ and the automorphism group of $u$ are the same.
As a zero in the multi-valued perturbation, $u$ contributes $\frac{1}{|\Gamma_u|}$.
The weight we assign to $\eta$ is $\frac{|\Gamma_y|}{|\Gamma_\eta|}$, which agrees with the count of zeroes, $|\Gamma_y|$, times $\frac{1}{|\Gamma_u|}$.

It is a basic result of Wehrheim \cite{KW} that, after imposing a stronger form of the Morse-Smale condition in which the metric is also assumed standard and flat near the critical points, the compactified moduli spaces 
of $f$-gradient flow lines between critical points $x,y$ are the morphism spaces in a flow category enriched in smooth manifolds with corners. The proof directly shows that if the finite group $\Gamma$ acts and both $f$ and the metric is $\Gamma$-invariant then $\Gamma$ acts by diffeomorphisms of the flow category.  In such a situation, there is a coherent system of  global chart lifts for the category for the trivial group $G=\{e\}$, with no obstruction bundles, and with thickenings the moduli spaces themselves. In this case, the associated construction of a chain complex from the flow category is just the obvious thing: one discards all moduli spaces of dimension $>1$ and writes down the usual Morse complex.

Suppose that $f'$ is another $\Gamma$-invariant stable Morse function and $J_Y'$ is a $\Gamma$-equivariant compatible almost complex structure such that the induced metric together with $f'$ is Morse-Smale. For each sector $[X^g/C(g)]$, when the metrics are generic enough, we can define a chain map 
\[C_{Morse}([X^g/C(g)], f) \to C_{Morse}([X^g/C(g)], f')\]
by counting rigid intersection points between the ascending manifolds of $f$ and descending  manifolds of $f'$.
These chain maps are chain homotopy equivalences, so the direct sum of them defines a chain homotopy equivalence from $C_{orb}(Y,f)$ to $C_{orb}(Y,f')$.

After tensoring with $\Lambda_{\Pi}$, we can deform this chain homotopy equivalence by adding pseudo-holomorphic spheres, similar to how the cup product can be deformed to the quantum product (cf. Section \ref{s:orbQH}).
More precisely, let $\frak{b}=\sum_{(g) \in \|\Gamma\|^{\circ}} \frak{v}_{(g)} PD[X^g/C(g)]$ be a bulk class (see \eqref{eq:bulkb}) and $J_Y$ be a compatible almost complex structure.
Let $\bfg=(g_0,\bfg',g_{\infty}) \in \|\Gamma\|^{k+2}$ be such that $\bfg' \in \|\Gamma^\circ\|^{k}$.
Let $\cK_{0,\bfg}(Y;\beta)$ be the disjoint union of $\cK_{0,\bfm}(Y;\beta)$ over all $\bfm \in \mathcal H_{h+2}$ such that the associated sequence of conjugacy classes of the ordered marked points is $\bfg$ (cf. \eqref{eq:closedmodulidecom2}).
Denote the first marked point and the last marked point by $z_0$ and $z_{\infty}$, and regard them as the input and output respectively.
The remaining $h$ marked points are special marked points $\{z_1\ldots,z_h\}$ which will carry the bulk insertions $\frak{b}$ later.
Let $g_i \in \|\Gamma\|^{\circ}$ be the associated conjugacy class at $z_i$ for $i=0,\dots,h, \infty$.
We have evaluation maps $ev_{z_i}:\cK_{0,\bfg}(Y;\beta) \to [X^{g_i}/C(g_i)]$.
After equipping $\cK_{0,\bfg}(Y;\beta)$ with a global Kuranishi chart, we can take transverse fibre products with the descending manifold of a critical point $p'$ of $([X^{g_{\infty}}/C(g_{\infty})], f')$
and the ascending manifold of a critical point $p$ of $([X^{g_{0}}/C(g_{0})], f)$.
Then we can coherently apply a multi-valued perturbation to the sections. Denote the signed rigid count by $n(p,p';\bfg, \beta)$.

\begin{definition}[Pearly continuation map]
The pearly continuation map is defined to be the chain map (cf. Definition \ref{d:chaincomplex}, Proposition \ref{p:dsquare=0}, Remark \ref{r:deform} and Definition \ref{d:chaincomplex2})
\begin{align*}
C_{orb}(Y,f) \otimes \Lambda_{\Pi} \to& C_{orb}(Y,f')\otimes \Lambda_{\Pi} \\ 
p \mapsto & \sum_{p'}\sum_{k} \sum_{\bfg' \in (\|\Gamma^\circ\|)^k}\sum_\beta \frac{n(p,p'; (g(p), \bfg', g(p')), \beta)}{k!}T^{\omega(\beta)}\prod_{i=1}^k \frak{v}_{(g_i)} |\Gamma_{p'}|p'
\end{align*}
where $g(p)$ and $g(p')$ are the conjugacy classes of the sectors that $p$ and $p'$ lie in respectively.
\end{definition}

We leave it to readers to check that it is a chain map. 

\begin{lemma}\label{l:pearly continuation}
    The pearly continuation map is a chain homotopy equivalence.
    Moreover, when $f=f'$, it is chain homotopic to the identity.
\end{lemma}

\begin{proof}
This argument is very close to Corollary \ref{c:constantcon}, and again has two steps.

    First, the energy $0$ contribution is a chain homotopy equivalence so by a filtration argument, the pearly continuation map is also a chain homotopy equivalence.

    Next, the composition of two pearly continuation maps is chain homotopic to a pearly continuation map by gluing the intermediate gradient trajectory and then shrinking the length to $0$. Therefore, it is chain homotopic to the identity map.
\end{proof}

\subsection{PSS and SSP isomorphisms}\label{s:compatibleprod}

\emph{We construct the PSS and SSP maps, and show that they are chain homotopic inverse to each other.}

We want to describe the $\SSP$ morphism.  To do that, we will use a variant of the space $\mathcal{D}\scrF_{0,2+h}^\dagger([\bP(V)/\Gamma],V^*)$, which was used in Section \ref{s:Haminv} to define the continuation map.

Let $\scrF_{0,2+h}([\bP(V)/\Gamma],V^*)$ and $\partial^{cyl} \scrF_{0,2+h}([\bP(V)/\Gamma],V^*)$ be as before.
Recall that the two distinguished marked points are denoted by $z_-,z_+$ (see Definition \ref{d:orbicylinder}).
We fix an asymptotic marker at $z_+$ of the universal domain curve of $\scrF_{0,2+h}([\bP(V)/\Gamma],V^*)$ smoothly and consistently, and denote the resulting moduli space by $\scrF_{0,2+h}^+([\bP(V)/\Gamma],V^*)$ and similarly $\partial^{cyl} \scrF_{0,2+h}^+([\bP(V)/\Gamma],V^*)$ for the boundary coming from $\partial^{cyl} \scrF_{0,2+h}([\bP(V)/\Gamma],V^*)$.
We consider the conic degeneration $\mathcal{D}\scrF_{0,2+h}^+([\bP(V)/\Gamma],V^*)$ with discriminant 
$\partial^{cyl} \scrF_{0,2+h}^+([\bP(V)/\Gamma],V^*)$ over $\scrF_{0,2+h}^+([\bP(V)/\Gamma],V^*)$.
Since the universal domain curve of $\scrF_{0,2+h}^+([\bP(V)/\Gamma],V^*)$ has no asymptotic markers on the cylindrical components yet (other than the output component), the $\{t=0\}$ locus of a cylindrical component is ambiguous up to  the $S^1$-action fixing the two special nodes in general. 
Therefore, a point $\phi \in \mathcal{D}\scrF_{0,2+h}^+([\bP(V)/\Gamma],V^*)$ has the data of a point in $\scrF_{0,2+h}^+([\bP(V)/\Gamma],V^*)$ together with a distinguished $S^1$-orbit $r_{\phi}$ on the coarse moduli space of a cylindrical component of the domain curve $\scrC_{\phi}$.
Let $\mathcal{D}^{\circ}\scrF_{0,2+h}^+([\bP(V)/\Gamma],V^*)$ be the dense open subspace of $\mathcal{D}\scrF_{0,2+h}^+([\bP(V)/\Gamma],V^*)$ such that $r_{\phi}$ is not a node in $\scrC_{\phi}$.
We call the irreducible component of $\scrC_{\phi}$ containing $r_\phi$ the distinguished component and the chain of irreducible components from the distinguished component to the output (but not including the distinguished one) the positive cylindrical components. Other cylindrical components are negative.
Let $\partial^+\mathcal{D}^{\circ}\scrF_{0,2+h}^+([\bP(V)/\Gamma],V^*)$ be the locus of $\mathcal{D}^{\circ}\scrF_{0,2+h}^+([\bP(V)/\Gamma],V^*)$ where the domain curve has at least one positive cylindrical component.
Let $\scrF^{SSP}([\bP(V)/\Gamma],V^*)$ be the real blow-up of $\mathcal{D}^{\circ}\scrF_{0,2+h}^+([\bP(V)/\Gamma],V^*)$ along $\partial^+\mathcal{D}^{\circ}\scrF_{0,2+h}^+([\bP(V)/\Gamma],V^*)$.
A point in $\phi \in \scrF^{SSP}([\bP(V)/\Gamma],V^*)$ is represented by $(\scrC^\dagger_{\phi}, \Sigma^\dagger_{\phi},r_\phi)$ such that
\begin{enumerate}
    \item $\scrC^\dagger_{\phi}$ is a prestable orbifold sphere with $2+h$ marked points (two of them are special, which are incoming and outgoing respectively) together with an embedding $\phi:\scrC^\dagger_{\phi} \to [\bP(V)/\Gamma]$ in degree $d$;
    \item a distinguished $S^1$-orbit $r_{\phi}$ on one of the cylindrical components of $\scrC^\dagger_{\phi}$;
    \item a pair of asymptotic markers at each node adjacent to a positive cylindrical component of $\scrC^\dagger_{\phi}$, and an asymptotic marker at the outgoing marked point such that the two markers on the same cylindrical component point towards the same geodesic;
    \item an admissible cover $\Sigma^\dagger_{\phi}$ of $\scrC^\dagger_{\phi}$ (it has the data of asymptotic markers for each point lying above a special point in $\scrC^\dagger_{\phi}$ which carries asymptotic markers, and an equivariant base-point preserving isomorphism between the asymptotic markers on the two sides of a node).
\end{enumerate}

Given a $\Gamma$-invariant time-dependent Hamiltonian $H_t$ on $X$  and a cut-off function $\rho:\mathbb{R} \to [0,1]$ with $\rho(s) = 0$ for $s\leq 0$ and $\rho(s) = 1$ for $s\geq 1$, there is a Hamiltonian
\[
H^{\SSP} = (1-\rho(s))\cdot C + \rho(s)H_t(x)
\]
for any chosen constant $C$, and a corresponding Hamiltonian vector field $X_{H^{\PSS}/\Gamma} \otimes dt$ on the cylinder which vanishes for $s \leq 0$.  We will choose $C < \inf_{t,x}H_t(x)$. For $\phi \in \scrF^{PSS}([\bP(V)/\Gamma],V^*)$, we consider equivariant maps $u:\Sigma_{\phi}^\dagger \to X$ such that the quotient $v:\scrC_{\phi}^\dagger \to Y$ satisfies the Floer equations
\begin{equation}\label{eqn:PSS}
\left. \begin{aligned}
(dv)_{J/\Gamma}^{0,1} = 0 \\ (dv - X_{H^{\PSS}/\Gamma} \otimes dt)^{0,1}_{J/\Gamma} = 0 \\ (dv - X_{H/\Gamma} \otimes dt)^{0,1}_{J/\Gamma} = 0 \end{aligned}\right \} \ \mathrm{for} \ \begin{cases} v & \mathrm{spherical \, or \, negative \, cylindrical } \\ v & \mathrm{distinguished} \\ v & \mathrm{positive \, cylindrical }\end{cases}
\end{equation}
 
 Let $\cM^{SSP}(\bullet,c_x)$ denote the corresponding moduli space with  capped orbit asymptotic $c_x$ at $+\infty$ where the Hamiltonian term is non-trivial. 
 Let $(g_-)$ be the conjugacy class of the marked point $z_-$.
 There is an evaluation map
 \begin{equation} \label{eqn:PSS evaluation}
\ev_{z_-}: \cM^{PSS}(\bullet,c_x) \longrightarrow [X^{g_-}/C(g_-)]
\end{equation}
Now suppose $y \in \mathrm{Crit}(f)$ and let $\overline{W}^u(f,y) \subset [X^{g_-}/C(g_-)]$ denote the compactified descending manifold at $y$. Recall this has a stratification in which the codimension one boundary strata are products 
\[
\overline{W}^u(f,y) \supset \cM_{y,z}(f) \times \overline{W}^u(f,z)
\]
and deeper strata involve breaking off longer chains of gradient flow-line moduli spaces.  The usual analysis of breaking of SSP solutions says:

\begin{lemma}
For a capped orbit $c_x$ and Morse critical point $y$, 
the fibre product $\cB(y,c_x)$ of evaluation $\overline{W}^u(f,y) \to X$  and the evaluation map \eqref{eqn:PSS evaluation}  define a bimodule $\cR$ over the Morse flow category for $f$ and the Floer ordered marked flow category for $H$. 
\end{lemma}

The moduli space $\cM^{SSP}(\bullet,c_x)$ admits a thickening $\scrT^{SSP}$ constructed on the usual lines: we fix a finite-dimensional approximation scheme $\{V_\mu\}$ for the space $C^{\infty}_c(\scrC^\dagger \times TX)$
and consider tuples 
\[
(\phi,u,F,e) 
\]
where $ \phi \in \scrF^{SSP}([\bP(V)/\Gamma],V^*)$, $u: \scrC^{\dagger}|_{\phi} \to X$ is $\Gamma$-equivariant, represents a fixed class $\epsilon$, $F:H^0(L_u)^{\Gamma} \to V^*$ is an equivariant isomorphism; and satisfying the perturbed version of the equations \eqref{eqn:PSS} arising from the choice of perturbation $e$.   Such global charts are constructed inductively with respect to integralised action.  There is an evaluation map 
\begin{equation} \label{eqn:PSS evaluation on thickening}
\ev_{z_-}: \scrT^{SSP} \longrightarrow X
\end{equation}
which we can assume (by inductively replacing the global charts by equivalent charts incorporating stabilisations by the pullback of $TX$, cf. \cite{AMS1}) are submersive.  Taking the fibre product  of \eqref{eqn:PSS evaluation on thickening} with the compactified unstable manifolds $\overline{W}^u(f,y)$, we deduce:
\begin{corol}
The PSS bimodule $\cR$ admits a smooth and equivariant global chart lift, and hence defines a chain map
\[
CF(H/\Gamma;\mathfrak{b}) \to C^*_{orb}(Y;f).
\]
\end{corol}

This defines the SSP morphism.

The inverse $\PSS$ morphism is constructed analogously but with a distinguished component $\Sigma^{\PSS}$ and a real blow-up which serves to equip those cylindrical components between the one containing $z_-$ and $\Sigma^{\PSS}$ with lateral lines instead.  We omit the details, but compare to \cite{BX}.

\begin{prop}
 For any non-degenerate $H$, the maps $PSS$ and $SSP$ between $CF(H/\Gamma;\mathfrak{b})$ and $C^*_{orb}(Y;f)$ are chain homotopic inverses of each other. 
\end{prop}

\begin{proof}
To show that $SSP \circ PSS$ is chain homotopic to the identity, we consider a 2-simplex of bimodules whose three edges correspond to $PSS$, $SSP$ and the gluing between the output of $PSS$ and the input of $SSP$. 
The Hamiltonian data on the glued operation can be homotoped to $0$, which gives a corresponding homotopy of bimodules from the last edge to the pearly continuation map. By constructing a global chart lift for the 2-simplex and the homotopy of bimodules as before, we conclude that $SSP \circ PSS$ is chain homotopic to the pearly continuation map from $f$ to itself, which by Lemma \ref{l:pearly continuation}, is chain homotopic to the identity.

To show that $PSS \circ SSP$  is chain homotopic to the identity, we consider a 2-simplex of bimodules whose three edges correspond to $SSP$, $PSS$ and the gluing between the output of $SSP$ and the input of $PSS$. This last configuration includes a finite length gradient trajectory.
We can further homotope it to shrink the length of the gradient trajectory to zero, smooth the node with respect to the asymptotic marker and then increase the Hamiltonian term to the constant continuation map.
By constructing a global chart lift for them as before, we know that $PSS \circ SSP$ is homotopic to the constant continuation map, which by Corollary \ref{c:constantcon}, is chain homotopic to the identity.

\end{proof}

We can combine the previously defined chain maps for the different inertia strata, to deduce that  $\PSS_H:QH_{\orb}(Y) \to HF(H/\Gamma;\mathfrak{b})$ is an additive isomorphism, which is moreover compatible with continuation maps.

\subsection{Products on the Hamiltonian Floer complex}\label{s:product}

\emph{We discuss the orbifold Floer product, which entwines different sectors of the inertia orbifold.  Our applications in \cite{MSS2} rely on the triangle inequality for spectral invariants, hence require the product structure.} 

In parallel to Definition \ref{d:orbicylinder}, we give the following definition.
\begin{definition}\label{d:orbipants}
    A marked prestable orbi-pair of pants $\scrC^\dagger$ consists of the following data
    \begin{enumerate}
        \item a prestable genus $0$ orbifold curve $\scrC$ with $3+h$ ordered marked points such that three of them, denoted by $z_0,z_1,z_\infty$ are distinguished ($z_0,z_1$ are incoming and $z_{\infty}$ is outgoing).
        \item a marked point can be a smooth point or an orbifold point, all orbifold points are marked points, nodes between irreducible components are not marked points,
        \item there is a distinguished irreducible component $C_0$ of $\scrC$ such that the chain of irreducible components from 
        $C_0$ to $z_0$, from $C_0$ to $z_1$, and from $C_0$ to $z_{\infty}$ are pairwise disjoint (other than at $C_0$). These $3$ chains of irreducible components are cylindrical components. The nodes separating cylindrical components, separating a cylindrical component and the distinguished component, and the $3$ distinguished marked points are called special points.
        \item for every irreducible component, we have an asymptotic marker at every special node it carries.
    \end{enumerate}   
\end{definition}
We define an admissible cover of a prestable orbi-pair of pants in a similar way.

In parallel to the definition of morphism spaces of a Floer flow category (see Section \ref{s:orbifoldFloer}), we can define a bilinear map $\cR_{H_1,H_2}: \cC(H_1/\Gamma) \times \cC(H_2/\Gamma) \to \cC((H_1 \# H_2)/\Gamma)$ over the flow categories associated to Hamiltonians $H_1$, $H_2$ and $H_3:=H_1\# H_2$ (see \eqref{eq:comp} for the definition of $H_1\#H_2$) using admissible covers of  prestable orbi-pair of pants.
More precisely, we consider the moduli of finite energy equivariant maps
$u: \Sigma^\dagger \to X$ such that $\Sigma^\dagger$ is an admissible cover of a prestable orbi-pair of pants $\scrC^\dagger$ and the following equations are satisfied (let $v$ be the induced map on the $\Gamma$-quotients as before): 
\begin{equation} \label{eqn:Floer equation2}
\left. \begin{aligned}
(dv - Y_{\kappa})_{J/\Gamma}^{0,1} = 0 \\ (dv - X_{H_i/\Gamma} \otimes dt)^{0,1}_{J/\Gamma} = 0 \\ (dv)^{0,1}_{J/\Gamma} = 0 \end{aligned}\right \} \ \mathrm{for} \ \begin{cases} v & \mathrm{distinguished} \\ v & \mathrm{cylindrical} \\ v & \mathrm{spherical} \end{cases}
\end{equation}
together with the convergence at the punctures (with respect to the chosen cylindrical ends), where $Y_{\kappa}$ is a Hamiltonian vector field valued $1$-form interpolating the $X_{H_i/\Gamma} \otimes dt$.
In other words, the $\omega$-dual of $Y_{\kappa}$ is a Hamiltonian function valued $1$-form $K_\kappa$ interpolating 
$H_i/\Gamma \otimes dt$.
Since $H_3=H_1\# H_2$, we can choose $K_{\kappa}$ such that $dK_{\kappa}$ is a non-negative $2$-form, which is needed to ensure that we have \eqref{eq:action_triangle} when we consider moduli of Floer solutions.

To define a global lift of $\cR_{H_1,H_2}$, we need to use an analogue of $\scrF^\dagger_{0,2+h}([\bP(V)/\Gamma],V^*)$.
Let $\scrF_{0,3+h}([\bP(V)/\Gamma],V^*)$ denote the space of representable stable genus zero curves in $[\bP(V)/\Gamma]$ with $3+h$ marked points, satisfying the conditions \eqref{eq:condition}.
Clearly, $\scrF_{0,3+h}([\bP(V)/\Gamma],V^*)$ is the same as $\scrF_{0,(h+1)+2}([\bP(V)/\Gamma],V^*)$ as spaces and the only difference is that we view 3 of the marked points as distinguished (indeed as punctures).
Two of the 3 marked points are incoming, and one is outgoing.
We choose incoming cylindrical ends for the two incoming marked points and an outgoing cylindrical end for the outgoing marked point smoothly and consistently over $\phi \in \scrF_{0,3+h}([\bP(V)/\Gamma],V^*)$.

Let $\scrF^\dagger_{0,3+h}([\bP(V)/\Gamma],V^*)$ be the moduli space of equivariant maps $u:\Sigma^{\dagger} \to \bP(V)$ such that
$\Sigma^{\dagger}$ is an admissible cover over a prestable orbi-pair of pants $\scrC^\dagger$ and
the underlying map $\scrC \to  [\bP(V)/\Gamma]$ (after forgetting the asymptotic markers) defines an element in $\scrF_{0,3+h}([\bP(V)/\Gamma],V^*)$.

Let $\partial^{cyl}\scrF_{0,3+h}([\bP(V)/\Gamma],V^*)$ be the locus on the boundary of $\scrF_{0,3+h}([\bP(V)/\Gamma],V^*)$ such that the domain curve has at least one cylindrical component (on top of the distinguished component).
By Lemma \ref{lem:normal crossing}, $\partial^{cyl}\scrF_{0,3+h}([\bP(V)/\Gamma],V^*)$ is a normal crossing divisor.
We have an analogue of Lemma \ref{l:realblowup}:

\begin{lemma}\label{lem:domains for product}
The space $\scrF^\dagger_{0,3+h}([\bP(V)/\Gamma],V^*)$ is the real blow-up of $\partial^{cyl}\scrF_{0,3+h}([\bP(V)/\Gamma],V^*)$ and hence a smooth manifold with corners.
\end{lemma} 

\begin{proof}
Similar to the proof of Lemma \ref{l:realblowup}. Example \ref{e:realblowup} gives the universal local model describing the real blow-up. Indeed, each real blow-up corresponds to choosing the two asymptotic markers adjacent to the node and a base-point preserving equivariant isomorphism of the asymptotic markers at the corresponding nodes of the admissible cover.
\end{proof}

A global Kuranishi chart for  $\cR_{H_1,H_2}(x,y;z)$ can be defined as before, using domains $\phi \in \scrF^\dagger_{0,3+h}([\bP(V)/\Gamma],V^*)$ rather than in $\scrF^\dagger_{0,2+h}([\bP(V)/\Gamma],V^*)$.

To run the inductive procedure to build homotopy coherent global chart lifts, we need to use the argument parallel to the one in Section \ref{s:Haminv}.
First, we can assume that $\cA_{\Pi,\mathbb{Z}}$ is common for all $\cC(H_i/\Gamma)$ and $\Omega$ vanishes in a neighborhood of the orbits of all the $H_i$.
Then, we need to pick $\tilde{K}_{\kappa}$ approximating $NK_{\kappa}$ as in Equation \eqref{eq:tildeH} such that $\tilde{K}_{\kappa}$ agrees with $\tilde{H}_i dt$ respectively at the three ends and $d\tilde{K}_{\kappa}$ is non-negative (at the cost of choosing a different collections of $\epsilon_x$ for the $H_i$ and apply Proposition \ref{p:chaincomplex2indep}, we can arrange that $d\tilde{K}_{\kappa} \ge NdK_{\kappa}$).
On the distinguished component, we need to replace $u^*L_H$ by the Hermitian line bundle whose curvature is given by
$u^*\Omega+d(\tilde{K}_\kappa(u))$.
The assumption that $d\tilde{K}_\kappa$ is non-negative ensures that Equation \eqref{eq:action_triangle} holds.

We construct a system of homotopy coherent global chart lifts for the bilinear map moduli spaces, compatible with breaking and chosen global chart lifts for the flow categories $\cC(H_i/\Gamma)$, inductively over the integralized action as before.    
 In this case, there is no canonical linear ordering for the cylindrical components so we don't have a preferred linear order to take the fibre products of the moduli spaces and hence the global charts. However, different ways of taking the fibre products are canonically isomorphic (see Remark \ref{r:assocom})\footnote{When we have a linear ordering, we only need the associativity part of Remark \ref{r:assocom}. When we don't have a linear ordering, we need both statements in Remark \ref{r:assocom}.}.
When we construct the master global chart and various stabilisation maps as in Proposition \ref{p:associativity}, the ordering of how we take fibre products does not affect the stabilisation maps up to isomorphisms\footnote{In contrast, in the approach in \cite{BX}, different ways of taking the product  affect the group embedding of the global charts (cf. \cite[Definition 5.6(1), Equation (5.4)]{BX}).}, so Definition \ref{d:coherenthomotopy}(3) and \ref{d:global lift}(4) can be achieved.

Using Lemma \ref{lem:domains for product} and standard gluing, one sees that the thickening has a $C^1_{loc}$-structure over the space of domains, which means that the global chart for the product can be equipped with outer collars and smoothed by the Bai-Xu \cite{BX}, Rezchikov  \cite{Rez} machinery, again compatibly with smoothings of the charts for the underlying flow categories; in particular, we can construct transverse multi-valued perturbations over the moduli spaces of three-marked orbifold spheres.

The Floer product 
\begin{equation} \label{eqn:product}
CF(H_1/\Gamma;\frak{b}) \otimes CF(H_2/\Gamma;\frak{b}) \longrightarrow CF((H_1\# H_2)/\Gamma;\frak{b})
\end{equation}
is then obtained from Proposition \ref{p:product}.  Moreover, it is an action decreasing map.

\begin{remark}
    The Hamiltonian vector field-valued $1$-form $Y_{\kappa}$ in Equation \eqref{eqn:Floer equation2} is a choice involved in the construction. Indeed, $H_3$ doesn't have to be $H_1\#H_2$. For any non-degenerate $H_3$, we can pick a large positive integer $C$ together with a choice of $K_{\kappa}$ interpolating $H_1/\Gamma \otimes dt$, $H_2/\Gamma \otimes dt$ and $(H_3-C)/\Gamma \otimes dt$ such that $dK_{\kappa}$ is non-negative.
    The same construction will give an action decreasing Floer product map
    \[
    CF(H_1/\Gamma;\frak{b}) \otimes CF(H_2/\Gamma;\frak{b}) \longrightarrow CF((H_3-C)/\Gamma;\frak{b}).
    \]
    Two different Floer product maps coming from two different choices of interpolation $K_{\kappa}$ are homotopy equivalent. 
\end{remark}

A continuation map data $H_s$ from $H_0$ to $H_1$ induces a continuation map data $H_s \# H_2$ from $H_0 \#H_2$ to $H_1 \# H_2$, and there are associated bimodules $\cW_{0,1}$ and $\cW_{02,12}$ between the flow categories $\cC(H_0/\Gamma)$ and $\cC(H_1/\Gamma)$ respectively $\cC((H_0\#H_2)/\Gamma)$ and $\cC((H_1\# H_2)/\Gamma)$.  The following reflects the usual algebra of breaking for compatibiltiy of the Floer product with continuation and with iterated composition; we will not give a detailed argument. 

\begin{lemma}\label{l:prodcommute}
There is a homotopy commutative diagram 
\[
\xymatrix{
CF(H_1/\Gamma;\frak{b}) \otimes CF(H_2/\Gamma;\frak{b}) \ar[r] \ar[d] & CF(H_0/\Gamma;\frak{b}) \otimes CF(H_2/\Gamma;\frak{b}) \ar[d] \\ 
CF((H_1 \# H_2)/\Gamma;\frak{b}) \ar[r] & CF((H_0 \# H_2)/\Gamma;\frak{b})
}
\]
Moreover, the Floer product is associative up to homotopy, i.e. given three Hamiltonians $H_i$ the two possible ways to apply the Floer product from
\[
CF(H_0/\Gamma;\frak{b}) \otimes CF(H_1/\Gamma;\frak{b}) \otimes CF(H_2/\Gamma;\frak{b}) 
\]
agree up to homotopy. 
\end{lemma}

\begin{proof}

Given  global chart lifts, this follows from Proposition \ref{p:product_compose}, which provides a chain homotopy from  the composition of a continuation map and a product to a product.

To build a homotopy coherent  global lift, we use the space $\scrF_{0,4+h}^\dagger([\bP(V)/\Gamma],V^*)$ with four distinguished marked points ($3$ incoming and $1$ outgoing).
We consider the $3$ chains of irreducible components from an incoming distinguished marked point to the outgoing distinguished marked point.  
It will give the domain curve either one or two distinguished components (one `$4$-valent component' or two `$3$-valent components') and possibly some cylindrical components.
We can define special points similarly, as well as $\partial^{cyl}\scrF_{0,4+h}([\bP(V)/\Gamma],V^*)$.
The corresponding $\scrF^\dagger_{0,4+h}([\bP(V)/\Gamma],V^*)$ is the real blow-up of $\partial^{cyl}\scrF_{0,4+h}([\bP(V)/\Gamma],V^*)$ and hence a smooth manifold with corners and we run the same inductive procedure as before.
\end{proof}

There is a pearl-type description of the orbifold quantum product on $H^*_{orb}(Y)$, deforming the Morse product briefly described at the end of Section \ref{sec:Morse}. The product depends on the additional piece of data which is the bulk deformation class $\frak{b}$,
yielding a ring $H^*_{orb}(Y,\frak{b})$.  The product again appeals to global charts $\bT_{0,3+h}(Y,\beta)$ for the spaces $\cK_{0,3+h}(Y,\beta)$ of stable genus $0$ orbifold curves with $h+3$ possibly stacky marked points. Let $\scrT_{0,3+h}(\beta)$ denote the thickening of such a chart; by construction, this has evaluation maps at the 3 distinguished marked points to the inertia orbifold $IY$. By taking fibre products with compactified descending manifolds, one obtains a $\Gamma$-equivariant flow bimodule over the Morse flow category of $f$, and the orbifold Morse-Smale-Witten product is the induced product $m_{Morse}$ on $\oplus_{g\in \|\Gamma\|} C_{Morse}^*(X^g)^{C(g)}$; note that the product will now in general mix up the different sectors.

\begin{prop}\label{p:proPss}
$m_{floer} \circ (\PSS_H \otimes \PSS_K)=\PSS_{H\#K} \circ m_{Morse}$. \qed
\end{prop}

\begin{figure}
    \centering
    \includegraphics[width=1\linewidth]{PSS2.pdf}
    \caption{From left to  right\label{fig:3 families}}
\end{figure}

\begin{proof}[Sketch]
The argument does not involve new ideas relative to constructions described in  previous parts of this paper.  Three different one-parameter families of domains interpolate between the two sides of Proposition \ref{p:proPss}, as depicted schematically in Figure \ref{fig:3 families}. One builds global charts for each of these cobordisms as before (in appropriate cases taking fibre products with ascending or descending manifolds of Morse moduli spaces).  Families of transvsere multivalued sections then give equivariant maps on the corresponding chain complexes.
\end{proof}

\subsection{Spectral invariants and their properties}\label{s:spectralprop}

\emph{We define spectral invariants for orbifold Hamiltonian Floer cohomology.}

Recall the definitions of the action of a capped orbit, Equation \eqref{eqn:action of capping}, and of the action spectrum, Definition \ref{d:action}.  The value of the action of a capped orbit defines a filtration on the Hamiltonian Floer complex (Definition \ref{d:chaincomplex}(2)), which leads to the standard definition of a spectral invariant:

\begin{definition}
Let $x \in QH_{orb}^*(Y)$, $H_Y \in C^{\infty}(Y \times S^1)$ be a non-degenerate Hamiltonian and $\mathfrak{b}$ as above.
The bulk-deformed spectral invariant is
\[
c^{\mathfrak{b}}(x,H)= \inf \{a \in \mathbb{R}| PSS_H(x) \in \im (HF(H_Y;\frak{b})^{<a} \to HF(H_Y;\frak{b}))\}
\]
 where $HF(H_Y;\frak{b})^{<a}$ is the homology of the subcomplex of $CF(H_Y,\frak{b})$ generated by elements with action less than $a$.
\end{definition}

When $Y$ is a manifold, spectral invariants $c^{\mathfrak{b}}(x,-): C^{\infty}(Y \times S^1) \to \bR$ typically satisfy a raft of `standard' properties, including:

\begin{enumerate}
   \item (Symplectic invariance) $c^{\mathfrak{b}}(x,H\circ \psi) = c^{\mathfrak{b}}(x,H)$ for any $\psi \in \Symp(Y, \omega_Y)$;
    \item (Spectrality) for non-degenerate  $H$, $c^{\mathfrak{b}}(x,H)$ lies in the action spectrum $\Spec(Y,H_Y)$;
\item (Hofer Lipschitz) for any $H, H'$
$$  \int_{0}^1 \min  (H_t - H'_t) dt \leq c^{\mathfrak{b}}(x,H) - c^{\mathfrak{b}}(x,H') \leq  \int_{0}^1 \max\,  (H_t - H'_t) dt;$$
\item (Monotonicity) if $H_t \leq H'_t$ then $c^{\mathfrak{b}}(x,H) \leq c^{\mathfrak{b}}(x,H')$;

\item (Homotopy invariance) if $H, H'$ are mean-normalized and determine the same point of the universal cover $\widetilde{\Ham}(Y,\omega_Y)$, then $c^{\mathfrak{b}}(x,H) = c^{\mathfrak{b}}(x,H')$;
\item (Shift) $c^{\mathfrak{b}}(x,H + s(t)) = c^{\mathfrak{b}}(x,H) +  \int_0^1 s(t)\, dt$.
\item (Subadditivity) for any $x,x' \in QH_{orb}(Y,\mathfrak{b})$ and $H,H'$, $c^{\mathfrak{b}}(x \cdot x',H \# H') \leq c^{\mathfrak{b}}(x,H) + c^{\mathfrak{b}}(x',H')$
\end{enumerate}

The same collection of properties hold in our case, on a global quotient orbifold $Y$.  Given the theory already developed, the proofs follow exactly the same lines as in the usual case, so we explain this rather briefly.

\begin{theo}[rephrasing Theorem \ref{t:spectral}]\label{t:spectral_thm}

The bulk-deformed orbifold spectral invariant satisfies all the listed properties, in particular the Hofer Lipschitz property, so we can define $c^{\mathfrak{b}}(x,H_Y)$ for degenerate $H_Y$ as the limit $\lim_n c^{\mathfrak{b}}(x,K_n)$ for a sequence of non-degenerate $K_n$ such that $\lim_n \|K_n-H_Y\|= 0$.
\end{theo}

\begin{proof}
\begin{enumerate}
\item (Symplectic invariance) follows from the fact that $\psi \in \Symp(Y, \omega_Y)$ induces an isomorphism between the flow Floer category of $H_Y \circ \psi$ and that of $H_Y$ as well as isomorphisms of $PSS$ flow bimodules;
\item (Spectrality) for non-degenerate $H_Y$ follows from the definition of $\Spec(Y,H_Y)$ in Definition \ref{d:action}.
Since $\Spec(Y,H_Y)$ is discrete (see Remark \ref{r:GenSpec}), the spectrality for general $H_Y$ can be deduced from 
the spectrality for non-degenerate $H_Y$ by the standard argument.

\item (Hofer Lipschitz) follows from the fact that we can choose a path of Hamiltonians from $H_Y$ to $H'_Y$ such that for $p\in \cP^{H_Y}$ and  $p'\in \cP^{H'_Y}$, the continuation flow bimodule $\scrB_{pp'}$ is empty if $\cA(p)-\cA(p') < \int_{0}^1 \min\,  ((H_Y)_t - (H'_Y)_t) dt$, and similarly for the opposite direction
\item (Monotonicity) is a special case of the Hofer Lipschitz property.
\item (Homotopy invariance). It is classical that the action spectrum does not change under a homotopy of $(\phi_{H_Y}^t)_{t \in [0,1]}$ relative to $t=0,1$. Our action $\cA$ (see \eqref{d:action}) differs from the classical one only by contributions from the bulks, so our action spectrum is also unchanged under such a homotopy. The homotopy invariance follows from (Hofer Lipschitz) and the  action spectrum being nowhere dense.
\item (Shift) follows from the definition of $\cA$  and the fact that the flow category and PSS flow bimodule do not change.
\item (Subadditivity) follows from the compatibility of the product structure with the PSS map (Proposition \ref{p:proPss}).
\end{enumerate}
\end{proof}

\bibliographystyle{alpha}
\bibliography{biblio}

@preamble{"\def\cprime{$'$} "}

@book {CLS,
    AUTHOR = {Cox, David A. and Little, John B. and Schenck, Henry K.},
     TITLE = {Toric varieties},
    SERIES = {Graduate Studies in Mathematics},
    VOLUME = {124},
 PUBLISHER = {American Mathematical Society, Providence, RI},
      YEAR = {2011},
     PAGES = {xxiv+841},
      ISBN = {978-0-8218-4819-7},
   MRCLASS = {14M25 (05A15 05E45 52B12)},
  MRNUMBER = {2810322},
MRREVIEWER = {Ivan\ Arzhantsev},
       DOI = {10.1090/gsm/124},
       URL = {https://doi.org/10.1090/gsm/124},
}

@article {Behrend,
    AUTHOR = {Behrend, K.},
     TITLE = {Gromov-{W}itten invariants in algebraic geometry},
   JOURNAL = {Invent. Math.},
  FJOURNAL = {Inventiones Mathematicae},
    VOLUME = {127},
      YEAR = {1997},
    NUMBER = {3},
     PAGES = {601--617},
      ISSN = {0020-9910,1432-1297},
   MRCLASS = {14D20 (14C25 14D22)},
  MRNUMBER = {1431140},
MRREVIEWER = {Barbara\ Fantechi},
       DOI = {10.1007/s002220050132},
       URL = {https://doi.org/10.1007/s002220050132},
}

@article{Bergh-Rydh,
      title={Functorial destackification and weak factorization of orbifolds}, 
      author={Bergh, Daniel and Rydh, David},
      year={2019},
      journal={arXiv:1905.00872},
      primaryClass={math.AG},
      url={https://arxiv.org/abs/1905.00872}, 
}

@book {BeyondLie,
    AUTHOR = {Knapp, Anthony W.},
     TITLE = {Lie groups beyond an introduction},
    SERIES = {Progress in Mathematics},
    VOLUME = {140},
   EDITION = {Second},
 PUBLISHER = {Birkh\"auser Boston, Inc., Boston, MA},
      YEAR = {2002},
     PAGES = {xviii+812},
      ISBN = {0-8176-4259-5},
   MRCLASS = {22-01},
  MRNUMBER = {1920389},
}

@article {Lashof,
    AUTHOR = {Lashof, R.},
     TITLE = {Equivariant smoothing theory},
   JOURNAL = {Bull. Amer. Math. Soc.},
  FJOURNAL = {Bulletin of the American Mathematical Society},
    VOLUME = {84},
      YEAR = {1978},
    NUMBER = {1},
     PAGES = {1--6},
      ISSN = {0002-9904},
   MRCLASS = {57D10 (57E10)},
  MRNUMBER = {464252},
MRREVIEWER = {Oscar\ Burlet},
       DOI = {10.1090/S0002-9904-1978-14396-1},
       URL = {https://doi.org/10.1090/S0002-9904-1978-14396-1},
}

@incollection {Abramovich_etal,
    AUTHOR = {Abramovich, Dan and Chen, Qile and Gillam, Danny and Huang,
              Yuhao and Olsson, Martin and Satriano, Matthew and Sun,
              Shenghao},
     TITLE = {Logarithmic geometry and moduli},
 BOOKTITLE = {Handbook of moduli. {V}ol. {I}},
    SERIES = {Adv. Lect. Math. (ALM)},
    VOLUME = {24},
     PAGES = {1--61},
 PUBLISHER = {Int. Press, Somerville, MA},
      YEAR = {2013},
      ISBN = {978-1-57146-257-2},
   MRCLASS = {14D20 (14A20)},
  MRNUMBER = {3184161},
MRREVIEWER = {Howard\ M.\ Thompson},
}

@article {Olsson,
    AUTHOR = {Olsson, Martin C.},
     TITLE = {({L}og) twisted curves},
   JOURNAL = {Compos. Math.},
  FJOURNAL = {Compositio Mathematica},
    VOLUME = {143},
      YEAR = {2007},
    NUMBER = {2},
     PAGES = {476--494},
      ISSN = {0010-437X,1570-5846},
   MRCLASS = {14D22 (14A20 14N35)},
  MRNUMBER = {2309994},
MRREVIEWER = {Charles\ D.\ Cadman},
       DOI = {10.1112/S0010437X06002442},
       URL = {https://doi.org/10.1112/S0010437X06002442},
}

@article {FG03,
    AUTHOR = {Fantechi, Barbara and G\"ottsche, Lothar},
     TITLE = {Orbifold cohomology for global quotients},
   JOURNAL = {Duke Math. J.},
  FJOURNAL = {Duke Mathematical Journal},
    VOLUME = {117},
      YEAR = {2003},
    NUMBER = {2},
     PAGES = {197--227},
      ISSN = {0012-7094,1547-7398},
   MRCLASS = {14N35 (14A20 14E15 14J81)},
  MRNUMBER = {1971293},
MRREVIEWER = {Dan\ Abramovich},
       DOI = {10.1215/S0012-7094-03-11721-4},
       URL = {https://doi.org/10.1215/S0012-7094-03-11721-4},
}

@article{AMS1,
      title={Complex cobordism, {H}amiltonian loops and global {K}uranishi charts}, 
      author={Mohammed Abouzaid and Mark McLean and Ivan Smith},
      year={2021},
      journal={arXiv:2110.14320},
      primaryClass={math.SG},
      url={https://arxiv.org/abs/2110.14320}, 
}

@incollection {Moe02,
    AUTHOR = {Moerdijk, Ieke},
     TITLE = {Orbifolds as groupoids: an introduction},
 BOOKTITLE = {Orbifolds in mathematics and physics ({M}adison, {WI}, 2001)},
    SERIES = {Contemp. Math.},
    VOLUME = {310},
     PAGES = {205--222},
 PUBLISHER = {Amer. Math. Soc., Providence, RI},
      YEAR = {2002},
      ISBN = {0-8218-2990-4},
   MRCLASS = {22A22 (55N30 55P15 58H05)},
  MRNUMBER = {1950948},
MRREVIEWER = {Janez\ Mr\v cun},
       DOI = {10.1090/conm/310/05405},
       URL = {https://doi.org/10.1090/conm/310/05405},
}

@article {Floer,
    AUTHOR = {Floer, Andreas},
     TITLE = {Morse theory for {L}agrangian intersections},
   JOURNAL = {J. Differential Geom.},
  FJOURNAL = {Journal of Differential Geometry},
    VOLUME = {28},
      YEAR = {1988},
    NUMBER = {3},
     PAGES = {513--547},
      ISSN = {0022-040X,1945-743X},
   MRCLASS = {58F05 (35J65 58E05)},
  MRNUMBER = {965228},
MRREVIEWER = {Jean-Claude\ Sikorav},
       URL = {http://projecteuclid.org/euclid.jdg/1214442477},
}

@incollection {PSS,
    AUTHOR = {Piunikhin, S. and Salamon, D. and Schwarz, M.},
     TITLE = {Symplectic {F}loer-{D}onaldson theory and quantum cohomology},
 BOOKTITLE = {Contact and symplectic geometry ({C}ambridge, 1994)},
    SERIES = {Publ. Newton Inst.},
    VOLUME = {8},
     PAGES = {171--200},
 PUBLISHER = {Cambridge Univ. Press, Cambridge},
      YEAR = {1996},
      ISBN = {0-521-57086-7},
   MRCLASS = {57R57 (58E05 58F05)},
  MRNUMBER = {1432464},
MRREVIEWER = {Vsevolod\ V.\ Shevchishin},
}

@article{Abouzaid-Blumberg,
      title={Foundation of {F}loer homotopy theory {I}: {F}low categories}, 
      author={Mohammed Abouzaid and Andrew J. Blumberg},
      year={2024},
      journal={arXiv:2404.03193},
      primaryClass={math.SG},
      url={https://arxiv.org/abs/2404.03193}, 
}

@article{HenGep07,
      title={Homotopy {T}heory of {O}rbispaces}, 
      author={Andre Henriques and David Gepner},
      year={2007},
      journal={arXiv:0701916},
}

@article{MMRM,
      title={Floer homology for global quotient orbifolds}, 
      author={Miguel Miranda Ribeiro Moreira},
      year={2019},
      journal={master thesis, available on his personal website},
      url={https://math.mit.edu/~miguel73/master_thesis.pdf}, 
}

@article {Cho-Hong,
    AUTHOR = {Cho, Cheol-Hyun and Hong, Hansol},
     TITLE = {Orbifold {M}orse-{S}male-{W}itten complexes},
   JOURNAL = {Internat. J. Math.},
  FJOURNAL = {International Journal of Mathematics},
    VOLUME = {25},
      YEAR = {2014},
    NUMBER = {5},
     PAGES = {1450040, 35},
      ISSN = {0129-167X,1793-6519},
   MRCLASS = {57R18 (37D15)},
  MRNUMBER = {3215216},
MRREVIEWER = {Karl\ Heinz\ Dovermann},
       DOI = {10.1142/S0129167X14500402},
       URL = {https://doi.org/10.1142/S0129167X14500402},
}

@article{Bao-Lawson,
      title={Morse homology and equivariance}, 
      author={Erkao Bao and Tyler Lawson},
      year={2024},
      journal={arXiv:2409.04694},
      primaryClass={math.GT},
      url={https://arxiv.org/abs/2409.04694}, 
}

@incollection {ACV,
    AUTHOR = {Abramovich, Dan and Corti, Alessio and Vistoli, Angelo},
     TITLE = {Twisted bundles and admissible covers},
      NOTE = {Special issue in honor of Steven L. Kleiman},
   JOURNAL = {Comm. Algebra},
  FJOURNAL = {Communications in Algebra},
    VOLUME = {31},
      YEAR = {2003},
    NUMBER = {8},
     PAGES = {3547--3618},
      ISSN = {0092-7872},
   MRCLASS = {14H10 (14A20 14H30)},
  MRNUMBER = {2007376},
MRREVIEWER = {Andrew Kresch},
       DOI = {10.1081/AGB-120022434},
       URL = {https://doi.org/10.1081/AGB-120022434},
}

@article{AMS2,
      title={Gromov-{W}itten invariants in complex and {M}orava-local {$K$}-theories}, 
      author={Mohammed Abouzaid and Mark McLean and Ivan Smith},
      year={2024},
      journal={Geometric and {F}unctional {A}nalysis},
volume = {34},
page = {1647-1733},
      primaryClass={math.SG},
      url={https://arxiv.org/abs/2307.01883}, 
}

@article {McDuff06,
    AUTHOR = {McDuff, Dusa},
     TITLE = {Groupoids, branched manifolds and multisections},
   JOURNAL = {J. Symplectic Geom.},
  FJOURNAL = {The Journal of Symplectic Geometry},
    VOLUME = {4},
      YEAR = {2006},
    NUMBER = {3},
     PAGES = {259--315},
      ISSN = {1527-5256,1540-2347},
   MRCLASS = {53D40 (22A22)},
  MRNUMBER = {2314215},
MRREVIEWER = {Andr\'e\ G.\ Henriques},
       DOI = {10.4310/jsg.2006.v4.n3.a2},
       URL = {https://doi.org/10.4310/jsg.2006.v4.n3.a2},
}

@article{MSS2,
    author = {Mak, Cheuk Yu and Seyfaddini, Sobhan and Smith, Ivan},
title = {Orbifold {F}loer spectral invariants, symmetric product links and  {W}eyl laws},
journal = {to appear},
year = {2025},
}

@article{Porcelli-Smith,
      title={Bordism of flow modules and exact {L}agrangians}, 
      author={Noah Porcelli and Ivan Smith},
      year={2024},
      journal={arXiv:2401.11766},
      primaryClass={math.SG},
      url={https://arxiv.org/abs/2401.11766}, 
}

@article {CMS,
    AUTHOR = {Cieliebak, Kai and Mundet i Riera, Ignasi and Salamon, Dietmar
              A.},
     TITLE = {Equivariant moduli problems, branched manifolds, and the
              {E}uler class},
   JOURNAL = {Topology},
  FJOURNAL = {Topology. An International Journal of Mathematics},
    VOLUME = {42},
      YEAR = {2003},
    NUMBER = {3},
     PAGES = {641--700},
      ISSN = {0040-9383},
   MRCLASS = {53D45 (55N91 57R20 58B15)},
  MRNUMBER = {1953244},
MRREVIEWER = {David\ E.\ Hurtubise},
       DOI = {10.1016/S0040-9383(02)00022-8},
       URL = {https://doi.org/10.1016/S0040-9383(02)00022-8},
}

@article{BX,
      title={Arnold conjecture over integers}, 
      author={Shaoyun Bai and Guangbo Xu},
      year={2022},
      journal={arXiv:2209.08599},
      primaryClass={math.SG},
      url={https://arxiv.org/abs/2209.08599}, 
}

@article {Chen-Ruan,
    AUTHOR = {Chen, Weimin and Ruan, Yongbin},
     TITLE = {A new cohomology theory of orbifold},
   JOURNAL = {Comm. Math. Phys.},
  FJOURNAL = {Communications in Mathematical Physics},
    VOLUME = {248},
      YEAR = {2004},
    NUMBER = {1},
     PAGES = {1--31},
      ISSN = {0010-3616,1432-0916},
   MRCLASS = {57R19 (53D45)},
  MRNUMBER = {2104605},
MRREVIEWER = {Paolo\ Lisca},
       DOI = {10.1007/s00220-004-1089-4},
       URL = {https://doi.org/10.1007/s00220-004-1089-4},
}

@article {CGHMSS,
    AUTHOR = {Cristofaro-Gardiner, Daniel and Humili\`ere, Vincent and Mak,
              Cheuk Yu and Seyfaddini, Sobhan and Smith, Ivan},
     TITLE = {Quantitative {H}eegaard {F}loer cohomology and the {C}alabi
              invariant},
   JOURNAL = {Forum Math. Pi},
  FJOURNAL = {Forum of Mathematics. Pi},
    VOLUME = {10},
      YEAR = {2022},
     PAGES = {Paper No. e27, 59},
      ISSN = {2050-5086},
   MRCLASS = {53D40 (37K06)},
  MRNUMBER = {4524364},
MRREVIEWER = {Christopher\ T.\ Woodward},
       DOI = {10.1017/fmp.2022.18},
       URL = {https://doi.org/10.1017/fmp.2022.18},
}

@article{Rez,
      title={Integral {A}rnol'd Conjecture}, 
      author={Semon Rezchikov},
      year={2022},
      journal={arXiv:2209.11165},
      primaryClass={math.SG},
      url={https://arxiv.org/abs/2209.11165}, 
}

@article {RT,
    AUTHOR = {Ross, Julius and Thomas, Richard},
     TITLE = {Weighted projective embeddings, stability of orbifolds, and
              constant scalar curvature {K}\"ahler metrics},
   JOURNAL = {J. Differential Geom.},
  FJOURNAL = {Journal of Differential Geometry},
    VOLUME = {88},
      YEAR = {2011},
    NUMBER = {1},
     PAGES = {109--159},
      ISSN = {0022-040X,1945-743X},
   MRCLASS = {32Q26},
  MRNUMBER = {2819757},
       URL = {http://projecteuclid.org/euclid.jdg/1317758871},
}

@inproceedings {KW,
    AUTHOR = {Wehrheim, Katrin},
     TITLE = {Smooth structures on {M}orse trajectory spaces, featuring
              finite ends and associative gluing},
 BOOKTITLE = {Proceedings of the {F}reedman {F}est},
    SERIES = {Geom. Topol. Monogr.},
    VOLUME = {18},
     PAGES = {369--450},
 PUBLISHER = {Geom. Topol. Publ., Coventry},
      YEAR = {2012},
   MRCLASS = {57R55 (37D15)},
  MRNUMBER = {3084244},
MRREVIEWER = {Marco\ Mazzucchelli},
       DOI = {10.2140/gtm.2012.18.369},
       URL = {https://doi.org/10.2140/gtm.2012.18.369},
}

@article {AV,
    AUTHOR = {Abramovich, Dan and Vistoli, Angelo},
     TITLE = {Compactifying the space of stable maps},
   JOURNAL = {J. Amer. Math. Soc.},
  FJOURNAL = {Journal of the American Mathematical Society},
    VOLUME = {15},
      YEAR = {2002},
    NUMBER = {1},
     PAGES = {27--75},
      ISSN = {0894-0347},
   MRCLASS = {14H10 (14A20 14D20 14N35)},
  MRNUMBER = {1862797},
MRREVIEWER = {Tyler J. Jarvis},
       DOI = {10.1090/S0894-0347-01-00380-0},
       URL = {https://doi.org/10.1090/S0894-0347-01-00380-0},
}

@incollection {Chen-RuanGW,
    AUTHOR = {Chen, Weimin and Ruan, Yongbin},
     TITLE = {Orbifold {G}romov-{W}itten theory},
 BOOKTITLE = {Orbifolds in mathematics and physics ({M}adison, {WI}, 2001)},
    SERIES = {Contemp. Math.},
    VOLUME = {310},
     PAGES = {25--85},
 PUBLISHER = {Amer. Math. Soc., Providence, RI},
      YEAR = {2002},
      ISBN = {0-8218-2990-4},
   MRCLASS = {53D45 (14N35)},
  MRNUMBER = {1950941},
MRREVIEWER = {Ignasi\ Mundet-Riera},
       DOI = {10.1090/conm/310/05398},
       URL = {https://doi.org/10.1090/conm/310/05398},
}

@article {Was,
    AUTHOR = {Wasserman, Arthur G.},
     TITLE = {Equivariant differential topology},
   JOURNAL = {Topology},
  FJOURNAL = {Topology. An International Journal of Mathematics},
    VOLUME = {8},
      YEAR = {1969},
     PAGES = {127--150},
      ISSN = {0040-9383},
   MRCLASS = {57.10},
  MRNUMBER = {250324},
MRREVIEWER = {R.\ Z.\ Goldstein},
       DOI = {10.1016/0040-9383(69)90005-6},
       URL = {https://doi.org/10.1016/0040-9383(69)90005-6},
}

@article {Costello,
    AUTHOR = {Costello, Kevin},
     TITLE = {Higher genus {G}romov-{W}itten invariants as genus zero
              invariants of symmetric products},
   JOURNAL = {Ann. of Math. (2)},
  FJOURNAL = {Annals of Mathematics. Second Series},
    VOLUME = {164},
      YEAR = {2006},
    NUMBER = {2},
     PAGES = {561--601},
      ISSN = {0003-486X,1939-8980},
   MRCLASS = {14N35},
  MRNUMBER = {2247968},
MRREVIEWER = {Hsian-Hua\ Tseng},
       DOI = {10.4007/annals.2006.164.561},
       URL = {https://doi.org/10.4007/annals.2006.164.561},
}

@article {AGOT,
    AUTHOR = {Abramovich, Dan and Graber, Tom and Olsson, Martin and Tseng,
              Hsian-Hua},
     TITLE = {On the global quotient structure of the space of twisted
              stable maps to a quotient stack},
   JOURNAL = {J. Algebraic Geom.},
  FJOURNAL = {Journal of Algebraic Geometry},
    VOLUME = {16},
      YEAR = {2007},
    NUMBER = {4},
     PAGES = {731--751},
      ISSN = {1056-3911,1534-7486},
   MRCLASS = {14D22 (14A20 14C05 14C40)},
  MRNUMBER = {2357688},
MRREVIEWER = {Charles\ D.\ Cadman},
       DOI = {10.1090/S1056-3911-07-00443-2},
       URL = {https://doi.org/10.1090/S1056-3911-07-00443-2},
}

@article {Mat71,
    AUTHOR = {Matumoto, Takao},
     TITLE = {Equivariant {$K$}-theory and {F}redholm operators},
   JOURNAL = {J. Fac. Sci. Univ. Tokyo Sect. IA Math.},
  FJOURNAL = {Journal of the Faculty of Science. University of Tokyo.
              Section IA. Mathematics},
    VOLUME = {18},
      YEAR = {1971},
     PAGES = {109--125},
      ISSN = {0040-8980},
   MRCLASS = {55.30 (22.00)},
  MRNUMBER = {290354},
MRREVIEWER = {A.\ Verona},
}

@book {Atiyah,
    AUTHOR = {Atiyah, M. F.},
     TITLE = {{$K$}-theory},
      NOTE = {Lecture notes by D. W. Anderson},
 PUBLISHER = {W. A. Benjamin, Inc., New York-Amsterdam},
      YEAR = {1967},
     PAGES = {v+166+xlix},
   MRCLASS = {55.30 (57.00)},
  MRNUMBER = {224083},
}

@article {Kuiper,
    AUTHOR = {Kuiper, Nicolaas H.},
     TITLE = {The homotopy type of the unitary group of {H}ilbert space},
   JOURNAL = {Topology},
  FJOURNAL = {Topology. An International Journal of Mathematics},
    VOLUME = {3},
      YEAR = {1965},
     PAGES = {19--30},
      ISSN = {0040-9383},
   MRCLASS = {55.40},
  MRNUMBER = {179792},
MRREVIEWER = {R.\ S.\ Palais},
       DOI = {10.1016/0040-9383(65)90067-4},
       URL = {https://doi.org/10.1016/0040-9383(65)90067-4},
}

@article {Janich,
    AUTHOR = {J\"anich, Klaus},
     TITLE = {Vektorraumb\"undel und der {R}aum der {F}redholm-{O}peratoren},
   JOURNAL = {Math. Ann.},
  FJOURNAL = {Mathematische Annalen},
    VOLUME = {161},
      YEAR = {1965},
     PAGES = {129--142},
      ISSN = {0025-5831,1432-1807},
   MRCLASS = {57.50 (57.30)},
  MRNUMBER = {190946},
MRREVIEWER = {N.\ Kuiper},
       DOI = {10.1007/BF01360851},
       URL = {https://doi.org/10.1007/BF01360851},
}

@article {LS03,
    AUTHOR = {Lehn, Manfred and Sorger, Christoph},
     TITLE = {The cup product of {H}ilbert schemes for {$K3$} surfaces},
   JOURNAL = {Invent. Math.},
  FJOURNAL = {Inventiones Mathematicae},
    VOLUME = {152},
      YEAR = {2003},
    NUMBER = {2},
     PAGES = {305--329},
      ISSN = {0020-9910,1432-1297},
   MRCLASS = {14C05 (14J28)},
  MRNUMBER = {1974889},
MRREVIEWER = {Roy\ M.\ Skjelnes},
       DOI = {10.1007/s00222-002-0270-7},
       URL = {https://doi.org/10.1007/s00222-002-0270-7},
}

@article {Lerman,
    AUTHOR = {Lerman, Eugene},
     TITLE = {Orbifolds as stacks?},
   JOURNAL = {Enseign. Math. (2)},
  FJOURNAL = {L'Enseignement Math\'ematique. Revue Internationale. 2e
              S\'erie},
    VOLUME = {56},
      YEAR = {2010},
    NUMBER = {3-4},
     PAGES = {315--363},
      ISSN = {0013-8584},
   MRCLASS = {18D05 (22A22)},
  MRNUMBER = {2778793},
MRREVIEWER = {Chenchang\ Zhu},
       DOI = {10.4171/LEM/56-3-4},
       URL = {https://doi.org/10.4171/LEM/56-3-4},
}

@book {OrbifoldBook,
    AUTHOR = {Adem, Alejandro and Leida, Johann and Ruan, Yongbin},
     TITLE = {Orbifolds and stringy topology},
    SERIES = {Cambridge Tracts in Mathematics},
    VOLUME = {171},
 PUBLISHER = {Cambridge University Press, Cambridge},
      YEAR = {2007},
     PAGES = {xii+149},
      ISBN = {978-0-521-87004-7; 0-521-87004-6},
   MRCLASS = {57R19 (19L64 55N35)},
  MRNUMBER = {2359514},
MRREVIEWER = {Yunfeng\ Jiang},
       DOI = {10.1017/CBO9780511543081},
       URL = {https://doi.org/10.1017/CBO9780511543081},
}

@article {Cheong-CiocanFontanine-Kim,
    AUTHOR = {Cheong, Daewoong and Ciocan-Fontanine, Ionu\c t{} and Kim,
              Bumsig},
     TITLE = {Orbifold quasimap theory},
   JOURNAL = {Math. Ann.},
  FJOURNAL = {Mathematische Annalen},
    VOLUME = {363},
      YEAR = {2015},
    NUMBER = {3-4},
     PAGES = {777--816},
      ISSN = {0025-5831,1432-1807},
   MRCLASS = {14D20 (14D23 14N35)},
  MRNUMBER = {3412343},
MRREVIEWER = {Alfonso\ Zamora},
       DOI = {10.1007/s00208-015-1186-z},
       URL = {https://doi.org/10.1007/s00208-015-1186-z},
}

@article {Pardon22,
    AUTHOR = {Pardon, John},
     TITLE = {Enough vector bundles on orbispaces},
   JOURNAL = {Compos. Math.},
  FJOURNAL = {Compositio Mathematica},
    VOLUME = {158},
      YEAR = {2022},
    NUMBER = {11},
     PAGES = {2046--2081},
      ISSN = {0010-437X,1570-5846},
   MRCLASS = {57R18 (19L10 19L47 55N32 55P91 55R91)},
  MRNUMBER = {4515091},
MRREVIEWER = {Thomas\ O.\ Rot},
       DOI = {10.1112/s0010437x22007783},
       URL = {https://doi.org/10.1112/s0010437x22007783},
}

@article {Pardon,
    AUTHOR = {Pardon, John},
     TITLE = {An algebraic approach to virtual fundamental cycles on moduli
              spaces of pseudo-holomorphic curves},
   JOURNAL = {Geom. Topol.},
  FJOURNAL = {Geometry \& Topology},
    VOLUME = {20},
      YEAR = {2016},
    NUMBER = {2},
     PAGES = {779--1034},
      ISSN = {1465-3060,1364-0380},
   MRCLASS = {53D35 (37J10 53D37 53D40 53D42 53D45 54B40 57R17)},
  MRNUMBER = {3493097},
MRREVIEWER = {Sonja\ Hohloch},
       DOI = {10.2140/gt.2016.20.779},
       URL = {https://doi.org/10.2140/gt.2016.20.779},
}

@misc{GZ21,
	archiveprefix = {arXiv},
	author = {Gironella, Fabio and Zhou, Zhengyi},
	eprint = {2108.12247},
	primaryclass = {math.SG},
	title = {Exact orbifold fillings of contact manifolds},
	year = {2021}}

@article{calabi,
	author = {Calabi, Eugenio},
	journal = {Problems in analysis ({L}ectures at the {S}ympos. in honor of {S}alomon {B}ochner, {P}rinceton {U}niv., {P}rinceton, {N}.{J}., 1969)},
	mrclass = {58D05 (53C15 57D25)},
	mrnumber = {0350776},
	mrreviewer = {Alan Weinstein},
	pages = {1--26},
	publisher = {Princeton Univ. Press, Princeton, N.J.},
	title = {On the group of automorphisms of a symplectic manifold},
	year = {1970}}

@article{Hirschi-Hugtenburg,
    author = {Hirschi, Amanda and Hugtenburg, Kai} ,
    title = {An open-closed {D}eligne-{M}umford field theory associated to a {L}agrangian submanifold},
    journal = {arXiv:2501.04687},
    year = {2025}
}

@article {Cho-Poddar,
    AUTHOR = {Cho, Cheol-Hyun and Poddar, Mainak},
     TITLE = {Holomorphic orbi-discs and {L}agrangian {F}loer cohomology of
              symplectic toric orbifolds},
   JOURNAL = {J. Differential Geom.},
  FJOURNAL = {Journal of Differential Geometry},
    VOLUME = {98},
      YEAR = {2014},
    NUMBER = {1},
     PAGES = {21--116},
      ISSN = {0022-040X},
   MRCLASS = {53D40 (32Q65 57R18)},
  MRNUMBER = {3263515},
MRREVIEWER = {Saibal Ganguli},
       URL = {http://projecteuclid.org/euclid.jdg/1406137695},
}

{\small

\medskip
\noindent Cheuk Yu Mak\\
\noindent School of Mathematical and Physical Sciences, Hicks Building, University of Sheffield,  S10 2TN, UK\\
{\it e-mail:} c.mak@sheffield.ac.uk

\medskip
 \noindent Sobhan Seyfaddini\\
\noindent Department of Mathematics, ETH Zurich, 8092, Zürich,
Switzerland. \\
 {\it e-mail:}  sobhan.seyfaddini@math.ethz.ch.

 \medskip
 \noindent Ivan Smith\\
\noindent Centre for Mathematical Sciences, University of Cambridge, Wilberforce Road, CB3 0WB, U.K.\\
{\it e-mail:} is200@cam.ac.uk

}

\end{document}